%% file: BarronClass_rev.tex
\documentclass[11pt,a4paper]{article}

\usepackage{graphicx}        

\usepackage[a4paper, hmargin={2.7cm,2.7cm},vmargin={3.6cm,2.9cm}]{geometry}
\usepackage{amsmath, amssymb, amsfonts, amsbsy, latexsym, xcolor,dsfont}
\usepackage{amscd,amsxtra}
\usepackage{amsthm}

\usepackage{enumitem}

\usepackage[T1]{fontenc}
\usepackage[utf8]{inputenc}
\usepackage[english]{babel}
\usepackage[bf,compact,pagestyles,small]{titlesec}
\titlespacing*{\section}{0pt}{14pt}{4pt}
\titlespacing*{\subsection}{0pt}{8pt}{3pt}

\usepackage{etoolbox}
\makeatletter
\patchcmd{\ttlh@hang}{\parindent\z@}{\parindent\z@\leavevmode}{}{}
\patchcmd{\ttlh@hang}{\noindent}{}{}{}
\makeatother

\usepackage{mathdots}
\usepackage{tikz}
\usetikzlibrary{arrows,%
                petri,%
                topaths}%
\usepackage{caption}

\usepackage{hyperref}
\hypersetup{
    colorlinks,
    linkcolor={red!50!black},
    citecolor={blue!50!black},
    urlcolor={blue!80!black},
 pdfview={FitH},
 pdfstartview={FitH},
 bookmarksnumbered={true},
 pdfauthor = {Andrei Caragea, Philipp Petersen, and Felix Voigtlaender},
 pdftitle = {Neural network approximation and estimation of classifiers
with classification boundary in a Barron class},
 pdfsubject = {},
 pdfkeywords = {},
 pdfcreator = {LaTeX with hyperref package},
 pdfproducer = PDFlatex}
\usepackage[noabbrev]{cleveref}

\usepackage{url}

\theoremstyle{plain}
\newtheorem{theorem}{Theorem}[section]
\newtheorem{proposition}[theorem]{Proposition}
\newtheorem{lemma}[theorem]{Lemma}

\theoremstyle{definition}
\newtheorem{definition}[theorem]{Definition}

\theoremstyle{remark}
\newtheorem{remark}[theorem]{Remark}
\newtheorem*{example*}{Example}
\theoremstyle{remark}
\newtheorem*{rem*}{Remark}

\numberwithin{equation}{section}

\newcommand{\Barron}{\mathcal{B}}
\newcommand{\BarronApproximation}{\mathcal{BA}}
\newcommand{\BarronHorizon}{\mathcal{BH}}
\newcommand{\BarronBoundary}{\mathcal{BB}}
\newcommand{\HorizonFunctions}{\mathrm{HF}}

\newcommand{\CalA}{\mathcal{A}}
\newcommand{\calF}{\mathcal{F}}
\newcommand{\CalF}{\mathcal{F}}
\newcommand{\calO}{\mathcal{O}}
\newcommand{\calG}{\mathcal{G}}
\newcommand{\CalG}{\mathcal{G}}
\newcommand{\VC}{\operatorname{VC}}
\newcommand{\BV}{\operatorname{BV}}

\newcommand{\NNWeights}{\mathcal{NN}}
\newcommand{\sign}{\operatorname{sign}}
\newcommand{\hypothesis}{\mathcal{H}}

\DeclareMathOperator*{\argmin}{argmin}

\newcommand*{\CC}{\mathbb{C}}
\makeatletter
\newcommand{\LeftEqNo}{\let\veqno\@@leqno}

\DeclareFontFamily{U}{mathx}{\hyphenchar\font45}
\DeclareFontShape{U}{mathx}{m}{n}{
      <5> <6> <7> <8> <9> <10>
      <10.95> <12> <14.4> <17.28> <20.74> <24.88>
      mathx10
      }{}
\DeclareSymbolFont{mathx}{U}{mathx}{m}{n}
\DeclareFontSubstitution{U}{mathx}{m}{n}
\DeclareMathAccent{\widecheck}{0}{mathx}{"71}
\DeclareMathAccent{\wideparen}{0}{mathx}{"75}

\newcommand{\Indicator}{{\mathds{1}}}

\newcommand{\Fourier}{\mathcal{F}}

\newcommand{\CalP}{\mathcal{P}}

\newcommand{\CalO}{\mathcal{O}}

\newcommand{\Schwartz}{\mathcal{S}}

\newcommand{\EE}{\mathbb{E}}
\newcommand{\PP}{\mathbb{P}}

\newcommand{\vertiii}[1]{{\left\vert \kern-0.25ex
                            \left\vert \kern-0.25ex
                              \left\vert #1\right\vert\kern-0.25ex
                            \right\vert \kern-0.25ex
                          \right\vert}}

\newcommand{\FirstN}[1]{\underline{#1}}

\newcommand{\eps}{\ensuremath{\varepsilon}}

\newcommand*{\numbersys}[1]{\ensuremath{\mathbb{#1}}}
\newcommand*{\R}{\numbersys{R}}

\newcommand*{\Z}{\numbersys{Z}}

\newcommand*{\N}{\numbersys{N}}

\usepackage{xspace}

\makeatletter
\def\blfootnote{\xdef\@thefnmark{}\@footnotetext}
\def\subjclass{\xdef\@thefnmark{}\@footnotetext}
\long\def\symbolfootnote[#1]#2{\begingroup%
\def\thefootnote{\fnsymbol{footnote}}\footnote[#1]{#2}\endgroup}
\makeatother

\let\emptyset\varnothing

\binoppenalty=\maxdimen
\relpenalty=\maxdimen
\begin{document}

\begin{NoHyper}
\blfootnote{{\it Funding.}
AC thankfully acknowledges support by the German Research Foundation (DFG),
project number PF~450/11--1.
FV thankfully acknowledges support by the German Research Foundation (DFG)
in the context of the Emmy Noether junior research group VO 2594/1--1.
}
\end{NoHyper}

\title{Neural network approximation and estimation of classifiers
with classification boundary in a Barron class}
\author{Andrei Caragea%
\thanks{KU Eichstätt--Ingolstadt,
Mathematisch--Geographische Fakultät,
Ostenstraße~26,
Kollegiengebäude~I Bau~B,
85072~Eichstätt,
Germany},
Philipp Petersen%
\thanks{Faculty of Mathematics  and Research Platform Data Science @ Uni Vienna,
University of Vienna,
Kolingasse 14-16,
1090 Vienna,
Austria,
e-mail: \texttt{philipp.petersen@univie.ac.at}},
Felix Voigtlaender%
\thanks{Faculty of Mathematics,
University of Vienna,
Kolingasse 14-16,
1090 Vienna,
Austria,
e-mail: \texttt{felix.voigtlaender@univie.ac.at}}
}

\date{\today}

\maketitle

\begin{abstract}
  \input{Abstract.tex}
\end{abstract}


\section{Introduction}%
\label{sec:Introduction}

\input{1-Introduction.tex}


\section{Uniform approximation of Barron-type functions using ReLU networks}%
\label{sec:UniformBarronApproximation}

\input{2-BarronUniform.tex}


\section{Approximation of sets with Barron class boundary}
\label{sec:BarronClassBoundaryApproximation}

\input{3-Approximation.tex}


\section{Lower bounds for approximating sets with Barron class boundary}%
\label{sec:LowerBounds}

\input{4-LowerBounds.tex}


\section{Estimation bounds}%
\label{sec:Generalization}

\input{5-LearningBounds.tex}


\section{The case against general measures}%
\label{sec:GeneralMeasures}

\input{6-GeneralMeasures.tex}


\section{Three kinds of Barron spaces}%
\label{sec:BarronSpaces}

\input{7-BarronSpaces.tex}

\appendix


\section{A bound for empirical processes with finite pseudo-dimension}
\label{sec:AppendixPseudoDimension}

\input{A-EmpiricalProcessBound.tex}


\section{A technical bound involving the total variation}%
\label{sec:TotalVariationTechnical}

\input{B-TotalVariation.tex}

\bibliographystyle{abbrvurl}

{\footnotesize
\bibliography{references}
}

\end{document}

%% file: Abstract.tex
We prove bounds for the approximation and estimation of certain binary classification functions
using ReLU neural networks.
Our estimation bounds provide a priori performance guarantees for empirical risk minimization
using networks of a suitable size, depending on the number of training samples available.
The obtained approximation and estimation \emph{rates} are independent of the dimension of the input,
showing that the curse of dimensionality can be overcome in this setting;
in fact, the input dimension only enters in the form of a polynomial factor.
Regarding the regularity of the target classification function,
we assume the interfaces between the different classes to be locally of Barron-type.
We complement our results by studying the relations between
various Barron-type spaces that have been proposed in the literature.
These spaces differ substantially more from each other than the current literature suggests.

%% file: 1-Introduction.tex
This article concerns the approximation and statistical estimation of \emph{high-dimensional,
discontinuous functions by neural networks}.
More precisely, we study a certain class of target functions for classification problems,
such as those encountered when automatically labeling images.
For such problems, deep learning methods---based on the training of deep neural networks
with gradient-based methods---achieve state of the art performance
\cite{lecun2015deep, krizhevsky2012imagenet}.
The underlying functional relationship of such an (image) classification task
is typically extremely high-dimensional.
For example, the most widely used image data-bases used to benchmark classification algorithms are
MNIST \cite{lecun1998gradient} with $28 \times 28$ pixels per image, CIFAR-10/CIFAR-100
\cite{cifar10ref} with $32 \times 32$ pixels per image and ImageNet
\cite{deng2009imagenet, krizhevsky2012imagenet} which contains high-resolution images
that are typically down-sampled to $256 \times 256$ pixels.
Compared to practical applications, these benchmark datasets are relatively low-dimensional.
Yet, already for MNIST, the simplest of those databases, the input dimension
for the classification function is $d = 784$.

It is well known in classical approximation theory that high-dimensional approximation problems
typically suffer from the so-called \emph{curse of dimensionality}
\cite{bellman1952theory, novak2009approximation}.
This term describes the fact that the problems of approximation or estimation typically become
exponentially more complex for increasing input dimension.
Yet, given the overwhelming success of deep learning methods in practice,
high-dimensional input does not seem to be a prohibitive factor.

One of the first theoretical results in neural network approximation offering a partial
explanation for this ostensible clash of theory and practical observations was found in
\cite{BarronUniversalApproximation}.
There it was demonstrated that for a certain class of functions with variation bounded in a
suitable sense (these functions are, in particular, Lipschitz continuous), neural networks with one
hidden layer of $N$ neurons achieve an approximation accuracy of the order of $N^{-1/2}$ in the
$L^2(\mu)$-norm for a probability measure $\mu$ on a $d$-dimensional ball.
Notably, this approximation rate is \emph{independent} of the ambient dimension $d$.
Neural networks can thus overcome the curse of dimensionality for this class of functions.
This is particularly significant, since the considered class of functions
(nowadays so-called a \emph{Barron class}) is so large that every \emph{linear} method
of approximation for it is subject to the curse of dimensionality;
see \cite[Theorem~6]{BarronUniversalApproximation}.
The result of \cite{BarronUniversalApproximation} has since been extended and generalized in
various ways; we refer to Subsection~\ref{sec:OtherWork} for an overview.

In contrast to the (Lipschitz) continuous functions considered in
\cite{BarronUniversalApproximation}, our interest lies in the approximation
of \emph{discontinuous classification functions}.
Such functions are of the form $\sum_{k=1}^K q_k \mathds{1}_{\Omega_k}$,
where the sets $\Omega_k \subset \R^d$ are disjoint and describe $K + 1\in \N$ classes
(we also consider $(\bigcup_{k=1}^K\Omega_k)^c$ as a class).
Here $\mathds{1}_{\Omega_k}$ denotes the indicator function of $\Omega_k$;
that is, $\mathds{1}_{\Omega_k}(x) = 1$ if $x \in \Omega_k$ and $0$ otherwise.
Moreover, $(q_k)_{k=1}^K$ correspond to the labels of the classes
and could for example be unit vectors, as in $q_k = e_k \in \R^K$ for $k = 1, \dots K$,
in the case of one-hot-encoding or $(q_k)_{k=1}^K \subset \N$ for integer labels. 
These functions were discussed previously in \cite{petersen2018optimal}
and \cite{pmlr-v89-imaizumi19a, ImaizumiEstimatingFunctionsWithSingularitiesOnCurves},
where it was shown that the regularity of the boundary determines the approximation rate.
However, the results of
\cite{petersen2018optimal,pmlr-v89-imaizumi19a,ImaizumiEstimatingFunctionsWithSingularitiesOnCurves}
are based on classical notions of
smoothness regarding the boundary and suffer from the curse of dimensionality.
In this article, we assume the class interfaces to be locally of bounded variation
in the sense used in \cite{BarronUniversalApproximation}.
The following subsection gives an overview of our results and the employed proof methods.

\subsection{Our results}

We present upper and lower bounds for the approximation and estimation of classification functions
using deep neural networks with the ReLU activation function as hypothesis space.
The classification functions that we consider are of the form
$\sum_{k=1}^K q_k \mathds{1}_{\Omega_k}$, where each $\Omega_k \subset \R^d$ is an open set
such that $\partial \Omega_k$ is locally a $d-1$-dimensional Barron function.
In the sequel, we only consider the case of two complementary classes, that is, $K=1$;
the generalization to more summands is straightforward.

\paragraph{Measure of approximation accuracy:}
In contrast to ReLU neural networks, the indicator functions $\Indicator_\Omega$
are discontinuous.
Uniformly approximating $\mathds{1}_\Omega$ using ReLU neural networks is thus impossible.
Therefore, we measure the approximation error in terms of the measure of the set
on which the true function and the approximation differ;
since both functions are bounded in absolute value by $1$,
this easily implies corresponding error estimates in $L^p(\mu)$
for arbitrary exponents $p \in [1, \infty)$.
Here, we consider those measures $\mu$ that are \emph{tube compatible}
with an exponent $\alpha \in (0,1]$, meaning that the measure around any $\eps$ tube
of the graph of a function decays like $\eps^\alpha$ as $\eps \downarrow 0$.
This notion is broad enough to include a large class of product measures
on $\R^d$, as well as all measures of the form $d\mu = f d\nu$,
where $f$ is a bounded density and $\nu$ a tube compatible measure.
We also show in Section~\ref{sec:GeneralMeasures} that
for general (not tube compatible) measures, no nontrivial approximation rates can be derived.

\paragraph{Regularity assumptions on the class interfaces:}

Similar to the notion of $C^k$-domains or Lipschitz domains,
we assume the boundary $\partial \Omega \subset \R^d$
to be locally parametrized by Barron-regular functions.
Here, inspired by \cite{BarronUniversalApproximation}, we say that a function
$f : U \subset \R^{k} \to \R$ is of Barron-type, if it can be represented as
\begin{equation}
  f(x) = c + \int_{\R^k} \bigl(e^{i \langle x, \xi \rangle} - 1\bigr) F(\xi) \, d \xi
  \quad \text{for } x \in U,
  \qquad \text{where} \qquad
  \int_{\R^k}
    |\xi| \cdot |F(\xi)|
  \, d \xi
  < \infty .
  \label{eq:IntroBarronDefinition}
\end{equation}
For more formal discussion of our assumptions,
we refer to \Cref{def:BarronClassFunction,def:BarronClassifiers}.
We also remark that recently other notions of Barron-type functions
have been proposed in the literature; these are discussed briefly below
and in full detail in \Cref{sec:BarronSpaces}.

\paragraph{Upper bounds on the approximation rate:}

A simplified but honest version of our main approximation result reads as follows:

\begin{theorem}\label{thm:main_informal}
  Let $\mu$ be a finite measure, tube compatible with exponent $\alpha \in (0,1]$.
  Let $\Omega\subset \R^d$ be such that $\partial \Omega$ can locally be parametrized
  by functions of Barron-type.
  Then, for every $N \in \N$ the function $\mathds{1}_\Omega$ can be approximated
  using ReLU neural networks with three hidden layers and a total of $\mathcal{O}(d+N)$ neurons
  to accuracy $\mathcal{O}\bigl(d^{3/(2p)} N^{-\alpha/(2p)}\bigr)$ in the $L^p(\mu)$ norm.
  Moreover, the magnitude of the weights in the approximating neural networks can be chosen to be
  $\mathcal{O}(d+N^{1/2})$.
\end{theorem}

For example, if $\mu$ is the Lebesgue measure, then $\alpha = 1$.
We note that the accuracy of our approximation \emph{does} depend on the dimension,
but the dimension enters only as a multiplicative factor which is polynomial in $d$.

The proof of \Cref{thm:main_informal} is structured as follows:
\begin{enumerate}
  \item We use a classical result of Barron \cite{BarronNeuralNetApproximation} that yields uniform
        approximation of functions with a bounded Fourier moment.
        Because of a minor inaccuracy in the original result, we reprove this theorem in
        Proposition~\ref{prop:BarronFunctionApproximation}.

  \item Approximation of \emph{horizon functions}.
        We show that we can efficiently approximate \emph{horizon functions},
        meaning functions of the form $\mathds{1}_{x_1 \leq f(x_2, \dots, x_{d})}$
        where $f$ is a $d-1$ dimensional function of Barron-type.
        For the proof, we use
        a) that ReLU neural networks efficiently approximate the Heaviside function,
        b) the compositional structure of NNs, and
        c) the approximation result from Step~1.

  \item The classification function $\Indicator_{\Omega}$ is only \emph{locally} represented
        by horizon functions as in Step~2.
        Using a \emph{ReLU-based partition of unity}, we show that the
        result from Step~2 can be improved to an approximation
        of the full classification function $\Indicator_{\Omega}$.
\end{enumerate}


The details of the above argument are presented
in the proof of Theorem~\ref{thm:BarronBoundaryApproxGuarantee}.

\paragraph{Lower bounds on the approximation rate:}

We show that the established upper bounds on the approximation rates can, in general,
not be significantly improved.
More precisely, for the Lebesgue measure $d \mu = \Indicator_{[-1,1]^d} d \lambda$,
we show that for the set of classification functions considered above,
approximation with $L^1(\mu)$ error decaying asymptotically faster
than $N^{-1/2 - 1/(d-1)}$ for $N \to \infty$ is not possible.
For large input dimensions $d$, this almost matches the upper bound $N^{-1/2}$
from \Cref{thm:main_informal}.

We prove two forms of this result.
First, in Theorem~\ref{thm:QuantizedLowerBound}, we consider neural networks
for which the individual weights are suitably quantized and
grow at most polynomially with the total number $W \in \N$ of neural network parameters.
We show that no sequence of such neural networks achieves
an asymptotic approximation rate faster than $W^{-1/2 - 1/(d-1)}$.
This result follows by showing that efficient approximation of horizon functions
implies efficient approximation of the associated interface functions,
a technique previously applied in \cite{petersen2018optimal}.
Then, known entropy bounds for certain Besov spaces
contained in the classical Barron spaces can be used;
this is inspired by ideas from \cite{BarronUniversalApproximation}.

For ``quantized'' networks, we can allow arbitrary network architectures.
As our second result, we show in Theorem~\ref{thm:LowerBoundWithWeightBound} that
the assumption of weight quantization can be dropped, provided that
the depths of the approximating neural networks are assumed to be uniformly bounded.
It is still required, however, that the magnitude of the individual weights only grows
polynomially with the total network size.
The proof of this second result is based on a previously established ``quantization lemma'';
see \cite[Lemma~3.7]{boelcskei2019optimal} and \cite[Lemma~VI.8]{ElbraechterDNNApproximationTheory}.

\paragraph{Upper bounds on learning:}

Based on our approximation results, we study the problem of estimating classifier functions
of the form described above from a given set of training samples.
Precisely, we analyze the performance of the standard empirical risk minimization procedure,
where we use the 0-1 loss as the loss function and a suitable class of
ReLU neural networks as the hypothesis space.

To describe the result in more detail, let us denote by $\Phi_S$ the empirical
risk minimizer based on a training sample $S = \big( (X_1,Y_1),\dots,(X_m,Y_m) \big)$
with $(X_1,\dots,X_m) \overset{i.i.d.}{\sim} \PP$ and $Y_i = \Indicator_\Omega (X_i)$.
Assuming that the boundary $\partial \Omega$ is locally parametrized by functions
of Barron class and that $\PP$ is tube compatible with exponent $\alpha \in (0,1]$,
we derive bounds on the risk of $\Phi_S$, that is, on
$\PP \bigl(\Phi_S(X) \neq \Indicator_\Omega(X)\bigr)$ where $X \sim \PP$.


In Theorem~\ref{thm:LearningBound}, we show that, if
the hypothesis class is a certain set of ReLU neural networks with three hidden layers
and $N \sim (d m /\ln(dm))^{1/(1+\alpha)}$ neurons, then%
---with probability at least $1 - \delta$ with respect to the choice of the training sample $S$---%
the risk of any empirical risk minimizer $\Phi_S$ is at most
\[
  \CalO
  \bigg(
    d^{3/2} \cdot \Big( \frac{\ln(d m)}{d m} \Big)^{\alpha / (2 + 2 \alpha)}
    + \Big( \frac{\ln(1/\delta)}{m} \Big)^{1/2}
  \bigg) .
\]
In particular, if $\alpha = 1$, which is the case for the uniform probability measure,
then the risk is at most
$\CalO \bigl(d^{3/4} \ln(d m) \cdot m^{-1/4} + \sqrt{\ln(1/\delta) / m} \,\bigr)$.
This is similar to the estimation bounds established in \cite{barron1994approximation}
for Barron regular functions.

\paragraph{Different notions of Barron spaces:}

In this article we mainly use the Fourier-analytic notion of Barron-type functions
as introduced in \cite{BarronUniversalApproximation}; see \Cref{eq:IntroBarronDefinition}.
We will refer to this space as the \emph{classical Barron space},
or the \emph{Fourier-analytic Barron space}.
In recent years, other types of function spaces have been
studied under the name ``Barron-type spaces'' as well; see for instance
\cite{ma2020towards, wojtowytsch2020priori, ma2018priori, wojtowytsch2020banach}.
In contrast to the Fourier-analytic definition of \cite{BarronUniversalApproximation},
these more recent articles consider Barron spaces that essentially consist
of all ``infinitely wide'' neural networks with a certain control over the network parameters.
More formally, given an activation function $\phi$
(which is either the ReLU or a Heaviside function),
the elements of the associated Barron space are all functions that can be written as
\[
  f(x)
  = \int_{\R \times \R^d \times \R}
      a \cdot \phi\bigl(\langle w, x \rangle + c\bigr)
    \, d \mu(a,w,c)
\]
for a probability measure $\mu$ satisfying
\[
  \int_{\R \times \R^d \times \R}
    |a| \cdot \phi ( | w| + |c| )
  \, d \mu(a,w,c)
  < \infty \, .
\]
We will refer to these spaces as the \emph{infinite-width Barron spaces}.
We emphasize that in contrast to the Fourier-analytic Barron spaces,
these infinite-width Barron spaces do depend on the choice of the activation function $\phi$;
they thus do \emph{not} contain \emph{all conceivable} ``infinite-width'' networks.

The relationship between the infinite-width and Fourier-analytic Barron spaces
is not immediately obvious.
Already in \cite{BarronNeuralNetApproximation} it was shown that the Fourier-analytic Barron space
is contained in the infinite-width Barron space associated to the Heaviside function.
It is not clear, however, whether this also holds for the ReLU activation function.
In Section~\ref{sec:BarronSpaces}, we will review approaches in the literature
that address this embedding problem and prove that the classical Barron space is
\emph{not} contained in the infinite-width Barron space associated to the ReLU.
In fact, we show in Proposition~\ref{prop:FourierBarronNotInReLUBarron} the stronger result
that if we consider a generalized Fourier-analytic Barron space that consists
of all functions $f : \R^d \to \R$ such that their Fourier transform $\widehat{f}$ exists
and satisfies $\big\| \xi \mapsto (1 + |\xi|)^\alpha \widehat{f}(\xi) \big\|_{L^1(\R^d)} < \infty$,
then this space is contained in the infinite-width Barron space for the ReLU function
only if $\alpha \geq 2$.

\subsection{Previous work}
\label{sec:OtherWork}

In this section, we discuss previous research concerning the performance of neural networks for
approximating and estimating classification functions, as well as existing results concerning
dimension-independence in approximation and estimation problems.
We distinguish between results of Barron-type, i.e., approaches following the ideas of
\cite{BarronUniversalApproximation}, and other approaches.
We first discuss extensions of \cite{BarronUniversalApproximation}
for shallow neural networks (i.e., networks with one hidden layer).
Here, we in particular discuss the article \cite{wojtowytsch2020priori},
which is the only other work that we are aware of that studies \emph{classification problems}
(as opposed to regression problems) in the context of Barron-type functions.
Secondly, we discuss extensions to deep neural networks and then review
other related approaches not involving Barron-type spaces.
Finally, we explain how our work complements the existing literature.

\subsubsection{Previous work considering shallow neural networks}

In \cite{BarronUniversalApproximation}, it was shown that shallow neural networks can break
the curse of dimensionality for approximating functions $f$ that have one finite Fourier moment;
more precisely, one can achieve $\| f - \Phi_N \|_{L^2(\mu)} \lesssim N^{-1/2}$,
where $\Phi_N$ is a shallow neural network with $N$ neurons and $\mu$ is a
probability measure on a ball in $\R^d$.
The main insight in \cite{BarronUniversalApproximation} is that functions with
one finite Fourier moment belong to the closed convex hull of the set of
half planes; that is, they admit an integral representation
\begin{equation}
  f(x) = \int_{\R^d \times \R}
           \alpha(w,c) H(c + w^T x)
         \, d \nu(w, c)
  \label{eq:IntegralRepresentation}
\end{equation}
where $\nu$ is a probability measure satisfying
$\int_{\R^d \times \R} |\alpha(w,c)| \, d \nu(w,c) < \infty$
and $H = \Indicator_{[0,\infty)}$ is the Heaviside function.
The approximation rate of $N^{-1/2}$ is then a consequence of an approximate and probabilistic
version of Caratheodory's theorem; see for instance
\cite[Theorem~0.0.2]{VershyninHighDimensionalProbability}.
The paper \cite{BarronNeuralNetApproximation} generalized these results from approximation
in $L^2(\mu)$ to uniform approximation.
Furthermore, in \cite{barron1994approximation} these results are extended to obtain
estimation bounds for the class of functions with one bounded Fourier moment.
Essentially, using $n \sim N^2$ i.i.d.~samples, a neural network with $N$ neurons can be found
that approximates $f$ up to an $L^2$-error of the order of $N^{-1/2} \sim n^{-1/4}$.

Recently, several extensions of these original results by Barron to
different spaces have been proposed.
The Barron-type spaces introduced in
\cite{ma2018priori,wojtowytsch2020representation,ma2020towards,wojtowytsch2020banach},
are motivated by the integral representation \eqref{eq:IntegralRepresentation}.
Specifically, given an activation function $\phi: \R \to \R$ and an exponent $p \in [1, \infty]$,
the \emph{$p$-infinite-width Barron space} consists of all functions of the form
\begin{align}\label{eq:BarronFunctions}
  f(x)
  = \int_{\mathbb{S}^{d-1} \times [-1,1]}
      a(w,b) \, \phi(b + \langle x, w \rangle)
    \, d \pi(w,b),
\end{align}
for $x \in \R^d$, where $\pi$ is a probability measure on $\mathbb{S}^{d-1} \times [-1,1]$
and $a \in L^p(\pi)$.
It is shown that for certain values of $p$,
the functions in the \emph{$p$-infinite-width Barron space}
can be efficiently estimated and approximated by neural networks with activation function $\phi$,
without dependency on the dimension.
Let us add here that it was shown in \cite{parhi2021banach}
that functions of the form \eqref{eq:BarronFunctions} arise naturally
as the solutions of appropriately regularized learning problems.

We also mention the result \cite{MakovozUniformApproximationByNN},
in which a slightly improved approximation rate is obtained
for networks with the Heaviside function,
albeit under a slightly stronger assumption on the functions to be approximated.
Essentially, it is assumed in \cite{MakovozUniformApproximationByNN} that
\Cref{eq:IntegralRepresentation} holds with a \emph{bounded} function $a$
instead of an integrable one.
A further related result has been obtained in \cite{siegel2020approximation},
where the above results were extended to more general activation functions
and to approximation with respect to $L^2$-Sobolev norms.
In addition, lower bounds on the approximation of Barron functions
by shallow neural networks have recently been studied in \cite{siegel2021optimal}.

The work in the present paper complements these results by clarifying the relation
between the spaces of functions that can be represented
as in \Cref{eq:IntegralRepresentation,eq:BarronFunctions}
and those that have one finite Fourier moment, as considered in the original papers by Barron;
see \Cref{sec:BarronSpaces} for more details.

\subsubsection{Shallow neural networks for classification problems}

The article \cite{wojtowytsch2020priori} studies the problem of learning
a classification function associated to two disjoint classes $C_+, C_-$.
Instead of describing the accuracy of approximation and estimation with respect to
the typical square loss, the paper focuses on the hinge loss and certain cross-entropy type losses.
In this framework, a classification problem is considered solvable
with respect to a hypothesis class if there exist elements in that hypothesis class
that assume different signs on the two classes.
It is shown in \cite{wojtowytsch2020priori} that for general  $C_+, C_-$ such a problem is solvable
by Barron regular functions if and only if the sets $C_+, C_-$ have positive distance.
Since for these functions the approximation and estimation behavior
using shallow neural networks is well studied, as reviewed in the previous subsection,
this observation yields approximation and estimation bounds by shallow neural networks
for the classification problem.

In contrast to the setting considered in \cite{wojtowytsch2020priori},
in the present paper we analyze classification problems for which the different
classes are \emph{not} required to have a positive distance.
Instead, we impose a regularity condition on the class boundaries
and assume that the underlying probability measure is \emph{tube-compatible},
meaning that it should not be too strongly concentrated at the class boundary.


\subsubsection{Deep neural networks and the curse of dimensionality}

%

It is natural to wonder whether deeper networks can improve on shallow neural networks regarding
approximation and estimation problems.
The fundamental property enabling ``di\-men\-sion-free''
approximation by \emph{shallow} neural networks is that the
function to be approximated should belong to the closed convex hull of the set of simple neurons.
The corresponding property for deep networks has been identified to be a certain summability
property of the weights of approximating neural networks.
In \cite[Equation 1]{barron2018approximation}, this summability property is called
the \emph{variation} of the neural network.
To avoid ambiguities with the (total) variation of a measure or function,
we will refer to this notion as the \emph{weight variation}.
More precisely, the weight variation is the $\ell^1$ norm of the entries
of the product of the weight matrices of the neural networks%
\footnote{Here, it should be noted that all weights are assumed to be non-negative
in \cite{barron2018approximation}, which is accomplished there without loss of generality
by a slight modification of the activation function.}.
In \cite[Theorem~1]{barron2018approximation} it is shown that if arbitrarily large neural
networks are of bounded weight variation, then these neural networks can be well approximated
by smaller neural networks.
Here the size of the neural networks is measured via the encoding complexity of the weights.
Moreover, the reduction in size is independent of the dimension.
The weight variation also serves as a motivation for the so-called \emph{path norm}
that is fundamental to the definition of generalized Barron spaces
associated to compositional function representation in \cite{wojtowytsch2020banach}.
This path norm can be understood as the continuous counterpart of the weight variation.
Correspondingly, the elements of the generalized Barron spaces in \cite{wojtowytsch2020banach}
are those functions that can be obtained as limits of deep neural networks with bounded variation,
for increasing width. We also mention \cite{lee2017ability},
which studies approximation of functions that are compositions
of $n$ classical Barron functions and shows that these can then be
efficiently approximated by neural networks with $n+1$ layers.

In a somewhat similar vein, we show that if one is interested not in approximating
Barron-regular functions themselves, but rather classification functions for which the
class boundaries are Barron-regular, then this can be done efficiently with
(somewhat) deep ReLU neural networks, namely using networks with $3$ hidden layers;
see \Cref{thm:BarronBoundaryApproxGuarantee}.

\subsubsection{Non-Barron-type results on curse of dimensionality}

Functions of Barron-type are not the only functions that can be approximated by deep neural
networks without the curse of dimensionality.
Other function classes that allow for approximation with only minor
(in particular sub-exponential) dependencies on the dimension include the following:
functions that have a graph-like structure and are compositions of low dimensional functions,
\cite{poggio2017and}, \mbox{\cite[Section~5]{petersen2018optimal}},
\cite{shaham2018provable, cloninger2020relu, nakada2020adaptive};
bandlimited functions \cite{montanelli2019deep};
and also solutions of some classes of high-dimensional PDEs
\cite{han2018solving,hutzenthaler2018overcoming,jentzen2018proof,beck2020overcoming,
berner2018analysis,laakmann2020efficient,grune2020overcoming,elbrachter2018dnn,schwab2019deep}
and SDEs \cite{becker2019solving,reisinger2019rectified},
under the assumption that the right-hand side of the equation is itself well-approximated
(i.e., without suffering from the curse of dimensionality) by neural networks.

In the present paper, we show that the classification functions with Barron regular
decision boundaries also belong to this list of well-approximable functions.

\subsubsection{Deep neural networks for classification problems}

The approximation and estimation of classification functions of the form
$\sum_{k=1}^K f_k \mathds{1}_{\Omega_k}$, where each $\Omega_k \subset \R^d$ is an open set
such that $\partial \Omega_k$ is piecewise smooth and $f_k: \R^d \to \R$ are smooth, is studied in
\cite{pmlr-v89-imaizumi19a,ImaizumiEstimatingFunctionsWithSingularitiesOnCurves,petersen2018optimal}.
In these works, it is shown that the achievable approximation and estimation rates
are primarily determined by the smoothness of the boundaries $\partial \Omega_k$, in the sense that,
given sufficient regularity of the $f_k$, smoother class boundaries yield better approximation
and estimation rates.
The general strategy of the approximation theoretical aspects of these works
is closely related to the approach taken in this article.
Indeed, the approximation of classification functions is reduced to that of horizon functions
$\mathds{1}_{x_1 \leq f(x_2, \dots, x_{d})}$ where $f$ is a $d-1$-dimensional smooth function.
In addition, the articles
\cite{pmlr-v89-imaizumi19a,ImaizumiEstimatingFunctionsWithSingularitiesOnCurves}
establish estimation bounds by invoking classical bounds on the covering numbers
of the involved neural network spaces to bound the generalization error
of empirical risk minimization.

\subsubsection{Delineation of our work}

In the present article, we discuss a concrete set of practically relevant functions,
namely those arising in classification tasks where the interfaces between classes
are sufficiently regular, which formally means that they are locally described
by Barron-type functions.
As indicated earlier, these results are based on a combination of two ideas:
First, a classical result of Barron showing uniform and dimension-independent
approximation of Barron-type functions \cite{BarronNeuralNetApproximation}
and, second, a strategy to emulate functions
with regular jump curves by neural networks, originally introduced in \cite{petersen2018optimal}.

The results are neither a direct consequence of the study of (generalized) Barron spaces nor can
they be derived directly from the results of \cite{petersen2018optimal}.
Indeed, the functions that we discuss (classification functions with Barron-regular boundary) do
not have a representation by neural networks with bounded weights
or bounded variation of the weights.
In fact, it can be shown (see \cite[Theorem~2.7]{wojtowytsch2020banach})
that functions in the (generalized) Barron spaces are always Lipschitz continuous,
which is not satisfied for the classification functions that we consider.
The key difference between our approach and alternative studies of Barron spaces is that
in those works the boundedness of the (sum of the) network weights or a related
property such as a bounded weight variation plays a central role.
In contrast, we allow a moderate weight growth that is essentially inversely proportional to
the approximation error.
Besides, in contrast to \cite{wojtowytsch2020priori} we study classification problems
for which the different classes do \emph{not} have positive distance to each other.
Furthermore, the required regularity of the class boundaries for our results is explicitly stated,
e.g. in terms of a finite Fourier moment;
this is in contrast to the more implicit integral representation property
required for the infinite-width Barron spaces considered in \cite{wojtowytsch2020priori}.

Finally, in contrast to
\cite{pmlr-v89-imaizumi19a,ImaizumiEstimatingFunctionsWithSingularitiesOnCurves,petersen2018optimal},
the results in the present paper do not suffer from the curse of dimensionality.

\subsection{Structure of the paper}%
\label{sub:Structure}

After introducing general and neural network related notation in
Subsections~\ref{sub:Notation} and \ref{sub:NNNotation},
we start in \Cref{sec:UniformBarronApproximation} by formally defining
the Fourier-analytic Barron class, and proving that such functions can be uniformly
approximated with error $\CalO(N^{-1/2})$ using shallow ReLU networks with $\CalO(N)$
neurons and controlled weights.
We reprove this result since the argument in \cite{BarronNeuralNetApproximation}
for handling general sigmoidal activation functions contains a technical inaccuracy.

In \Cref{sec:BarronClassBoundaryApproximation}, we give the precise definition
of sets with boundary in the Barron class, and we show that indicator functions
of such sets can be well approximated by ReLU neural networks.
The complementing lower bounds and estimation bounds are derived in
\Cref{sec:LowerBounds,sec:Generalization}.
For the approximation and estimation results, we always assume that the measure
under consideration is tube compatible; \Cref{sec:GeneralMeasures} shows that this
is unavoidable.
Finally, in \Cref{sec:BarronSpaces}, we discuss the relation between the Fourier-analytic
Barron space that we consider and the alternative Barron spaces considered in the literature.

Several mainly technical results are deferred to the appendices.

\subsection{General notation}%
\label{sub:Notation}

We will use the following notation:
For $n \in \N_0 = \{ 0,1,2,3,\dots \}$, we write $\FirstN{n} := \{ 1,2,\dots,n \}$;
in particular, $\FirstN{0} = \emptyset$.
For an arbitrary set $M$, we write $|M| = \# M \in \N_0 \cup \{ \infty \}$
for the number of elements of $M$.

Given $a \in \R^d$, we denote the entries of $a$ by $a_1,\dots,a_d \in \R$.
For $a,b \in \R^d$ we write $a \leq b$ if and only if $a_i \leq b_i$ for all $i \in \FirstN{d}$.
In this case, we define $[a,b] := \prod_{i=1}^d [a_i, b_i]$.
For $x = (x_1,\dots,x_d) \in \R^d$ with $d > 1$ and $i \in \FirstN{d}$,
we set $x^{(i)} := (x_1, \dots, x_{i-1}, x_{i+1}, \dots, x_d) \in \R^{d-1}$.

The standard scalar product of $x,y \in \R^d$ will be denoted by
$\langle x,y \rangle = \sum_{i=1}^d x_i \, y_i$,
and the Euclidean norm of $x$ is written as $|x| := \sqrt{\langle x,x \rangle}$.
For a continuous function $f$ defined on a set $Q \subset \R^d$,
we define $\|f\|_{\sup} := \sup_{x\in Q} |f(x)|$.
For a set $X$ and two functions $f,g\colon X \to \R^+$,
we write $f(x) \lesssim g(x)$ if $f(x) \leq C g(x)$ for a constant $C > 0$ and all $x \in X$.
This constant is referred to as the implicit constant of the estimate.

Finally, given a class $\CalF$ of $\{ 0,1 \}$-valued (or $\{ \pm 1 \}$-valued) functions,
we denote the VC-dimension of $\CalF$ by $\VC(\CalF) \in \N_0 \cup \{ \infty \}$.
We refer to \cite[Chapter~6]{ShalevShwartzUnderstandingMachineLearning}
for the definition of the VC dimension.

\subsection{Neural network notation}
\label{sub:NNNotation}

In this subsection, we briefly introduce our notation regarding neural networks.
To avoid ambiguities, we define neural networks in a way that allows a precise counting of the
number of neurons and layers.
This is done by differentiating between a neural network as a set of weights and the associated
realization which represents the function that is described through these weights.

\begin{definition}
  Let $d,L\in\N$.
  A \emph{neural network (NN) $\Phi$ with input dimension $d$ and $L$ layers}
  is a sequence of matrix-vector tuples
  \[
    \Phi = \bigl( (A_1,b_1), (A_2,b_2), \dots, (A_L,b_L) \bigr),
  \]
  where, for $N_0 = d$ and certain $N_1, \ldots, N_L \in \N$,
  each $A_\ell$ is an $N_\ell \times N_{\ell-1}$ matrix, and $b_\ell \in \R^{N_\ell}$.

  For a NN $\Phi$ and an activation function $\phi: \R \to \R$,
  we define the associated \emph{realization of the NN $\Phi$} as
  \[
    R_\phi\Phi : \quad
    \R^d \to \R^{N_L} , \quad
    x \mapsto x_L = R_\phi \Phi(x),
  \]
  where the output $x_L \in \R^{N_L}$ results from the scheme
  \begin{align*}
    x_0      &:= x \in \R^d = \R^{N_0}, \\
    x_{\ell} &:= \phi\left(A_{\ell} \, x_{\ell-1} + b_\ell\right) \in \R^{N_\ell}
    \quad \text{ for } \ell = 1, \dots, L-1,\\
    x_L      &:= A_{L} \, x_{L-1} + b_{L} \in \R^{N_L}.
  \end{align*}
  Here $\phi$ is understood to act component-wise.
  We call $N(\Phi):=d+\sum_{j=1}^L N_j$ the \emph{number of neurons} of the NN $\Phi$,
  $L=L(\Phi)$ the \emph{number of layers}, and $W(\Phi) := \sum_{j=1}^L (\|A_j\|_{0}+\|b_j\|_{0})$
  is called the \emph{number of weights} of $\Phi$.
  Here, $\| A \|_0$ and $\| b \|_0$ denote the number of non-zero entries
  of the matrix $A$ or the vector $b$.
  Moreover, we refer to $N_L$ as the \emph{output dimension} of $\Phi$.
  The activation function $\varrho : \R \to \R, x \mapsto \max\{0, x\}$ is called the \emph{ReLU}.
  We call $R_\varrho \Phi$ a \emph{ReLU neural network}.
  Finally, the vector $(d, N_1, N_2, \dots, N_L) \in \N^{L+1}$ is called
  the \emph{architecture} of $\Phi$.
\end{definition}

\begin{remark}
  With notation as above, the number of \emph{hidden} layers of $\Phi$ is $L - 1$.
  A special type of neural networks are those with one hidden layer, i.e., $L = 2$;
  these are called \emph{shallow neural networks}.
  Realizations of such networks have the form
  \[
   \R^d \ni x \mapsto e + \sum_{i = 1}^N
                                 a_i \, \phi( \langle c_i, x \rangle + b_i) ,
  \]
  where $N \in \N$, $a_i, b_i, e \in \R$ and $c_i \in \R^d$ for $i = 1, \dots, N$.
\end{remark}

One important property of neural networks is that one can construct
complicated neural networks by combining simpler ones.
The following remark collects several standard operations
that were analyzed in \cite{petersen2018optimal}.

\begin{remark}\label{rem:StandardOperations}
 Let $\Phi_1, \Phi_2$ be two neural networks with input dimensions $d_1,d_2 \in \N$,
 $L_1, L_2$ layers and architectures $(d_1, N_1, N_2, \dots, N_{L_1}) \in \N^{L_1+1}$
 and $(d_2, M_1, M_2, \dots, M_{L_2}) \in \N^{L_2+1}$, respectively.
 Furthermore, let $\phi: \R \to \R$.
 \begin{itemize}
   \item If $d_2 = N_{L_1}$, then there exists a neural network $\Phi_3$
         such that $R_\phi\Phi_3 = R_\phi\Phi_2 \circ R_\phi\Phi_1$.
         Moreover, $\Phi_3$ can be chosen to have architecture
         \[
           \left(
             d_1, \,\,
             N_1, \,\,
             N_2, \,\,
             \dots, \,\,
             N_{L_1-1}, \,\,
             M_1, \,\,
             M_2, \,\,
             \dots, \,\,
             M_{L_2}
           \right)
           \in \N^{L_1 + L_2}
         \]
         and to satisfy $L(\Phi_3) = L_1 + L_2 - 1$
         and $W(\Phi_3) \leq W(\Phi_1) + W(\Phi_2) + N_{L_1 - 1} M_1$.

   \item If $L_1 = L_2$, $d_1 = d_2$ and $N_{L_1} = M_{L_1}$, then,
         given arbitrary $a,b \in \R$ there exists a neural network $\Phi_4$
         such that $R_\phi\Phi_4 = a R_\phi\Phi_1 + b R_\phi\Phi_2$.
         Moreover, $\Phi_4$ can be chosen to have architecture
         \[
           \left(
             d_1, \,\,
             N_1 + M_1, \,\,
             N_2 + M_2, \,\,
             \dots, \,\,
             N_{L_1-1} + M_{L_1-1}, \,\,
             N_{L_1}
           \right)
         \]
         and to satisfy $L(\Phi_4) = L_1$ and $W(\Phi_4) \leq W(\Phi_1) + W(\Phi_2)$.
 \end{itemize}
\end{remark}

%% file: 2-BarronUniform.tex

In this section, we formalize the notion of the (Fourier-analytic) Barron space
that we will use in the sequel.
We then prove that functions in the Barron class can be approximated
up to error $\calO (N^{-1/2})$ using shallow ReLU neural networks with $N$ neurons.
For neural networks with the Heaviside activation function,
this result is due to Barron \cite{BarronNeuralNetApproximation}.
Furthermore, it is claimed in \cite{BarronNeuralNetApproximation} that the result extends to
neural networks with sigmoidal activation functions, which would then also imply the same property
for the ReLU activation function $\varrho$, since $\phi(x) = \varrho(x) - \varrho(x-1)$
is sigmoidal.
However, regarding the extension to sigmoidal activation functions
there seems to be a gap in the proof presented in \cite{BarronNeuralNetApproximation}.
Namely, it is argued in the bottom left column on Page~3 of \cite{BarronNeuralNetApproximation}
that if $f$ is uniformly continuous and $\| f - f_T \|_{\sup} \lesssim T^{-1/2}$
where $f_T$ is of the form
${f_T (x) = c_0 + \sum_{k=1}^T c_k \Indicator_{(0,\infty)} (\langle a_k, x \rangle + b_k)}$
with $c_k, b_k \in \R$ and $a_k \in \R^d$, then one can also achieve
$\| f - g_T \|_{\sup} \lesssim T^{-1/2}$ for
${g_T (x) = C_0 + \sum_{k=1}^T C_k \, \phi (\langle A_k, x \rangle + B_k)}$,
where $\phi$ is measurable and \emph{sigmoidal}, meaning $\phi$ is bounded
with $\lim_{x \to \infty} \phi(x) = 1$ and $\lim_{x \to - \infty} \phi(x) = 0$.
As we could not verify this claim,
we provide an alternative proof for the case of the ReLU activation function,
based on the main ideas in \cite{BarronNeuralNetApproximation}.
In addition, our more careful proof shows that one can choose the weights of the neural network
to be uniformly bounded, independent of the desired approximation accuracy.

We first formalize the notion of \emph{Barron class functions},
essentially as introduced in \cite{BarronUniversalApproximation,BarronNeuralNetApproximation}.

\begin{definition}\label{def:BarronClassFunction}
  Let $\emptyset \neq X \subset \R^d$ be bounded.
  A function $f : X \to \R$ is said to be of \emph{Barron class with constant $C > 0$},
  if there are $x_0 \in X$, $c \in [-C,C]$, and a measurable function $F : \R^d \to \CC$ satisfying
  \begin{equation}
    \int_{\R^d} |\xi|_{X,x_0} \cdot |F(\xi)| \, d \xi \leq C
    \quad \text{and} \quad
    f(x)
    = c + \int_{\R^d}
            \big(
              e^{i \langle x, \xi \rangle}
              - e^{i \langle x_0, \xi \rangle}
            \big)
            \cdot F(\xi)
          \, d \xi
    \qquad \forall \, x \in X,
    \label{eq:FourierRepresentation}
  \end{equation}
  where we used the notation $|\xi|_{X,x_0} := \sup_{x \in X} |\langle \xi, x - x_0 \rangle|$.
  We write $\Barron_C (X, x_0)$ for the class of all such functions.
\end{definition}

\begin{rem*}
  The precise choice of the ``base point'' $x_0 \in X$ is immaterial,
  in the sense that it at most changes the resulting norm by a factor of $2$.
  Indeed, let $x_0, x_1 \in X$ and assume that $f$ satisfies \eqref{eq:FourierRepresentation}
  with $|c| \leq C$.
  Then we see for arbitrary $\xi \in \R^d$ and $x \in X$ that
  \[
    |\langle \xi , x - x_1 \rangle|
    \leq |\langle \xi , x - x_0 \rangle| + |\langle \xi, x_0 - x_1 \rangle|
    =    |\langle \xi , x - x_0 \rangle| + |\langle \xi, x_1 - x_0 \rangle|
    \leq 2 \, |\xi|_{X,x_0} ,
  \]
  meaning $|\xi|_{X, x_1} \leq 2 \, |\xi|_{X, x_0}$ and hence
  $\int_{\R^d} |\xi|_{X,x_1} \cdot |F(\xi)| \, d \xi \leq 2 C$.
  Furthermore, setting
  \({
    c' := c + \int_{\R^d}
                \bigl(e^{i \langle x_1, \xi \rangle} - e^{i \langle x_0, \xi \rangle}\bigr)
                \, F(\xi)
              \, d \xi ,
  }\)
  we have
  \(
    f(x)
    = c' + \int_{\R^d}
             \bigl(e^{i \langle x,\xi \rangle} - e^{i \langle x_1, \xi \rangle}\bigr)
             F(\xi)
           \, d \xi
  \)
  and
  \(
    |e^{i \langle x_1, \xi \rangle} - e^{i \langle x_0, \xi \rangle}|
    \leq |\langle x_1 - x_0, \xi \rangle|
    \leq |\xi|_{X,x_0},
  \)
  which implies
  \(
    |c'|
    \leq C + \int_{\R^d} |\xi|_{X,x_0} \, |F(\xi)| \, d \xi
    \leq 2 C .
  \)
  Overall, this shows that $f \in \Barron_{2C}(X,x_1)$
  and hence $\Barron_C(X,x_0) \subset \Barron_{2 C}(X, x_1)$.

  Based on this, it is straightforward to see
  \[
    \forall \, \emptyset \neq Y \subset X \text{ and } x_0 \in X, y_0 \in Y: \quad
    \Barron_C(X, x_0) \subset \Barron_{2 C}(Y, y_0) .
  \]
  For the sake of clarity, note that if $\{x_0\}\subset Y\subset X$ and $f\in\Barron_C(X,x_0)$,
  then clearly $f|_Y\in\Barron_C(Y,x_0)$ as the conditions in \eqref{eq:FourierRepresentation}
  are already satisfied.
  Therefore the inclusion $\Barron_C(X, x_0) \subset \Barron_{2 C}(Y, y_0)$
  from above is to be understood, by slight abuse of notation, in the sense of function restrictions.
\end{rem*}

The following result shows that functions from the Barron class can be uniformly
approximated with error $\CalO(N^{-1/2})$ using shallow ReLU neural networks
with $\CalO(N)$ neurons.
It also shows that the weights of the approximating network can be chosen to be bounded
in a suitable way.
We emphasize that the result is \emph{not} covered by \cite[Theorem~12]{ma2020towards},
since the Fourier-analytic Barron space that we use here is \emph{not}
contained in the Barron space considered in \cite{ma2020towards};
see \Cref{prop:FourierBarronNotInReLUBarron}.

\begin{proposition}\label{prop:BarronFunctionApproximation}
  There is a universal constant $\kappa > 0$ with the following property:
  For any bounded set $X \subset \R^d$ with nonempty interior, for any $C > 0$, $x_0 \in X$
  and $f \in \Barron_C (X, x_0)$, and any $N \in \N$, there is a shallow neural network $\Phi$
  with $8 N$ neurons in the hidden layer such that
  \[
    \| f - R_\varrho \Phi \|_{\sup} \leq \kappa \, \sqrt{d} \cdot C \cdot N^{-1/2} .
  \]
  Furthermore, one can choose all weights and biases of $\Phi$ to be bounded by
  \[
    \bigl(5 + \vartheta(X, x_0)\bigr)
    \cdot \bigl(1 + \| x_0 \|_{\ell^1}\bigr)
    \cdot \sqrt{C} ,
    \quad \text{where} \quad
    \vartheta(X, x_0)
    := \sup_{\xi \in \R^d \setminus \{ 0 \}}
         \Big(
           \| \xi \|_{\ell^\infty} \big/ |\xi|_{X,x_0}
         \Big)
    .
  \]
\end{proposition}

\begin{remark}\label{rem:explanationofVartheta}
  The quantity $\vartheta(X, x_0)$ roughly speaking measures how big of a rectangle
  the set $X$ contains.
  More precisely, assume that $X \supset [a,b]$ where $b_i - a_i \geq \eps > 0$
  for all $i \in \FirstN{d}$.
  Then we see with the standard basis $(e_1,\dots,e_d)$ of $\R^d$
  that
  \[
    \eps \, |\xi_i|
    = \bigl|
        \langle \xi, a + \eps \, e_i - x_0  \rangle
        - \langle \xi, a - x_0 \rangle
      \bigr|
    \leq |\langle \xi, a + \eps e_i - x_0 \rangle| + |\langle \xi, a - x_0 \rangle|
    \leq 2 \sup_{x \in X}
             |\langle \xi, x - x_0 \rangle| .
  \]
  Since this holds for all $i \in \FirstN{d}$,
  we see $|\xi|_{X,x_0} \geq \frac{\eps}{2} \, \| \xi \|_{\ell^\infty}$
  and hence $\vartheta(X, x_0) \leq \frac{2}{\eps}$.

  Note that since $X$ has nonempty interior, we can always find a sufficiently
  small non-degenerate rectangle in $X$; therefore, $|\xi|_{X,x_0} \gtrsim \| \xi \|_{\ell^\infty}$
  for all $\xi \in \R^d$.
\end{remark}

\begin{proof}
  It is enough to prove the claim for the case $C = 1$.
  Indeed, for $f \in \Barron_C (X, x_0)$, we have $\widetilde{f} := f / C \in \Barron_1 (X, x_0)$.
  Applying the claim to $\widetilde{f}$, we thus get
  $\| \widetilde{f} - \widetilde{g} \|_{\sup} \leq \kappa \sqrt{d} \cdot N^{-1/2}$,
  where $\widetilde{g}(x) = \sum_{i=1}^{8 N} a_i \, \varrho(b_i + \langle w_i, x \rangle)$ with
  $\| w_i \|_{\ell^\infty}, |a_i|, |b_i| \leq (5 + \vartheta(X,x_0)) \cdot (1 + \| x_0 \|_{\ell^1})$.
  Hence, defining
  $g(x) = \sum_{i=1}^{8 N} \sqrt{C} a_i \, \varrho (\sqrt{C} b_i + \langle \sqrt{C} w_i, x \rangle)$,
  we have $g(x) = C \cdot \widetilde{g}(x)$, which easily yields the claim for $f$.
  We will thus assume $C = 1$ in what follows.
  The actual proof is divided into three steps.

  \smallskip{}

  \noindent
  \textbf{Step~1} \emph{(Writing $f$ as an expectation of indicators of half-spaces):}
  Let $c \in [-C, C]$ and ${F : \R^d \to \CC}$ such that \Cref{eq:FourierRepresentation} is satisfied.
  The case where $F = 0$ almost everywhere is easy to handle;
  we thus assume that $F \neq 0$ on a set of positive measure.

  Set $X_0 := \{ x - x_0 \colon x \in X \}$, and define
  $f_{0} : X_0 \to \R$ by $f_{0}(x) := f(x + x_0) - c$ and
  $F_0 : \R^d \to \CC, \xi \mapsto e^{i \langle x_0, \xi \rangle} \, F(\xi)$.
  With this notation, we have
  $f_{0}(x) = \int_{\R^d} (e^{i \langle x, \xi \rangle} - 1) \cdot F_0(\xi) \, d \xi$
  and ${\int_{\R^d} |\xi|_{X_0} \cdot |F_0(\xi)| \, d \xi \leq C}$,
  where $|\xi|_{X_0} := \sup_{x \in X_0} |\langle x,\xi \rangle| = |\xi|_{X,x_0}$.
  Thus, (the proof of) \cite[Theorem~2]{BarronNeuralNetApproximation} shows for all $x \in X_0$ that
  \[
    f_0 (x)
    = v \cdot
      \int_{\R^d}
        \int_0^1
          \Big(
            \Indicator_{(0,\infty)} \bigl(-\langle \xi / |\xi|_{X_0}, \,\, x \rangle - t\bigr)
            - \Indicator_{(0,\infty)} \bigl(\langle \xi / |\xi|_{X_0}, \,\, x \rangle - t\bigr)
          \Big)
          \cdot s(\xi,t) \cdot p(\xi,t)
        \, dt
      \, d \xi ,
  \]
  where, using the polar decomposition $F_0(\xi) = |F_0(\xi)| \cdot e^{i \, \theta_\xi}$,
  the function ${s : \R^d \times [0,1] \to \{ \pm 1 \}}$ is given by
  $s(\xi, t) = \sign \bigl( \sin(t \, |\xi|_{X_0} + \theta_\xi)\bigr)$,
  while $p : \R^d \times [0,1] \to [0,\infty)$ is defined as
  \({
    p(\xi ,t)
    = \frac{1}{v}
      \cdot |\xi|_{X_0}
      \cdot \bigl|\sin(t \, \bigr|\xi|_{X_0} + \theta_\xi)|
      \cdot |F_0(\xi)|
    .
  }\)
  Finally,
  \[
    v
    = \int_{\R^d}
        \int_0^1
          |\xi|_{X_0}
          \cdot |\sin(t |\xi|_{X_0} + \theta_\xi)|
          \cdot |F_0(\xi)|
        \, d t
      \, d \xi
    \leq C
  \]
  is chosen such that $p$ is a probability density function.
  It is easy to see $v > 0$ since $F_0 \neq 0$ on a set of positive measure.

  For brevity, define $\Omega := (\R^d \setminus \{ 0 \}) \times [0,1]$.
  Furthermore, set $\xi^\ast := \xi / |\xi|_{X_0}$ for $\xi \in \R^d \setminus \{ 0 \}$
  (where we note that $|\xi|_{X_0} > 0$ since $X_0$ has nonempty interior),
  and for $x \in X_0$ define
  \[
    \Gamma_x : \quad
    \Omega \to [-1,1], \quad
    (\xi,t) \mapsto \Indicator_{(0,\infty)} (- \langle \xi^\ast, x \rangle - t)
                    - \Indicator_{(0,\infty)} (\langle \xi^\ast, x \rangle - t) .
  \]
  Finally, let us set
  $V_{\pm} := \int_{\R^d} \int_0^1 \Indicator_{s(\xi,t) = \pm 1} \cdot p(\xi,t) \, d t \, d \xi$,
  and define probability measures $\mu_{\pm}$ on $\Omega$ via
  \[
    d \mu_{\pm}
    := \frac{1}{V_{\pm}}
       \cdot \Indicator_{s(\xi,t) = \pm 1}
       \cdot p(\xi,t) \, d t \, d \xi .
  \]
  Note that $V_+, V_- \geq 0$ and $V_+ + V_- = 1$.
  Also note that strictly speaking $\mu_{\pm}$ is only well-defined in case of $V_{\pm} > 0$.
  In case of $V_{\pm} = 0$, one can simply drop the respective term in what follows;
  we leave the straightforward modifications to the reader.

  Given all these notations, we see that $f_0  = v \cdot (V_+ \cdot f_+ - V_- \cdot f_- )$, where
  \[
    f_{\pm} : X_0 \to \R
    \quad \text{is defined by} \quad
    f_{\pm} (x)
    := \int_{\Omega}
         \Gamma_x (\xi,t)
       \, d \mu_{\pm} (\xi,t) .
  \]

  It is enough to show
  \(
    \| f_{\pm} - R_\varrho \Phi_{\pm} \|_{\sup}
    \leq N^{-1/2} \cdot \bigl(\frac{C}{v \, V_{\pm}} + \kappa_0 \sqrt{d} \, \bigr)
  \)
  for a shallow neural network $\Phi_{\pm}$ with $4 \, N$ neurons in the hidden layer
  and with all weights and biases bounded by ${4 + \vartheta(X, x_0)}$.
  Indeed, once this is shown, it is easy to see that there exists
  a shallow network $\Phi$ with $8 N$ neurons in the hidden layer satisfying
  \[
    R_\varrho \Phi (x)
    = c + v \, V_+ \cdot R_\varrho \Phi_+ (x - x_0) - v \, V_- \cdot R_\varrho \Phi_- (x - x_0) .
  \]
  Because of
  $f(x) = c + f_0(x - x_0) = c + v \, V_+ \cdot f_+ (x - x_0) - v \, V_- \cdot f_-(x-x_0)$
  and $0 < v \leq C$, this yields
  \begin{align*}
    \| f - R_\varrho \Phi \|_{\sup}
    & \leq N^{-1/2}
           \cdot \Big(
                     v \, V_+ \cdot \big( \tfrac{C}{v \, V_+} + \kappa_0 \sqrt{d} \, \big)
                   + v \, V_- \cdot \big( \tfrac{C}{v \, V_-} + \kappa_0 \sqrt{d} \, \big)
                 \Big) \\
    & = N^{-1/2} \cdot \bigl(2 C + v \kappa \, \sqrt{d} \,\bigr)
      \leq \bigl( 2 + \kappa_0 \, \sqrt{d} \, \bigr) \cdot C \cdot N^{-1/2}
      \leq \kappa \, \sqrt{d} \cdot C \cdot N^{-1/2}
  \end{align*}
  for a suitable absolute constant $\kappa > 0$.
  Again, since $0 < v \leq C$ and $c \in [-C,C]$ as well as $0 \leq V_{\pm} \leq 1$,
  and since we assume $C = 1$, it is easy to see that $\Phi$ can be chosen in such a way
  that all weights of $\Phi$
  are bounded by $\bigl(4 + \vartheta(X, x_0)\bigr) \cdot (1 + \| x_0 \|_{\ell^1})$.
  Here, we use that if $\| w \|_{\ell^\infty}, |b| \leq 4 + \vartheta(X, x_0)$, then
  \(
    \varrho(\langle w, x - x_0 \rangle + b)
    = \varrho(\langle w, x \rangle + b - \langle w,x_0 \rangle)
    ,
  \)
  where
  \[
    \bigl|b - \langle w, x_0 \rangle\bigr|
    \leq |b| + \bigl(4 + \vartheta(X, x_0)\bigr) \, \| x_0 \|_{\ell^1}
    \leq \bigl(4 + \vartheta(X, x_0)\bigr) \cdot \bigl(1 + \| x_0 \|_{\ell^1}\bigr) .
  \]

  \medskip{}

  \noindent
  \textbf{Step 2} \emph{(Approximating $f$ by an expectation of ReLU networks):}
  For $\eps > 0$, define
  \[
    H_\eps : \quad
    \R \to [0,1], \quad
    x \mapsto \frac{1}{\eps} \bigl( \varrho(x) - \varrho(x - \eps) \bigr) ,
  \]
  noting that $H_\eps (x) = \Indicator_{(0,\infty)} (x)$ for all $x \in \R \setminus (0,\eps)$.
  Next, for $\eps > 0$ and $x \in X_0$, set
  \[
    N_{\eps,x} : \quad
    \Omega \to [-1,1], \quad
    (\xi,t) \mapsto H_\eps (-\langle \xi^\ast, x \rangle - t)
                    - H_\eps (\langle \xi^\ast, x \rangle - t) .
  \]
  Setting
  \({
    J_{\xi,x}^{(\eps)}
    := [-\langle \xi^\ast,x \rangle - \eps, \,\, - \langle \xi^\ast, x \rangle]
       \cup [\langle \xi^\ast,x \rangle - \eps, \,\, \langle \xi^\ast, x \rangle] ,
  }\)
  we have ${\Gamma_x (\xi,t) = N_{\eps,x} (\xi,t)}$ for all $(\xi,t) \in \Omega$
  with $t \notin J_{\xi,x}^{(\eps)}$.
  Thus, using the bound ${0 \leq p(\xi,t) \leq \frac{1}{v} |\xi|_{X_0} \, |F(\xi)|}$
  and the definitions of $f_{\pm}$ and $\mu_{\pm}$, we see for all $x \in X_0$ that
  \begin{align*}
    \Big| f_{\pm}(x) - \int_{\Omega} N_{\eps,x} (\xi,t) \, d \mu_{\pm} (\xi,t) \Big|
    & \leq \int_{\R^d \setminus \{ 0 \}}
             \int_0^1
               2 \cdot \Indicator_{J_{\xi,x}^{(\eps)}} (t)
                 \cdot \frac{1}{V_{\pm}} \, p(\xi,t)
             \, d t
           \, d \xi \\
    & \leq \frac{4 \eps}{v \, V_{\pm}} \int_{\R^d} |\xi|_{X_0} \cdot |F(\xi)| \, d \xi
      \leq \frac{4 \eps C}{v \, V_{\pm}} .
  \end{align*}
  We now choose $\eps := \frac{1}{4} N^{-1/2}$ and define
  ${f_{\pm,\eps} : X_0 \to \R, x \mapsto \int_{\Omega} N_{\eps,x}(\xi,t) \, d \mu_{\pm} (\xi,t)}$.
  Then the preceding estimate shows that
  $\| f_{\pm} - f_{\pm,\eps} \|_{\sup} \leq N^{-1/2} \cdot \frac{C}{v \, V_{\pm}}$.

  \begin{figure}[ht]
    \begin{center}
      \includegraphics[width=15cm]{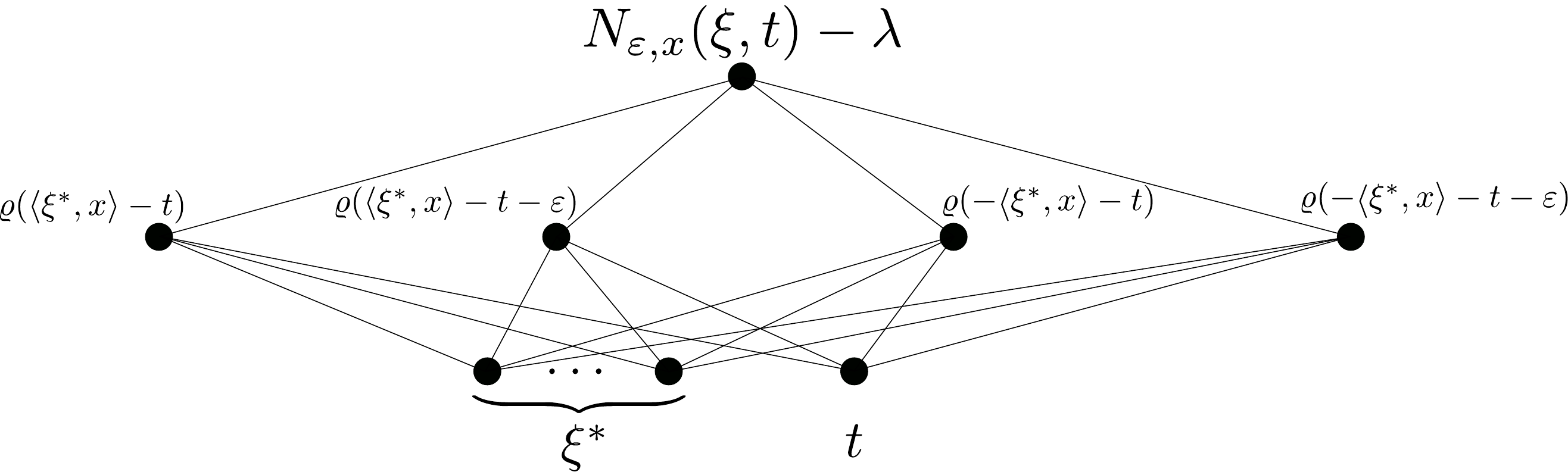}
    \end{center}
    \caption{\label{fig:BarronProofNetworkFigure}
      Representation of the function $(\xi, t) \mapsto N_{\eps,x}(\xi, t) - \lambda$
      as a ReLU network with  $L = 2$ layers,
      $W = 4d + 16$ weights, and $U = 5$ computation units
      (using the notation of \cite{BartlettNearlyTightVCDimensionBounds}).}
  \end{figure}

  \medskip{}

  \noindent
  \textbf{Step 3} \emph{(Using bounds for empirical processes to complete the proof):}
  Denote by $\calG_0$ the set of all functions $g : \R^d \times \R \to \R$ that are implemented
  by ReLU neural networks with the architecture shown in Figure~\ref{fig:BarronProofNetworkFigure}
  (that is, fully connected with one hidden layer containing four neurons).
  Then the VC dimension bound for neural networks
  shown in \cite[Theorem~6]{BartlettNearlyTightVCDimensionBounds}
  implies that there is an absolute constant $\kappa_1 \in \N$ such that
  \[
    \VC(\{ \Indicator_{g > 0} \colon g \in \calG_0 \})
    \leq \kappa_1 \, d .
  \]

  Moreover, using the map
  ${\Theta : \Omega \to \R^d \times [0,1], (\xi,t) \mapsto (\xi^\ast , t) = (\xi/|\xi|_{X_0}, t)}$,
  the construction in Figure~\ref{fig:BarronProofNetworkFigure} shows for arbitrary $\lambda \in \R$
  that
  \[
    \{ \Indicator_{N_{\eps,x} > \lambda} \colon x \in X_0 \}
    \subset \{ \Indicator_{g \circ \Theta > 0} \colon g \in \calG_0 \} .
  \]
  Directly from the definition of the VC dimension, we see that composing a class of functions with
  a fixed map (in this case, $\Theta$) can not increase the VC dimension, so that we get
  \({
    \VC (\{ \Indicator_{N_{\eps,x} > \lambda} \colon x \in X_0 \})
    \leq \kappa_1 \, d
  }\)
  for all $\eps > 0$ and $\lambda \in \R$.

  Now, using the bound in \Cref{prop:PseudoDimensionGeneralizationBound}
  and recalling that $\EE_{(\xi,t) \sim \mu_{\pm}} [N_{\eps,x}(\xi, t)] = f_{\pm,\eps}(x)$,
  we see that if we choose $(\xi_1,t_1),\dots,(\xi_N,t_N) \overset{\text{i.i.d.}}{\sim} \mu_{\pm}$,
  then there is a universal constant $\kappa_2 > 0$ satisfying for all $N \in \N$ that%
  \footnote{Strictly speaking, \Cref{prop:PseudoDimensionGeneralizationBound}
  yields a bound for
  \[
    \sup_{X_{00} \subset X_0 \text{ finite}} \,\,
      \EE
      \Big[
        \sup_{x \in X_{00}}
          \big| f_{\pm,\eps}(x) - N^{-1} \textstyle{\sum_{i=1}^N} N_{\eps,x}(\xi_i, t_i) \big|
      \Big] .
  \]
  But since $x \mapsto f_{\pm,\eps}(x)$ and $x \mapsto N_{\eps,x}(\xi_i, t_i)$
  are continuous, this coincides with the expression in \Cref{eq:PseudoDimensionBoundApplication}.}
  \begin{equation}
    \EE
    \Big[
      \sup_{x \in X_0}
        \Big|
          f_{\pm,\eps}(x)
          - \frac{1}{N} \sum_{i=1}^N N_{\eps,x}(\xi_i, t_i)
        \Big|
    \Big]
    \leq \kappa_2 \cdot \sqrt{\frac{\kappa_1 d}{N}} .
    \label{eq:PseudoDimensionBoundApplication}
  \end{equation}
  In particular, there is one specific realization
  $\bigl( (\xi_1,t_1),\dots,(\xi_N,t_N)\bigr) \in \Omega^N$ such that
  \[
    \sup_{x \in X_0}
    \Big|
      f_{\pm,\eps}(x) - \frac{1}{N} \sum_{i =1}^N N_{\eps,x}(\xi_i,t_i)
    \Big|
    \leq \kappa \, \sqrt{d} \, N^{-1/2} .
  \]
  Clearly,
  \(
    g_{\pm, \eps} :
    \R^d \to \R,
    x \mapsto \frac{1}{N} \sum_{i=1}^N N_{\eps,x}(\xi_i,t_i)
  \)
  is implemented by a shallow ReLU network with $4N$ neurons in the hidden layer, as follows from
  \[
    \frac{1}{N}N_{\eps,x}(\xi_i,t_i)
    = \frac{\eps^{-1}}{N}
      \cdot \Big(
              \varrho \bigl(-\langle \xi_i^\ast, x \rangle - t_i\bigr)
              - \varrho \bigl(-\langle \xi_i^\ast, x \rangle - t_i - \eps\bigr)
              - \varrho \bigl(\langle \xi_i^\ast, x \rangle - t_i\bigr)
              + \varrho \bigl(\langle \xi_i^\ast, x \rangle - t_i - \eps\bigr)
            \Big) .
  \]
  Now, note by definition of $\vartheta(X,x_0)$ and $\xi^\ast = \xi/|\xi|_{X_0}$ that
  $\| \xi_i^\ast \|_{\ell^\infty} \leq \vartheta(X, x_0)$.
  Furthermore, $|t_i| \leq 1$.
  Finally, by choice of $\eps = \frac{1}{4} N^{-1/2}$, we see $\eps^{-1}/N = 4 N^{-1/2} \leq 4$.
  Overall, we thus see that $g_{\pm,\eps} = R_\varrho \Phi_{\pm}$ where the shallow neural network
  $\Phi_{\pm}$ has $4 N$ neurons in the hidden layer and all weights and biases
  bounded by $4 + \vartheta(X, x_0)$.
\end{proof}


%% file: 3-Approximation.tex

In this section, we show that indicator functions of sets with Barron class boundary
are well approximated by ReLU neural networks.
Essentially the only property of Barron class functions that we will need is
that they can be uniformly approximated up to error $\CalO (N^{-1/2})$ by shallow ReLU
networks with $N$ neurons and suitably bounded weight.
Thus, to allow for a slightly more general result, we introduce a
``Barron approximation space'' containing all such functions.

\begin{definition}\label{def:BarronApproxSpace}
Let $d \in \N$ and let $X \subset \R^d$ be bounded with nonempty interior.
For $C > 0$, we define the \emph{Barron approximation set} $\BarronApproximation_C (X)$
as the set of all functions $f : X \to \R$ such that for every $N \in \N$
there is a shallow neural network $\Phi$ with $N$ neurons in the hidden layer such that
\[
  \| f - R_\varrho \Phi \|_{\sup} \leq \sqrt{d} \cdot C \cdot N^{-1/2}
\]
and such that all weights (and biases) of $\Phi$ are bounded in absolute value by
\[
  \sqrt{C}
  \cdot \Big(
          5 + \inf_{x_0 \in X}
                \bigl[\, \| x_0 \|_{\ell^1} + \vartheta(X, x_0) \,\bigr]
        \Big) ,
  \quad \text{where} \quad
  \vartheta(X, x_0)
  := \sup_{\xi \in \R^d \setminus \{ 0 \}}
       \big( \| \xi \|_{\ell^\infty} \big/ |\xi|_{X,x_0} \big) .
\]
The set $\BarronApproximation(X) = \bigcup_{C > 0} \BarronApproximation_C (X)$
is called the \emph{Barron approximation space}.
\end{definition}

\begin{remark}\label{rem:AllBarronSetsAreSubsetsOfTheApproximationSpace}
a) Using \Cref{prop:FourierBarronNotInReLUBarron}, it is not hard to see
$\mathcal{B}_C(X, x_0) \subset \mathcal{BA}_{\kappa_0 C}(X)$
for every $C > 0$, with a constant $\kappa_0 > 0$ that is \emph{absolute},
(i.e., independent of all other quantities and objects).

\smallskip{}

b) For the infinite-width Barron space $\Barron_\varrho (X)$ associated to the ReLU
function (which will be formally introduced in \Cref{sec:BarronSpaces}), it follows from
\cite[Theorem~12]{ma2020towards} that
\[
  \Barron_{\varrho, C}(X)
  := \big\{ f \in \Barron_{\varrho}(X) \,\,\colon\,\, \|f\|_{\Barron_{\varrho}(X)} \leq C \big\}
  \subset \BarronApproximation_{\sigma C} (X) ,
\]
where the constant $\sigma > 0$ scales polynomially with $d$
and linearly with $\sup_{x \in X} \| x \|_{\ell^\infty}$.

\smallskip{}

c) If $Y \subset X$ has nonempty interior, we have
$\vartheta(X,y_0) \leq \vartheta(Y,y_0)$ for all $y_0 \in Y$ and hence
\(
  \inf_{x_0 \in X}
    \big[ \| x_0 \|_{\ell^1} + \vartheta(X, x_0) \big]
  \leq \inf_{y_0 \in Y}
    \big[ \| y_0 \|_{\ell^1} + \vartheta(X, y_0) \big]
  \leq \inf_{y_0 \in Y}
    \big[ \| y_0 \|_{\ell^1} + \vartheta(Y, y_0) \big] .
\)
Based on this, it is straightforward to see
\begin{equation}
  f|_Y \in \BarronApproximation_C (Y)
  \quad \text{if } f \in \BarronApproximation(X)
  \text{ and } Y \subset X \text{ has nonempty interior} .
  \label{eq:BarronApproximationSpaceRestriction}
\end{equation}
\end{remark}

Using the notion of Barron approximation spaces, we can now formally define
sets with Barron class boundary. 

\begin{definition}\label{def:BarronClassifiers}
  Let $d \in \N_{\geq 2}$ and $B > 0$ and let $Q = [a,b] \subset \R^d$ be a rectangle.
  A function $F : Q \to \R$ is called a \emph{Barron horizon function with constant $B$},
  if there are $i \in \FirstN{d}$ and ${f \in \mathcal{BA}_B \bigl([a^{(i)}, b^{(i)}]\bigr)}$
  as well as $\theta \in \{ \pm 1 \}$ such that
  \[
    F(x) = \Indicator_{\theta x_i \leq f(x^{(i)})}
    \qquad \forall \, x \in Q .
  \]
  We write $\BarronHorizon_B (Q)$ for the set of all such functions.

  Finally, given $M \in \N$ and $B > 0$, a compact set $\Omega \subset \R^d$
  is said to have a \emph{Barron class boundary with constant $B$} if there exist rectangles
  $Q_1,\dots,Q_M \subset \R^d$ such that $\Omega \subset \bigcup_{i=1}^M Q_i$ where the rectangles
  have disjoint interiors (i.e., $Q_i^{\circ} \cap Q_j^{\circ} = \emptyset$ for $i \neq j$)
  and such that ${\Indicator_{Q_i \cap \Omega} \in \BarronHorizon_{B} (Q_i)}$
  for each $i \in \FirstN{M}$.
  We write $\BarronBoundary_{B,M}(\R^d)$ for the class of all such sets.
  Also, a family $(Q_j)_{j=1}^M$ of rectangles as above is called an \emph{associated cover}
  of $\Omega$.
\end{definition}

\begin{remark}
  By \Cref{rem:AllBarronSetsAreSubsetsOfTheApproximationSpace}, the set of functions
  with Barron class boundary contains all characteristic functions
  of sets whose boundary is locally described by functions in the Fourier-analytic
  Barron space or the infinite-width Barron space associated to the ReLU.
\end{remark}

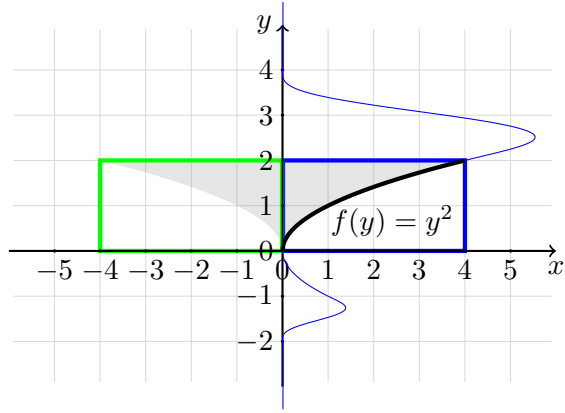
\begin{figure}[ht]
  \begin{center}
    \begin{tikzpicture}[scale=0.6]
      \draw[very thin, gray!30, step=1 cm](-5.9,-2.9) grid (5.9,4.9);

      \fill [gray, samples=300, opacity=0.2, domain=-4:4, variable=\x]
        (-4, 2)
        -- plot ({\x}, {sqrt(abs(\x))})
        -- (4, 2)
        -- cycle;

      \draw [ultra thick, blue]
        (0.01,0)
        -- (.01,2)
        -- (4,2)
        -- (4,0)
        -- cycle;

      \draw [ultra thick, green]
        (-0.01,0)
        -- (-0.01,2)
        -- (-4,2)
        -- (-4,0)
        -- cycle;

      \draw [thick] [->] (-6,0)--(6,0) node[right, below] {$x$};
       \foreach \x in {-5,...,5}
         \draw[xshift=\x cm, thick] (0pt,-1pt)--(0pt,1pt) node[below] {$\x$};

      \draw [thick] [->] (0,-3)--(0,5) node[above, left] {$y$};
       \foreach \y in {-2,...,4}
         \draw[yshift=\y cm, thick] (-1pt,0pt)--(1pt,0pt) node[left] {$\y$};

      \draw [ultra thick, domain=0:2, samples=300, variable=\x]
        plot ({\x*\x}, {\x}) node[right] at (0.8,0.65) {$f(y)=y^2$};

      \draw [thin, blue, domain=-1:0, samples=300, variable=\x]
        plot ({\x*\x}, {\x}) ;

      \draw [thin, blue, domain=-2:-1, samples=300, variable=\y]
        plot({(\y+2)^4*(1+(-6)*(\y+1)+19*(\y+1)^2-44*(\y+1)^3)},{\y});

      \draw [thin, blue, domain=2:4, samples=300, variable=\y]
        plot ({(\y-4)^4*(1/4+3/4*(\y-2)+19/16*(\y-2)^2+44/32*(\y-2)^3}, {\y});

      \draw [thin, blue]
        (0.01,-2)
        --(0.01,-3.5)
        -- cycle;

      \draw [thin, blue]
        (0.01,4)
        --(0.01,5.5)
        -- cycle;

      %
      %
      %
      %
      %
    \end{tikzpicture}
  \end{center}
  \captionsetup{width=.7\linewidth}
  \caption{\label{fig:CuspDomain} \small An illustration of the cusp domain $\Omega$ (shown in gray)
  discussed in Part~(2) of the example below.
  The blue and green boxes show the rectangles $Q_1,Q_2 \subset \R^2$ satisfying
  $\Omega \subset Q_1 \cup Q_2$ and such that $\Indicator_{\Omega \cap Q_i}$
  is a Barron horizon function.
  For more details see the example below.}
\end{figure}

The following example illustrates the above definition.

\begin{example*}
  \emph{(1)} Every set $\Omega$ of the form $\Omega = \{ x \in Q \colon x_1 \leq f(x_2,\dots,x_d) \}$
  for a rectangle $Q = [a^{(1)},b^{(1)}] \times Q' \subset \R^d$ and a function $f : Q' \to \R$
  from the Fourier-analytic Barron class $\Barron_{B}(Q',x_0)$ (for arbitrary $x_0 \in Q'$)
  belongs to $\BarronBoundary_{\kappa_0 B,1}(\R^d)$, for the absolute constant $\kappa_0 > 0$
  from \Cref{rem:AllBarronSetsAreSubsetsOfTheApproximationSpace}.

  Examples for such functions $f$ are discussed in great length in
  \cite[Section~IX]{BarronUniversalApproximation};
  here, we just mention three special cases.
  First, for the Gaussian $f(x) = e^{-|x|^2 / 2}$, it holds that
  ${f \in \Barron_{\sqrt{d}}(Q',x_0)}$ for any rectangle $Q' \subset \R^{d-1}$
  and any $x_0 \in Q'$; thus, one only has a polynomial dependence on the dimension.
  Second, if $Q' = [0,1]^{d-1}$ and ${f(x) = \sum_{k \in \Z^{d-1}} c_k e^{2\pi i \langle k,x \rangle}}$,
  then $f \in \Barron_{C}(Q',0)$ for $C = |c_0| +  \sum_{k \in \Z^{d-1}} |k| \,|c_k|$;
  this essentially follows as in \cite[Section~IX, Point~(16)]{BarronUniversalApproximation}.
  Finally, if $f \in C^{k}(Q')$ for $k \geq 2 + \lfloor (d-1) / 2 \rfloor $,
  then \cite[Section~IX, Point~(15)]{BarronUniversalApproximation} shows that $f$ belongs to
  $\Barron_C (Q',x_0)$, for a suitable $C = C(f,Q') > 0$.
  This last observation, however, is more of qualitative than of quantitative use,
  since the resulting constant $C$ is often quite large if $d$ is large.

  \medskip{}

  \emph{(2)} The class of sets with Barron class boundary also contains sets that are not necessarily
  Lipschitz domains.
  An example of such a domain is the cusp domain
  \[
    \Omega
    = \big\{
        (x,y) \in [-4,4] \times [0, 2]
        \quad\colon\quad
        y \geq \sqrt{|x|}
    \big\}
  \]
  shown in \Cref{fig:CuspDomain}.
  Indeed, we claim for the rectangles $Q_1 = [0,4] \times [0, 2]$
  and $Q_2 = [-4,0] \times [0, 2]$ that $\Indicator_{\Omega \cap Q_i}$
  is a Barron horizon function.
  We only verify this for $Q_1$.
  To see this, note that the function $f : [0, 2] \to \R, y \mapsto y^2$
  can be extended to a function $f \in C_c^3 (\R)$;
  one such extension is shown in \Cref{fig:CuspDomain}.
  As seen above, this implies that $f \in \Barron_C([0,2], 0)$
  for a certain $C > 0$.
  Because of $\Indicator_{\Omega \cap Q_1} (x,y) = \Indicator_{x \leq f(y)}$
  for $(x,y) \in Q_1$, this implies that $\Indicator_{\Omega \cap Q_1}$ is a Barron horizon function.
\end{example*}

We will show in \Cref{sec:GeneralMeasures} that it is impossible to derive nontrivial minimax
bounds for the class of sets with Barron boundary
for the case of \emph{general} probability measures.
For this reason, we will restrict to the following class of measures.

\begin{definition}\label{def:TubeCompatibleMeasures}
  Let $\mu$ be a finite Borel measure on $\R^d$.
  We say that $\mu$ is \emph{tube compatible} with parameters $\alpha \in (0,1]$ and $C > 0$
  if for each measurable function $f : \R^{d-1} \to \R$, each $i \in \FirstN{d}$
  and each $\eps \in (0,1]$, we have
  \[
    \mu\bigl(T_{f,\eps}^{(i)}\bigr) \leq C \cdot \eps^\alpha
    \quad \text{where} \quad
    T_{f,\eps}^{(i)} := \big\{ x \in \R^d \colon | x_i - f(x^{(i)}) | \leq \eps \big\} .
  \]
  The set $T_{f,\eps}^{(i)}$ is called a \emph{tube of width $\eps$ (associated to $f$)}.
\end{definition}

\begin{remark}\label{rem:TubeCompatibleMeasures}
  The definition might appear technical, but it is satisfied for a wide class of product measures.
  For instance, if $\mu_1,\dots,\mu_d$ are Borel probability measures on $\R^d$
  such that each distribution function $F_i (x) = \mu_i ( (-\infty,x] )$
  is $\alpha$-Hölder continuous with constant $C$,
  then the product measure $\mu = \mu_1 \otimes \cdots \otimes \mu_d$ is tube compatible
  with parameters $\alpha$ and $2^\alpha C$, since Fubini's theorem shows for
  $\mu^{(i)} := \mu_1 \otimes \cdots \mu_{i-1} \otimes \mu_{i+1} \otimes \cdots \otimes \mu_d$ that
  \[
    \mu(T_{f,\eps}^{(i)})
    = \int_{\R^{d-1}}
        \int_{\R}
          \Indicator_{|y - f(x)| \leq \eps}
        \, d \mu_i (y)
      \, d \mu^{(i)} (x) ,
  \]
  where
  \[
    \int_{\R}
      \Indicator_{|y - f(x)| \leq \eps}
    \, d \mu_i (y)
    = \mu_i ( [f(x)-\eps, f(x) + \eps])
    \!=\! F_i \bigl( f(x) + \eps\bigr) - F_i \bigl(f(x) - \eps\bigr)
    \leq C \cdot (2\eps)^\alpha
    =    2^\alpha C \cdot \eps^\alpha ,
  \]
  from which we easily get $\mu\bigl(T_{f,\eps}^{(i)}\bigr) \leq 2^\alpha C \cdot \eps^\alpha$,
  as claimed.

  Measures that do not have a product structure can be tube compatible as well.
  For example, if $\mu$ is tube compatible with parameters $\alpha \in (0,1]$ and $C > 0$,
  then any measure $\nu$ of the form $d \nu = f \, d \mu$ with a bounded density function $f$
  will be tube compatible, with parameters $\alpha$ and $C \cdot \sup_{x} f(x)$.
\end{remark}

Next, we give our main approximation result for functions $\Indicator_\Omega$,
where $\Omega$ is a set with Barron class boundary.


\begin{theorem}\label{thm:BarronBoundaryApproxGuarantee}
  Let $d \in \N_{\geq 2}$, $M,N \in \N$, $B,C > 0$, and $\alpha \in (0,1]$,
  and let $\Omega \in \BarronBoundary_{B,M}(\R^d)$.

  There exists a neural network $I_N$ with $3$ hidden layers
  such that for each tube compatible measure $\mu$ with parameters $\alpha,C$, we have
  \[
    \mu
    (
      \{
        x \in \R^d : \Indicator_\Omega (x) \neq R_\varrho I_N(x)
      \}
    )
    \leq 6 C M B^\alpha \, d^{3/2} \, N^{-\alpha/2}
    .
  \]

  Moreover, $0 \leq R_\varrho I_N(x) \leq 1$ for all $x \in \R^d$
  and the architecture of $I_N$ is given by
  \[
    \mathcal{A}
    = \big(
        d, \,\, M(N+2d+2), \,\, M(4d+2), \,\, M, \,\, 1
      \big)
    .
  \]
  Thus, $I_N$ has at most $7 M (N+d)$ neurons and at most $54 d^2 M \, N$ non-zero weights.
  The weights (and biases) of $I_N$ are bounded in magnitude by
  $d (4 + R) (1 + B) + \sqrt{N} \cdot \bigl(B^{-1} + B^{-1/2}\bigr)$,
  where $R = \sup_{x \in \Omega} \| x \|_{\ell^\infty}$.
\end{theorem}

\begin{proof}
The proof will proceed in three parts.
First we construct a neural network that satisfies a certain approximation accuracy,
without going into much detail regarding the architecture of this network.
Afterwards, we analyze the network architecture, and bound the network weights.

\paragraph{Network construction and approximation bound:} \ 

\medskip{}

\noindent
\textbf{Step~1. (Construction of neural networks locally approximating boundaries)}
Let $(Q_j)_{j=1}^M$ be an associated cover of $\Omega$.
Fix $m \in \FirstN{M}$ and write $Q_m := [a,b]$.
By the assumption $\Omega\in\BarronBoundary_{B,M}(\R^d)$, there exist $i = i(m) \in \FirstN{d}$
and $\theta_m \in \{ \pm 1 \}$ as well as a function ${f_m \in \BarronApproximation_B (Q_m^i)}$
such that $\Indicator_{\Omega}(x) = \Indicator_{\theta_m x_i \leq f_m(x^{(i)})}$ for all $x \in Q_m$.
Here, we used the notation $Q_m^i := \prod_{j \neq i} [a_j, b_j]$.
With $R = \sup_{x \in \Omega} \| x \|_{\ell^\infty}$ as in the theorem statement,
note that if we replace each $Q_j$ by $\widetilde{Q_j} := Q_j \cap [-R,R]^d$,
then the family $(\widetilde{Q_j})_{j=1}^M$ is still a cover of $\Omega$ consisting of rectangles.
Furthermore, \Cref{eq:BarronApproximationSpaceRestriction}
shows that $f_m \in \BarronApproximation_C (\widetilde{Q_m^i})$,
and we clearly have $\Indicator_{\Omega}(x) = \Indicator_{\theta_m x_i \leq f_m(x^{(i)})}$
for all $x \in \widetilde{Q_m}$.
Therefore, we can assume in the following that $Q_m \subset [-R,R]^d$ for all $m \in \FirstN{M}$.

Now, by Definitions~\ref{def:BarronClassifiers} and \ref{def:BarronApproxSpace},
there exists a shallow neural network $I_N^m$ with $N$ neurons in the hidden layer
such that $\|f_m - R_\varrho I_N^m\|_{\sup} \leq \gamma N^{-1/2}$
where $\gamma := B \sqrt{d-1}$.
Furthermore, all weights and biases of $I_N^m$ are bounded by
$\sqrt{B} \cdot \bigl(6 + \vartheta(Q_m^i, q_m) + \| q_m \|_{\ell^1}\bigr)$
for some ${q_m \in Q_m^i \subset [-R,R]^{d-1}}$.

\medskip{}

\noindent
\textbf{Step 2. (Construction of neural networks approximating horizon functions)}
Set
\[
  S_m := \big\{
           x \in Q_m
           :
           f_m (x^{(i)}) \geq \theta_m \, x_i
         \big\}
  \qquad \text{and} \qquad
  T_m := \big\{
           x \in Q_m
           :
           R_\varrho I_N^m (x^{(i)}) \geq \theta_m \, x_i
         \big\} ,
\]
where $I_N^m$ is the network obtained in the previous step.
Recalling $\| f_m - R_\varrho I_N^m \|_{\sup} \leq \gamma \, N^{-1/2}$
and using the notation $S_m \triangle T_m = (S_m \setminus T_m) \cup (T_m \setminus S_m)$,
we then see
\begin{align*}
    & S_m \triangle T_m \\
    & = \big\{
          x \in Q_m
          :
          f_m(x^{(i)})
          < \theta_m x_i
          \leq R_\varrho I_N^m(x^{(i)})
        \big\} 
        \cup \big\{
               x \in Q_m
               :
               R_\varrho I_N^m(x^{(i)}) < \theta_m x_i \leq f_m(x^{(i)})
             \big\} \\
    & \subset \big\{
                x \in Q_m
                :
                - \gamma N^{-1/2} \leq f_m(x^{(i)})-\theta_m x_i < 0
              \big\} 
              \cup \big\{
                     x \in Q_m
                     :
                     0 \leq f_m(x^{(i)}) - \theta_m x_i < \gamma N^{-1/2}
                   \big\} \\
    & \subset \big\{
                x\in Q_m
                :
                |f_m(x^{(i)})-\theta_m x_i| \leq \gamma N^{-1/2}
              \big\}.
\end{align*}
Since $\mu$ is $\alpha,C$ tube compatible and since $\Indicator_{\Omega}(x) = \Indicator_{S_m}(x)$
for $x \in Q_m$, it follows that
\begin{align*}
  \mu
  (
    \{
      x\in Q_m
      :
      \Indicator_\Omega(x) \neq \Indicator_{T_m}(x)
    \}
  )
  & = \mu
      (
       \{
         x\in Q_m
         :
         \Indicator_{S_m}(x) \neq \Indicator_{T_m}(x)
       \}
      ) \\
  & = \mu(S_m \triangle T_m)
    \leq C \gamma^\alpha N^{-\alpha / 2}.
\end{align*}

Next, we define the approximate Heaviside function $H_\delta : \R \to [0,1]$ by
\begin{equation*}
  H_\delta(x)
  := \begin{cases}
      0                 & \text{ if } x \leq 0             \\
      \tfrac{x}{\delta} & \text{ if } 0 \leq x \leq \delta \\
      1                 & \text{ if } x \geq 1.
    \end{cases}
\end{equation*}
Since $H_\delta$ can be realized by a ReLU neural network
(via ${H_\delta(x) = \tfrac{1}{\delta}(\varrho(x) - \varrho(x-\delta))}$),
we next approximate the characteristic function of $T_m$ by an appropriate
approximate Heaviside function applied to $R_\varrho I_N^m(x^{(i)}) - \theta_m x_i$.

To this end, note for $\delta > 0$ and an arbitrary measurable function $\phi : \R^{d-1} \to \R$ that
\begin{align*}
  \big\{
    (t,u) \in \R^{d-1} \times \R
    :
    \Indicator_{\phi(t) \geq u} \neq H_\delta(\phi(t)-u)
  \big\}
  & = \{
        (t,u)
        :
        0 < H_\delta(\phi(t) - u) < 1
      \} \\
  & \subset \{
              (t,u)
              :
              0 \leq\phi(t) - u \leq \delta
            \} \\
  & \subset \{
              (t,u)
              :
              |\phi(t)-u| \leq \delta
            \}.
\end{align*}
Therefore, by picking $\delta = \gamma N^{-1/2}$ and using the tube compatibility of the measure
we see that
\(
  \mu
  (
    \{
      x \in Q_m
      :
      \Indicator_{T_m}(x) \neq R_\varrho J_N^m(x)
    \}
  )
  \leq C \gamma^\alpha N^{-\alpha/2} ,
  \strut
\)
where $J_N^m$ is chosen such that
${R_\varrho J_N^m (x) = H_{\gamma N^{-1/2}}\bigl(R_\varrho I_N^m(x^{(i)}) - \theta_m x_i\bigr)}$.
Note that $0 \leq R_\varrho J_N^m \leq 1$.

\medskip{}

\noindent
\textbf{Step 3. (Localization to patches)}
Next, we want to truncate each realization $R_\varrho J_N^m$ such that it is supported
on $Q_m$ and we want to realize these truncations as ReLU neural networks.
This is based on a simplified version of the argument in \cite[Lemma A.6]{petersen2018optimal}
For the sake of completeness, we recall the construction from \cite[Lemma A.6]{petersen2018optimal}.

Let ${[a,b] = \prod_{i=1}^d [a_i, b_i]}$ be a rectangle in $\R^d$,
let $0 < \eps \leq \frac{1}{2} \min_{i \in \FirstN{d}} (b_i - a_i)$ and define
$[a + \varepsilon, b - \varepsilon] := \prod_{i=1}^d [a_i + \varepsilon, b_i - \varepsilon]$.
Furthermore, define the functions $t_i : \R \to \R$, for $i \in \FirstN{d}$, by
\begin{equation*}
  t_i(u)
  := \begin{cases}
       0                          & \text{ if } u \in \R \setminus [a_i, b_i]                \\
       1                          & \text{ if } u \in [a_i + \varepsilon, b_i - \varepsilon] \\
       \tfrac{u-a_i}{\varepsilon} & \text{ if } u \in [a_i, a_i + \varepsilon]               \\
       \tfrac{b_i-u}{\varepsilon} & \text{ if } u \in [b_i - \varepsilon, b_i],
     \end{cases}
\end{equation*}
and $\eta_\varepsilon : \R^d \times \R \to \R$ by
$\eta_\varepsilon(x,y) = \varrho\big(\sum_{i=1}^d t_i(x_i) + \varrho(y) - d\big)$.
Note that for $y \in [0,1]$, if $x \in [a + \varepsilon, b - \varepsilon]$,
we have ${\eta_\varepsilon(x,y) = \varrho(\varrho(y)) = y}$;
furthermore, if ${x \in \R^d \setminus [a,b]}$,
we have $0 \leq \eta_\varepsilon(x,y) \leq \varrho(d - 1 + \varrho(y) - d) = \varrho(y - 1) = 0$.
This implies for any function $g : \R^d \to [0,1]$ that
\(
  \{
    x \in \R^d
    :
    \eta_\varepsilon(x, g(x)) \neq \Indicator_{[a,b]}(x) \cdot g(x)
  \}
  \subset [a,b] \setminus [a + \varepsilon, b - \varepsilon]
  .
\)
Note additionally that the function $\eta_\varepsilon$ can be implemented by a ReLU neural network
and that $0 \leq t_i \leq 1$, so that $0 \leq \eta_\eps (x,y) \leq \varrho(\varrho(y)) \leq 1$
for all $y \in [0,1]$, by monotonicity of the ReLU.

Returning now to the neural networks constructed in the previous step we distinguish two cases:
First, if the rectangle $Q_m$ has width along some coordinate direction $i$
less than $2\gamma N^{-1/2}$ ($Q_m$ is a ``small rectangle''),
then we see for a suitable (constant) function $g_m : \R^{d-1} \to \R$
that $Q_m \subset T_{g_m, 2\gamma N^{-1/2}}^{(i)}$ and hence
$\mu(Q_m) \leq 2^\alpha C \gamma^\alpha N^{-\alpha/2} \leq 2 d C \gamma^\alpha N^{-\alpha/2}$,
since $\alpha \leq 1$.
We thus choose $L_N^m$ to be a trivial neural network with input dimension $d+1$,
meaning $R_\varrho L_N^m (x,y) = 0$ for all $x \in \R^d$ and $y \in \R$.
We then have
\[
  \mu
  \big(
    \big\{
      x \in \R^d
      :
      \Indicator_{Q_m}(x) R_\varrho J_N^m (x) \neq R_\varrho L_N^m (x, R_\varrho J_N^m (x))
    \big\}
  \big)
  \leq \mu(Q_m)
  \leq 2 d \, C \, \gamma^\alpha \, N^{-\alpha/2} .
\]

Otherwise (if $Q_m$ is a ``large rectangle''), writing $Q_m \!=\! [a,b]$, we have
${\frac{\gamma}{\sqrt{N}} \leq \frac{1}{2} \min_{i \in \FirstN{d}} (b_i \!-\! a_i)}$, and
it is not hard to see that $[a,b] \setminus [a + \gamma N^{-1/2}, b - \gamma N^{-1/2}]$
is contained in the union of $2 d$ tubes of width $\gamma N^{-1/2}$.
Therefore, choosing $L_N^m$ such that $R_\varrho L_N^m = \eta_{\gamma N^{-1/2}}$, we obtain
\[
  \mu
  \big(
    \big\{
      x \in \R^d
      :
      \Indicator_{Q_m}(x) R_\varrho J_N^m(x) \neq R_\varrho L_N^m(x, R_\varrho J_N^m(x))
    \big\}
  \big)
  \leq 2 d C \gamma^\alpha N^{-\alpha/2} .
\]
In both cases, the function $x \mapsto R_\varrho L_N^m (x,R_\varrho J_N^m(x))$ is supported on $Q_m$
and vanishes on the boundary of $Q_m$ (due to continuity).

\medskip{}

\noindent
\textbf{Step 4. (Finishing the construction and error estimate)}
To summarize, on each rectangle $Q_m$ we have
\begin{align*}
  & \mu
    (
      \{
        x\in \R^d
        :
        \Indicator_{\Omega\cap Q_m}(x) \neq R_\varrho L_N^m (x, R_\varrho J_N^m(x))
      \}
    ) \\
  & \leq \mu
         (
          \{
            x\in\R^d
            :
            R_\varrho L_N^m (x, R_\varrho J_N^m(x)) \neq \Indicator_{Q_m}(x) R_\varrho J_N^m(x)
          \}
         ) \\
  & \quad + \mu
            (
              \{
                x \in Q_m
                :
                R_\varrho J_N^m(x) \neq \Indicator_{T_m}(x)
              \}
            )
            + \mu
              (
               \{
                  x \in Q_m
                  :
                  \Indicator_{T_m}(x) \neq \Indicator_\Omega(x)
               \}
              ) \\
  & \leq 2 d C \gamma^\alpha N^{-\alpha/2}
         + C \gamma^\alpha N^{-\alpha/2}
         + C \gamma^\alpha N^{-\alpha/2} \\
  & =    2 (d+1) C \gamma^\alpha N^{-\alpha/2}.
\end{align*}
Now, defining the neural network $I_N$ such that
$R_\varrho I_N(x) := \sum_{m=1}^M R_\varrho L_N^m (x, R_\varrho J_N^m(x))$, we obtain because of
$\Indicator_{\Omega} = \sum_{m=1}^M \Indicator_{\Omega \cap Q_m}$ (almost everywhere) that
\[
  \mu
  \big(
    \big\{
      x \in\R^d
      :
      \Indicator_\Omega(x) \neq R_\varrho I_N(x)
    \big\}
  \big)
  \leq 2 M (d+1) C \gamma^\alpha N^{-\alpha/2}
  = 2 (d+1) (d-1)^{\alpha/2} C M B^\alpha N^{-\alpha/2}
  .
\]
To simplify the estimate, using that $\alpha \leq 1$, we see
$(d+1) (d-1)^{\alpha/2} \leq (d+1)^{3/2} \leq (2d)^{3/2}$, since $d \geq 2$.
Finally, note that $2^{1+3/2}=2^{5/2} < 6$.
Combining these estimates we see that
\(
  2 (d+1) (d-1)^{\alpha/2} C M B^\alpha N^{-\alpha/2}
  \leq 6 C M B^\alpha \, d^{3/2} N^{-\alpha/2}
  .
\)

Additionally, recall from above that $0 \leq R_\varrho J_N^m \leq 1$ for every $m \in \FirstN{M}$.
As seen in Step~3, this implies that $\zeta_m(x) := R_\varrho L_N^m(x, R_\varrho J_N^m(x))$
satisfies $0 \leq \zeta_m(x) \leq 1$ for all $x \in \R^d$.
Since each $\zeta_m$ is supported on $Q_m$ and vanishes on the boundary of $Q_m$,
and since the rectangles $Q_m$ have disjoint interiors,
this implies that $0 \leq R_\varrho I_N \leq 1$ as well.

\paragraph{The architecture:} \ \smallskip

Now let us examine the architecture of each $L_N^m$ in more detail.
For each rectangle $Q_m$, the flowchart of computations performed by each $L_N^m$
can be visually represented as in Figure~\ref{fig:cartoonOfNEtwork}.

\begin{figure}[htb]
\centering
\begin{tikzpicture}
\begin{scope}[every node/.style={rectangle,thick,draw}]
  \node (L00) at (0,0) {$x$};

  \node (L1-1) at (-5,2) {$x$};
  \node (L10) at (0,2) {$\pi^m(x)$};
  \node (L11) at (5,2) {$\widetilde{\pi}^m(x)$};

  \node (L2-1) at (-5,4) {$x$};
  \node (L20) at (0,4) {$R_\varrho I_N^m (\pi^m(x))$};
  \node (L21) at (5,4) {$-\theta_m\widetilde{\pi}^m(x)$};

  \node (L3-1) at (-5,6) {$x$};
  \node (L31) at (2.5,6) {$([R_\varrho I_N^m] \circ \pi^m - \theta_m\widetilde{\pi}^m)(x)$};

  \node (L4-1) at (-5,8) {$x$};
  \node (L41) at (2.5,8) {$R_\varrho J_N^m(x) = H_\delta(([R_\varrho I_N^m] \circ \pi^m - \theta_m\widetilde{\pi}^m)(x))$};

  \node (L50) at (0,10) {$(x,J_N^m(x))$};

  \node (L60) at (0,12) {$L_N^m(x,J_N^m(x))$};
\end{scope}

\begin{scope}[every node/.style={fill=white,rectangle},
              every edge/.style={draw=blue, thick}]
  \path [->] (L00) edge node {$id$} (L1-1);
  \path [->] (L00) edge node {$\pi^m$} (L10);
  \path [->] (L00) edge node {$\widetilde{\pi}^m$} (L11);

  \path [->] (L1-1) edge node {$id$} (L2-1);
  \path [->] (L10) edge node {$I_N^m$} (L20);
  \path [->] (L11) edge node {$t \mapsto -\theta_m t$} (L21);

  \path [->] (L2-1) edge node {$id$} (L3-1);
  \path [->] (L20) edge node {$+$} (L31);
  \path [->] (L21) edge node {$+$} (L31);

  \path [->] (L3-1) edge node {$id$} (L4-1);
  \path [->] (L31) edge node {$H_\delta$} (L41);

  \path [->] (L4-1) edge node {$\hookrightarrow$} (L50);
  \path [->] (L41) edge node {$\hookrightarrow$} (L50);

  \path [->] (L50) edge node {$L_N^m$} (L60);
\end{scope}
\draw [red,dashed] (-7,0.75)--(7,0.75) node[anchor=north west] {Input};
\draw [red,dashed] (-7,4.75)--(7,4.75) node[anchor=north west] {Layer 1};
\draw [red,dashed] (-7,8.75)--(7,8.75) node[anchor=north west] {Layer 2};
\end{tikzpicture}
\caption{Visualization of the neural network $L_N^m$ for the case of a ``large'' rectangle $Q_m$.}
\label{fig:cartoonOfNEtwork}
\end{figure}
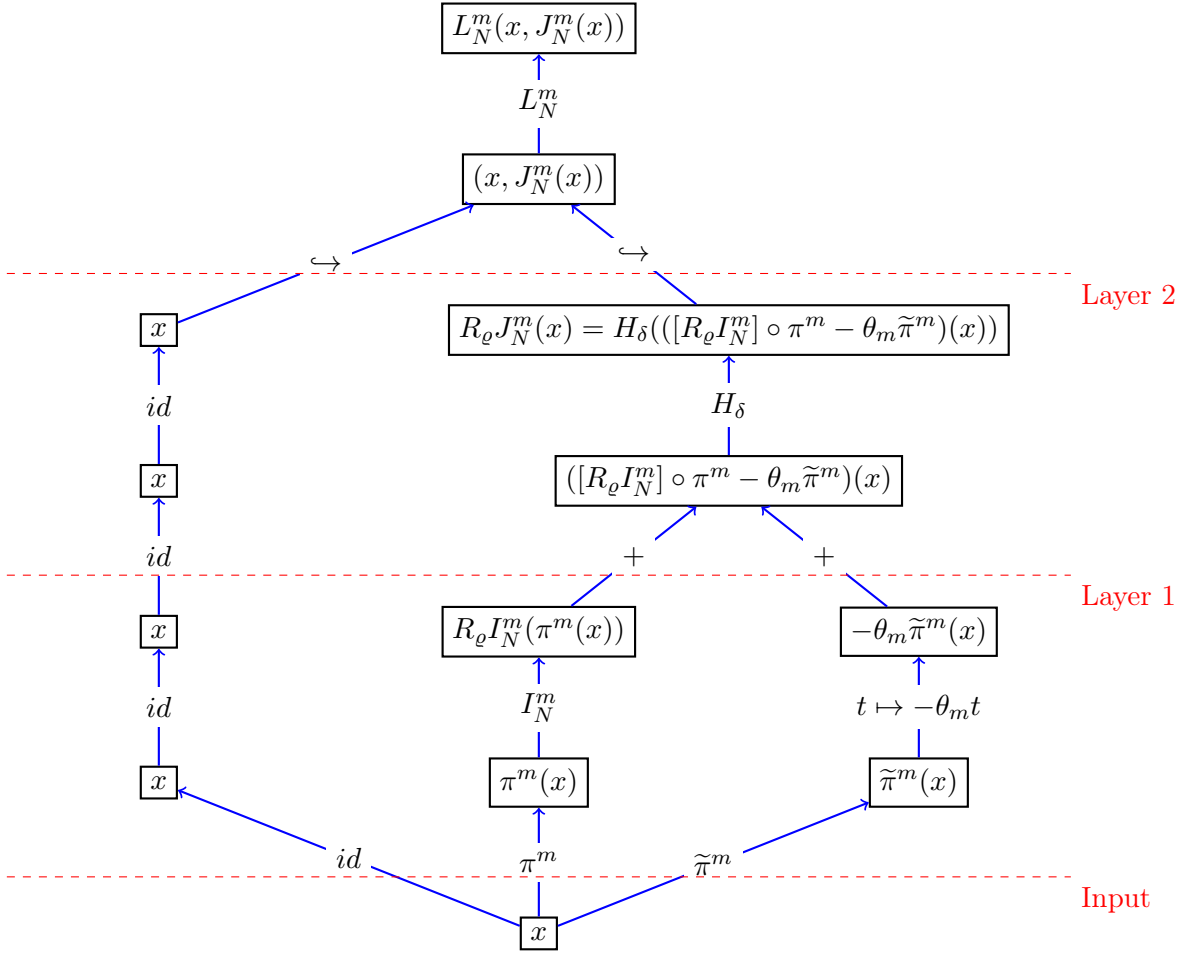

In the following, we explicitly describe each of the layers of the network computing $L_N^m$;
we then describe how these networks are combined to obtain $I_N$.

\medskip{}

\noindent
\textit{Inputs}.
The input layer with $d$ neurons corresponding to the $d$ coordinates of an input $x \in \R^d$.

\medskip{}

\noindent
\textit{Layer 1}. This layer will contain $N + 2d + 2$ neurons split into $3$ categories:
\begin{itemize}
  \item $2d$ neurons computing $\varrho(x_i)$ and $\varrho(-x_i)$
        respectively for each $i \in \FirstN{d}$.

  \item $N$ neurons corresponding to the neurons in the hidden layer of the networks $I_N^m$.
        Explicitly, writing
        $R_\varrho I_N^m(x) = D + \sum_{k=1}^N C_k \, \varrho(B_k + \langle A_k, x \rangle)$
        with $D, B_k, C_k \in \R$ and $A_k \in \R^{d-1}$ for $k \in \FirstN{N}$,
        the $k$-th of these neurons will compute
        $\phi_k(x) = \varrho\bigl(B_k + \langle (\pi^m)^T A_k, x \rangle\bigr)$, 
        where $\pi^m$ is the projection that sends $x$ to $x^{(i)}$
        (with $i = i(m)$), viewed as a $(d-1) \times d$ matrix. 

  \item $2$ neurons computing $\varrho(\pm\theta_m \widetilde{\pi}^m(x))$, respectively,
        where $\widetilde{\pi}^m$ is the projection that sends $x$ to $x_i$ (where $i = i(m)$),
        viewed as a $1 \times d$ matrix.
\end{itemize}

\medskip{}

\noindent
\textit{Layer 2}. This layer will contain $4d+2$ neurons split into $2$ categories:
\begin{itemize}
  \item $4$ neurons for each coordinate $i\in\FirstN{d}$ computing the building blocks
        for the $t_i$ functions in Step 3:
        $t_i^1(u_i) = \varrho(u_i - a_i)$,
        $t_i^2(u_i) = \varrho(u_i - a_i - \varepsilon)$,
        $t_i^3(u_i) = \varrho(u_i - b_i + \varepsilon)$
        and $t_i^4(u_i) = \varrho(u_i - b_i)$, where $u_i := \varrho(x_i) - \varrho(-x_i) = x_i$.
        Note that $t_i(u_i) = \tfrac{t_i^1 - t_i^2 - t_i^3 + t_i^4}{\varepsilon}(u_i)$.
        Furthermore, recall that we chose $\eps = \gamma N^{-1/2}$.

  \item $2$ neurons computing the parts of the approximate Heaviside function $H_\delta$,
        computing, respectively,
        \[
          \psi_1(x)
          := \varrho
             \Big(
               D
               + \sum_{k=1}^{N}
                   C_k \phi_k(x)
               - \varrho\bigl(\theta_m \widetilde{\pi}^m(x)\bigr)
               + \varrho\bigl(-\theta_m \widetilde{\pi}^m(x)\bigr)
             \Big)
        \]
        and
        \[
          \psi_2(x)
          := \varrho
             \Big(
               D
               + \sum_{k=1}^{N}
                   C_k \phi_k(x)
               - \varrho\bigl(\theta_m \widetilde{\pi}^m(x)\bigr)
               + \varrho\bigl(-\theta_m \widetilde{\pi}^m(x)\bigr)
               - \delta
             \Big),
        \]
        where we recall from above that
        $R_\varrho I_N^m (\pi^m (x)) = D + \sum_{k=1}^{N} C_k \phi_k(x)$
        and $\delta = \gamma N^{-1/2}$.
        Therefore,
        \(
          \frac{1}{\delta} (\psi_1(x) - \psi_2(x))
          = H_\delta
            \bigl(
              R_\varrho I_N^m (\pi^m (x))
              - \theta_m \widetilde{\pi}^m (x)
            \bigr)
          = R_\varrho J_N^m (x) ;
        \)
        in particular, $\psi_1(x) - \psi_2(x) \geq 0$.
\end{itemize}

\medskip{}

\noindent
\textit{Layer 3}.
This layer will have a single neuron, either computing the zero function
(in the case of a ``small rectangle'' $Q_m$), or (in the case of a ``large rectangle'') computing
\[
  \eta_\eps (x, R_\varrho J_N^m(x))
  = R_\varrho L_N^m(x,R_\varrho J_N^m(x))
  = \varrho
    \bigg(
    \tfrac{1}{\varepsilon}
      \sum_{i=1}^d
        \bigl(t_i^1 - t_i^2 - t_i^3 + t_i^4\bigr)(x_i)
    + \tfrac{1}{\delta} \bigl(\psi_1(x) - \psi_2(x)\bigr)
    - d
    \bigg) .
\]
We used here that $(\psi_1 - \psi_2)(x) \geq 0$, so the difference is invariant under $\varrho$.

Now, the full network $I_N$ can be realized with one more layer (the output layer),
so that $R_\varrho I_N(x) = \sum_{m=1}^M R_\varrho L_N^m(x, R_\varrho J_N^m (x))$.

Thus, $I_N$ can be realized by a ReLU neural network with $3$ hidden layers,
architecture ${\CalA = \big(d, \, M(N+2d+2), \, M(4d+2), \, M, \, 1\big)}$,
and $d + 1 + M (N + 6d + 5)\leq 7 M (N + d)$ neurons.

Now let us estimate the number of non-zero weights of $I_N$ which we will denote by $W(I_N)$.
An immediate bound can be found by taking the product of the number of neurons
on every pair of consecutive layers in the $L_N^m$ networks, summing up over the layers,
multiplying by $M$, adding $M$ to account for the weights of the final output layer,
and finally adding the total number of non-input neurons to account for the biases.
We thus see
\[
  W(I_N)
  \leq M \cdot
      \big(
        d(N + 2d + 2)
        + (N + 2d + 2)(4d + 2)
        + (4d + 2) \cdot 1
      \big)
      + M
      + M N
      + 6 M d
      + 5 M
      + 1
  ,
\]
so that a rough estimate shows $W(I_N) \leq 54 M d^2 \, N$.

\paragraph{Bounding the magnitude of the weights and biases:} \ \smallskip

Let us now acquire an upper bound for the absolute value of the weights and biases of $I_N$.
Note first of all that for the networks $I_N^m$ we have two cases depending on the size
of the corresponding rectangle $Q_m = \prod_{i=1}^d [a_i,b_i]$:
\begin{itemize}
  \item If $\min_i(b_i-a_i) < 2 \gamma N^{-1/2}$,
        we can set all weights of the ``subnetwork'' corresponding to the rectangle $Q_m$
        to be zero.

  \item If $\min_i (b_i - a_i) \geq 2 \gamma N^{-1/2}$, then by \Cref{rem:explanationofVartheta},
        we have $\vartheta(Q_m, q_m) \leq \gamma^{-1} N^{1/2}$.
        Since furthermore $\| q_m \|_{\ell^1} \leq (d-1) R$, our choice of $I_N^m$
        in Step~1 ensures that the weights and biases of $I_N^m$ are bounded by
        \(
          \sqrt{B} \cdot (6 + \vartheta(Q_m, q_m) + \| q_m \|_{\ell^1})
          \leq \sqrt{B} \cdot (6 + \gamma^{-1} N^{1/2} + d R) .
        \)
\end{itemize}

In either case, we see that the weights and biases on the first layer are bounded in absolute
value by $1 + \sqrt{B} \cdot (6 + \gamma^{-1} N^{1/2} + d R)$.

For the second layer, the weights corresponding to the first $4d$ neurons are bounded
by $1 + \eps + R$ and for the last $2$ neurons again by 
$1 + \sqrt{B} \cdot (6 + \gamma^{-1} N^{1/2} + d R)$.
Finally for the third layer, the weights and biases are bounded by
$\max(\tfrac{1}{\varepsilon}, \tfrac{1}{\delta}, d) \leq d + \gamma^{-1} N^{1/2}$.

In conclusion, the weights of $I_N$ will have magnitudes bounded by
\begin{align*}
  & \max
    \big\{
      1 + 6 \sqrt{B} + \sqrt{B} \gamma^{-1} N^{1/2} + \sqrt{B} d R, \,\,
      1 + \eps + R, \,\,
      d + \gamma^{-1} N^{1/2}
    \big\} \\
  & \leq d (4 + R) (1 + B) + \sqrt{N} \cdot \bigl(B^{-1} + B^{-1/2}\bigr) .
\end{align*}
Here, we used that $d \geq 2$, combined with several elementary estimates including the bound
$\sqrt{B} \leq 1 + B$.
\end{proof}

%% file: 4-LowerBounds.tex

In this section, we present a lower bound on the achievable minimax rate
for approximating functions $f \in \BarronBoundary_{B,M}(\R^d)$
in $L^1$ with respect to the Lebesgue measure on $[-1,1]^d$.
In fact, we show that the approximation rate provided by \Cref{thm:BarronBoundaryApproxGuarantee}
is \emph{almost optimal} even when only horizon functions
with boundary from the Fourier-analytic Barron space are considered.
More precisely, we will see in \Cref{thm:QuantizedLowerBound} below that
neural networks with $W$ weights cannot obtain a better approximation error
than $\CalO(W^{-\frac{1}{2} - \frac{1}{d-1}})$ over the class of horizon functions
with boundary from the Fourier-analytic Barron space; in contrast, our upper bound
from \Cref{thm:BarronBoundaryApproxGuarantee} guarantees that an approximation error of
$\CalO(W^{-1/2})$ is achievable.
Thus, even though the two rates of approximation do not precisely agree,
the difference between them vanishes for increasing input dimension $d \to \infty$;
therefore, we speak of \emph{almost optimality}.

Since the arguments in this section are heavily based on the Fourier transform,
we start by fixing its normalization.
Concretely, for $f \in L^1(\R^d)$, we define
\[
  \Fourier f(\xi)
  = \widehat{f} (\xi)
  = (2\pi)^{-d/2}
    \int_{\R^d}
      f(x) e^{-i \langle x, \xi \rangle}
    \, d x ,
\]
so that the inverse Fourier transform is given by
$\Fourier^{-1} f (x) =\widehat{f}(-x)$; see e.g.\ \mbox{\cite[Section~4.3.1]{EvansPDE}}.

Our first step towards lower bounds is to relate the covering numbers of certain
sets of horizon functions to covering numbers of certain subsets of Besov spaces.
To fix the terminology, recall that if $\Theta$ is a subset of a normed vector space $X$,
then a set $\emptyset \neq M \subset X$ is called an $\eps$-net for $\Theta$ (in $X$),
if $\sup_{x \in \Theta} \inf_{m \in M} \| x - m \|_X \leq \eps$.

\begin{proposition}\label{prop:HorizonFunctionsCoveringNumbers}
  Let $d \in \N_{\geq 2}$ and $\Omega := (-1,1)^{d-1}$.
  Given a function $f : \Omega \to \R$, define the associated \emph{horizon function} as
  \[
    H_f : \quad
    (-1,1)^{d-1} \times (-1,1) \to \{ 0,1 \}, \quad
    (x,y) \mapsto \Indicator_{y \leq f(x)} .
  \]

  For each $s > \frac{d+1}{2}$ and $C > 0$, there is a constant $\lambda = \lambda(d,s,C) > 0$
  with the following property:
  If $\eps > 0$ and if $M \subset L^1 \bigl( (-1,1)^d \bigr)$ is an $\eps$-net
  (in $L^1 ( (-1,1)^d )$) for
  \[
    \HorizonFunctions (\Barron_C)
    := \bigl\{ H_f \,\,\colon\,\, f \in \Barron_C ( [-1,1]^{d-1} ) \bigr\} ,
  \]
  then there exists a set $M' \subset B_0^{1,\infty}(\Omega)$ satisfying $|M'| \leq |M|$
  which is a $\lambda\eps$-net (in $B_0^{1,\infty}(\Omega)$) for
  \(
    \CalG
    := \bigl\{
         f \in B_s^{2,2} (\Omega)
         \colon
         \| f \|_{B_s^{2,2}} \leq 1
       \bigr\} .
  \)
\end{proposition}

\begin{remark}
  Here, for an open set $\Omega \subset \R^{d-1}$, we use the definition of the Besov spaces
  $B_s^{p,q}(\Omega)$ as in \cite[Section~2.5.1]{EdmundsTriebelEntropyNumbers}, that is,
  \(
    B_s^{p,q} (\Omega)
    = \bigl\{ f|_\Omega \,\,\colon f \in B_s^{p,q}(\R^{d-1}) \bigr\}
  \)
  with norm
  \[
    \| f \|_{B_s^{p,q}(\Omega)}
    := \inf
       \big\{
         \| g \|_{B_s^{p,q}(\R^{d-1})}
         \colon
         g \in B_s^{p,q}(\R^{d-1}) \text{ and } f = g|_{\Omega}
       \big\} .
  \]
\end{remark}

\begin{proof}
We divide the proof into four steps.

\medskip{}

\noindent
\textbf{Step 1:}
For completeness, we prove the well-known embedding
$L^1(\Omega) \hookrightarrow B_0^{1,\infty}(\Omega)$.
Clearly, it is enough to prove $L^1(\R^{d-1}) \hookrightarrow B_0^{1,\infty}(\R^{d-1})$.
To this end, recall from \cite[Section~2.2.1]{EdmundsTriebelEntropyNumbers}
that the norm on $B_0^{1,\infty}(\R^{d-1})$ is given by
\[
  \| f \|_{B_0^{1,\infty}}
  = \sup_{j \in \N_0}
      \Big\| \Fourier^{-1} \bigl[ \varphi_j \, \widehat{f} \,\, \bigr] \Big\|_{L^1} ,
\]
where $\varphi_0, \varphi \in \Schwartz(\R^{d-1})$ are suitably chosen
and $\varphi_k (\xi) = \varphi_1(2^{-k+1} \xi)$ for $k \in \N$.
Note that
\(
  \| \Fourier^{-1} \varphi_k \|_{L^1}
  = \big\| 2^{(k-1) (d-1)} (\Fourier^{-1} \varphi_1)(2^{k-1} \cdot) \big\|_{L^1}
  = \| \Fourier^{-1} \varphi_1 \|_{L^1} ,
\)
whence $\| \Fourier^{-1} \varphi_k \|_{L^1} \leq C_1 < \infty$ for all $k \in \N_0$.
By Young's inequality for convolutions, this implies
\[
  \big\| \Fourier^{-1} \bigl[ \varphi_j \, \widehat{f} \,\, \bigr] \big\|_{L^1}
  = C_0 \, \big\| (\Fourier^{-1} \varphi_j ) \ast f \big\|_{L^1}
  \leq C_1' \cdot \| f \|_{L^1} ,
\]
so that $\| f \|_{B_0^{1,\infty}} \leq C_1 ' \cdot \| f \|_{L^1}$ for $f \in L^1(\R^{d-1})$,
with $C_1 ' = C_1 ' (d) > 0$.

\medskip{}

\noindent
\textbf{Step 2:}
We show existence of $c = c(s,d,C) > 0$ such that every $f \in \CalG$ satisfies
${\| c f \|_{\sup} \leq \frac{1}{4}}$ and $c f \in \Barron_C([-1,1]^{d-1})$.
We remark that this inclusion was (up to minute differences) already observed in
\cite[Example~15 on Page~941]{BarronUniversalApproximation}
and \cite[Theorem~3.1]{wojtowytsch2020representation}.
We provide the proof here for the sake of completeness.

To this end, we first prove
$\int_{\R^{d-1}} (1 + |\xi|) \, |\widehat{f}(\xi)| \, d \xi \leq C_2 \cdot \| f \|_{B_s^{2,2}}$
for all $f \in B_s^{2,2}(\R^{d-1})$, for a suitable constant $C_2 = C_2(s, d) > 0$.
First, recall from \cite[Sections~2.2.2, 2.3.2, and 2.5.6]{TriebelFunctionSpaces1}
the well-known identity $B_s^{2,2}(\R^{d-1}) = F_s^{2,2}(\R^{d-1}) = H^{s,2}(\R^{d-1})$,
where the norm on the Sobolev space $H^{s,2}(\R^{d-1})$ is given by
$\| f \|_{H^{s,2}}^2 = \int_{\R^{d-1}} (1 + |\xi|^2)^s \, |\widehat{f} (\xi)|^2 \, d \xi$.
Using the Cauchy-Schwarz inequality, we therefore see for $f \in B_s^{2,2}(\R^{d-1})$ that
\begin{align*}
  \int_{\R^{d-1}}
    \bigl(1 + |\xi|\bigr) \, |\widehat{f} (\xi)|
  \, d \xi
  & \lesssim \int_{\R^{d-1}}
               \bigl(1 + |\xi|^2\bigr)^{\frac{1-s}{2}} \,
               \bigl(1 + |\xi|^2\bigr)^{s/2} \,
               |\widehat{f}(\xi)|
             \, d \xi \\
  & \leq \Big(
           \int_{\R^{d-1}}
             \bigl(1 + |\xi|^2\bigr)^{1 - s}
           \, d \xi
         \Big)^{1/2}
         \Big(
           \int_{\R^{d-1}}
             \bigl(1 + |\xi|^2\bigr)^{s} \, |\widehat{f}(\xi)|^2
           \, d \xi
         \Big)^{1/2} \\
  & \lesssim \| f \|_{H^{s,2}} \lesssim \| f \|_{B_s^{2,2}} .
\end{align*}
Here, we used that $2 \cdot (1 - s) < - (d-1)$,
so that $\int_{\R^{d-1}} (1 + |\xi|^2)^{1-s} \, d \xi < \infty$.

Now, by definition of $\calG$ and of $B_s^{2,2}(\Omega)$, for each $f \in \calG$,
there is $F \in B_s^{2,2}(\R^{d-1})$ with $f = F|_{\Omega}$ and $\| F \|_{B_s^{2,2}} \leq 2$.
As seen above, this entails
$\int_{\R^{d-1}} (1 + |\xi|) \, |\widehat{F}(\xi)| \, d \xi \leq 2 C_2$.
On the one hand, this implies by Fourier inversion for all $x \in \Omega$ that
\[
  |f(x)|
  = |F(x)|
  = |(\Fourier^{-1} \widehat{F})(x)|
  \leq (2\pi)^{-(d-1)/2} \int_{\R^{d-1}} |\widehat{F}(\xi)| \, d \xi
  \leq (2\pi)^{-(d-1)/2} 2 C_2
  \leq 2 C_2 .
\]
On the other hand,
\(
  f(x)
  = F(x)
  = \Fourier^{-1} \widehat{F} (x)
  =\!
    \int_{\R^{d-1}}
      e^{i \langle x,\xi \rangle}
      (2\pi)^{-\frac{d-1}{2}} \widehat{F}(\xi)
    \, d \xi
\)
for ${x \in [-1,1]^{d-1}}$ and, in the notation of \Cref{def:BarronClassFunction},
\({
  |\xi|_{[-1,1]^{d-1}, 0}
  = \sup_{x \in [-1,1]^{d-1}}
      |\langle \xi, x \rangle|
  = \| \xi \|_{\ell^1}
  \lesssim |\xi| ,
}\)
meaning
\(
  \int_{\R^{d-1}}
    |\xi|_{[-1,1]^{d-1},0} \cdot (2\pi)^{-(d-1)/2} |\widehat{F}(\xi)|
  \, d \xi
  \lesssim C_2
  \lesssim 1.
\)
By combining these observations, it is easy to see $\| c f \|_{\sup} \leq \frac{1}{4}$
and $c f \in \Barron_C ([-1,1]^{d-1})$, for $c = c(s,d,C) > 0$ small enough.


\medskip{}

\noindent
\textbf{Step 3:}
We show that
\[
  \| H_f - H_g \|_{L^1 ( (-1,1)^d )}
  \geq \| f - g \|_{L^1 (\Omega)}
  \qquad \forall \, f,g : \Omega \to \bigl[-\tfrac{1}{2}, \tfrac{1}{2}\bigr] \text{ measurable} .
\]
To see this, first note by Fubini's theorem that
\begin{align*}
  \| H_f - H_g \|_{L^1 ( (-1,1)^d )}
  & = \int_{(-1,1)^{d-1}}
        \int_{-1}^1
          |\Indicator_{y \leq f(x)} - \Indicator_{y \leq g(x)}|
        \, d y
      \, d x \\
  & \overset{(\ast)}{=} \int_{(-1,1)^{d-1}} |f(x) - g(x)| \, d x
    = \| f - g \|_{L^1 (\Omega)} .
\end{align*}
Here, the step marked with $(\ast)$ used that
\(
  \int_{-1}^1
    |\Indicator_{y \leq f(x)} - \Indicator_{y \leq g(x)}|
  \, d y
  = |f(x) - g(x)| ,
\)
which is trivial if $f(x) = g(x)$.
Otherwise, if $f(x) > g(x)$, then
\[
  \bigl|\Indicator_{y \leq f(x)} - \Indicator_{y \leq g(x)}\bigr|
  = \Indicator_{(g(x), f(x))} (y) ,
\]
which implies the claimed estimate.
Here, we implicitly used that $(g(x), f(x)) \subset [-1,1]$,
since $f(x), g(x) \in [-\frac{1}{2}, \frac{1}{2}]$.
For $f(x) < g(x)$, one can argue similarly.

\medskip{}

\noindent
\textbf{Step 4:} We complete the proof.
To this end, write $N := |M|$ and ${M = \{ G_1, \dots, G_N \}}$.
With $c > 0$ as in Step~2, for each $i \in \FirstN{N}$, choose $f_i \in \CalG$ with
\begin{equation}
  \| H_{c f_i} - G_i \|_{L^1}
  \leq \eps + \inf_{f \in \CalG} \| H_{c f} - G_i  \|_{L^1} .
  \label{eq:HorizonFunctionsCoveringNumbersAlmostOptimality}
\end{equation}
We claim that $M' := \{ f_1, \dots, f_N \} \subset L^1(\Omega) \subset B_0^{1,\infty}(\Omega)$
is a $\lambda \eps$-net for $\CalG$ (in $B_0^{1,\infty}(\Omega)$), for a suitable choice of
$\lambda = \lambda(d,s,C) > 0$.

To see this, let $f \in \CalG$ be arbitrary.
By Step~2, we have $c f \in \Barron_C ([-1,1]^{d-1})$
and hence $H_{c f} \in \HorizonFunctions(\Barron_C)$.
Since $M$ is an $\eps$-net for $\HorizonFunctions(\Barron_C)$ (in $L^1 ( (-1,1)^d)$),
this implies that there exists $i \in \FirstN{N}$ with $\| H_{c f} - G_i \|_{L^1} \leq 2 \eps$.
Since $f, f_i \in \CalG$ and hence $\| c \, f \|_{\sup}, \| c \, f_i \|_{\sup} \leq \frac{1}{2}$
by Step~2, the estimates from Steps~1 and 3 show
\begin{align*}
  \| f - f_i \|_{B_0^{1,\infty}(\Omega)}
  & \leq C_1' \, \| f - f_i \|_{L^1(\Omega)}
    =    \frac{C_1'}{c} \,
         \big\| c \, f - c \, f_i \big\|_{L^1(\Omega)}
    \leq \frac{C_1'}{c} \,
         \big\| H_{c f} - H_{c f_i} \big\|_{L^1 ( (-1,1)^d)} \\
  & \leq \frac{C_1'}{c}
         \big( \| H_{c f} - G_i \|_{L^1} + \| G_i - H_{c f_i} \|_{L^1} \big) \\
  & \overset{(\ast\ast)}{\leq}
         \frac{C_1'}{c}
         \big( \| H_{c f} - G_i \|_{L^1} + \eps + \| H_{c f} - G_i \|_{L^1} \big)
    \leq \frac{5 C_1'}{c} \eps .
\end{align*}
Here, the step marked with $(\ast\ast)$ is justified
by \Cref{eq:HorizonFunctionsCoveringNumbersAlmostOptimality}.
\end{proof}

Based on \Cref{prop:HorizonFunctionsCoveringNumbers}, we can now prove our first lower bound
for the approximation of Barron-class horizon functions.
This result uses the notion of \emph{$(\tau,\eps)$-quantized networks}
introduced in \mbox{\cite[Definition~2.9]{petersen2018optimal}}.
Precisely, given $\tau \in \N$ and $\eps \in (0,\frac{1}{2})$, we say that a network $\Phi$
is $(\tau,\eps)$-quantized, if all the weights and biases of $\Phi$ belong to the set
$[-\eps^{-\tau}, \eps^{-\tau}] \cap 2^{-\tau \lceil \log_2 (1/\eps) \rceil} \Z$.
Similar notions of quantized networks have been employed in
\cite{boelcskei2019optimal, ElbraechterDNNApproximationTheory}
in the context of lower bounds on approximation rates.

\begin{theorem}\label{thm:QuantizedLowerBound}
  Let $d \in \N_{\geq 2}$, $\tau \in \N$, and $C, \sigma > 0$.
  With notation as in \Cref{prop:HorizonFunctionsCoveringNumbers}, assume that
  there are $C_1,C_2 > 0$ and a null-sequence $(\eps_n)_{n \in \N} \subset (0,\infty)$
  such that for every ${f \in \Barron_C ([-1,1]^{d-1})}$ and $n \in \N$,
  there is a network $\Phi$ with $d$-dimensional input and $1$-dimensional output,
  with $(\tau, \eps_n)$-quantized weights, and such that
  \[
    \| H_f - R_\varrho \Phi \|_{L^1 ( (-1,1)^d)} \leq C_1 \, \eps_n
    \qquad \text{and} \qquad
    W(\Phi) \leq C_2 \cdot \eps_n^{-\sigma} .
  \]
  Then $\frac{1}{\sigma} \leq \frac{1}{2} + \frac{1}{d-1}$.
\end{theorem}

\begin{proof}
  Let $\Omega_0 := B_{1/2}(0) = \{ x \in \R^{d-1} \colon |x| < 1/2 \}$, noting that this is a
  bounded $C^\infty$-domain in the sense of \cite[Section~3.2.1]{TriebelFunctionSpaces1},
  and that $\Omega_0 \subset \Omega = (-1,1)^{d-1}$.
  Let us fix $s > \frac{d+1}{2}$ for the moment,
  and define $A := B_s^{2,2}(\Omega_0)$ and $B := B_0^{1,\infty}(\Omega_0)$.
  The proof is based on existing \emph{entropy bounds} for the embedding $A \hookrightarrow B$.
  More precisely, writing $U_A := \{ x \in A \colon \| x \|_A \leq 1 \}$ (and similarly for $U_B$),
  the $k$-th entropy number of this embedding is defined as
  \[
    e_k
    := \inf \Big\{
              \eps > 0
              \quad\colon\quad
              \exists f_1, \dots, f_{2^{k-1}} \in B \text{ such that }
                U_A \subset \bigcup_{i=1}^{2^{k-1}} (f_i + \eps \, U_B)
            \Big\}
    \,\,;
  \]
  see \cite[Definition~1 in Section~1.3.1]{EdmundsTriebelEntropyNumbers}.
  Furthermore, \cite[Theorem~1 in Section~3.3.3]{EdmundsTriebelEntropyNumbers}
  shows that there is a constant $c = c(d,s) > 0$ satisfying
  \[
    e_k \geq c \cdot k^{-s/(d-1)}
    \qquad \forall \, k \in \N .
  \]

  Given a neural network $\Phi$, let us write $d_{\mathrm{in}}(\Phi)$ and $d_{\mathrm{out}}(\Phi)$
  for the input- and output-dimension of $\Phi$, respectively.
  Fix $n \in \N$ with $\eps := \eps_n < 1/2$, and define
  \[
    M_n := \big\{
             R_\varrho \Phi
             \,\,\colon
             \Phi \text{ is $(\tau,\eps)$-quantized NN with }
             d_{\mathrm{in}}(\Phi) = d,
             d_{\mathrm{out}}(\Phi) = 1,
             \text{ and } W(\Phi) \leq C_2 \cdot \eps^{-\sigma}
           \big\} .
  \]
  Note that $\lceil \log_2(1/\eps) \rceil \leq 1 + \log_2(1/\eps) \leq 2 \log_2(1/\eps)$,
  whence $2^{\tau \lceil \log_2(1/\eps) \rceil} \leq 2^{2\tau \log_2(1/\eps)} = \eps^{-2\tau}$.
  Furthermore, note for arbitrary $a,b > 0$ that
  \[
    \big| [-a,a] \cap b \Z \big|
    = \bigl| [-b^{-1}a, b^{-1} a] \cap \Z \bigr|
    \leq 1 + 2 b^{-1} a
    \,\,,
  \]
  which shows
  \[
    \bigl| [-\eps^{-\tau}, \eps^{-\tau}] \cap 2^{-\tau \lceil \log_2 (1/\eps) \rceil} \Z \bigr|
    \leq 1 + 2 \, \eps^{-\tau} \, 2^{\tau \lceil \log_2(1/\eps) \rceil}
    \leq 1 + 2 \eps^{-3\tau}
    \leq \eps^{-5\tau} ,
  \]
  and hence
  \(
    \bigl|
      [-\eps^{-\tau}, \eps^{-\tau}]
      \cap 2^{-\tau \lceil \log_2 (1/\eps) \rceil} \Z
    \bigr|
    \leq 2^{K}
  \)
  for $K := \lceil \log_2(\eps^{-5\tau}) \rceil \leq 6 \tau \log_2(1/\eps)$.
  Therefore, an application of \cite[Lemma~B.4]{petersen2018optimal} shows that
  there is a constant $C_3 = C_3(d) \in \N$ satisfying
  \[
    |M_n|
    \leq 2^{C_3 C_2 \eps^{-\sigma}
         \cdot (\lceil \log_2(C_2 \, \eps^{-\sigma}) \rceil + 6 \tau \log_2(1/\eps) )}
    \leq 2^{C_4 \eps^{-\sigma} \log_2(1/\eps)} ,
  \]
  with $C_4 = C_4(C_2,d,\tau,\sigma) > 0$.

  By assumption of the theorem to be proven and because of our choice $\eps = \eps_n$,
  we see with notation as in \Cref{prop:HorizonFunctionsCoveringNumbers} that
  $M_n$ is a $C_1 \eps$-net (in $L^1 ( (-1,1)^d )$) for $\HorizonFunctions(\Barron_C)$.
  Therefore, with $\lambda = \lambda(d,s,C) > 0$ as in \Cref{prop:HorizonFunctionsCoveringNumbers},
  there is a $\lambda C_1 \eps$-net $M_n' \subset B_0^{1,\infty}(\Omega)$ for
  ${\CalG := \{ f \in B_s^{2,2}(\Omega) \colon \| f \|_{B_s^{2,2}} \leq 1 \}}$
  satisfying $|M_n '| \leq |M_n| \leq 2^{k-1}$
  for $k := 1 + \lceil C_4 \eps^{-\sigma} \log_2(1/\eps) \rceil$.
  Defining $M_n '' := \{ f|_{\Omega_0} \colon f \in M_n' \} \subset B_0^{1,\infty}(\Omega_0)$,
  we thus see that $M_n ''$ is a $\lambda C_1 \eps$-net for $U_A$.

  Overall, we thus see because of $k \leq C_5 \eps^{-\sigma} \log_2(1/\eps)$ that
  \[
    \lambda \, C_1 \, \eps
    \geq e_k
    \geq c \cdot k^{-s/(d-1)}
    \geq c
         \cdot C_5^{-s/(d-1)}
         \cdot \eps^{s \sigma / (d-1)}
         \cdot \bigl(\log_2(1/\eps)\bigr)^{-s/(d-1)} .
  \]
  Note that this holds for all $\eps = \eps_n \to 0$ as $n \to \infty$.
  This is only possible if $s \sigma / (d-1) \geq 1$, meaning
  $\frac{1}{\sigma} \leq \frac{s}{d-1}$.
  Since $s > \frac{d+1}{2}$ can be chosen arbitrarily, this implies as claimed that
  $\frac{1}{\sigma} \leq \frac{1}{2} \frac{d+1}{d-1} = \frac{1}{2} + \frac{1}{d-1}$.
\end{proof}

The strength of the lower bound in \Cref{thm:QuantizedLowerBound} is that it applies to networks
of \emph{arbitrary depth}; but it requires the neural networks to be quantized.
Our final lower bound shows that for neural networks of a fixed maximal depth,
one can replace the quantization assumption by a suitable growth condition
on the magnitude of the weights.

\begin{theorem}\label{thm:LowerBoundWithWeightBound}
  Let $d \in \N_{\geq 2}$, $L,N \in \N$, and $\gamma, C, C_1, C_2, C_3 > 0$.
  Suppose that there is an infinite set $\mathcal{W} \subset \N$
  such that for each $W \in \mathcal{W}$ and each $f \in \Barron_C ([-1,1]^{d-1})$
  there is a neural network $\Phi$ with $d$-dimensional input and $1$-dimensional output
  and with all weights bounded in absolute value by $C_1 \, W^N$ such that
  \[
    \| H_f - R_\varrho \Phi \|_{L^1( (-1,1)^d )} \leq C_2 \cdot W^{-\gamma} ,
    \qquad
    W(\Phi) \leq C_3 \cdot W ,
    \quad \text{and} \quad
    L(\Phi) \leq L .
  \]
  Then $\gamma \leq \frac{1}{2} + \frac{1}{d-1}$.
\end{theorem}

\begin{proof}
  Let $k := \lceil \max \{ \gamma^{-1} N + C_1, \,\, \gamma^{-1} + C_3 \} \rceil$ and $m := 3 k L$.
  For $W \in \mathcal{W}$ large enough, we have $\eps := \eps_W := W^{-\gamma} \leq \frac{1}{2}$.
  For this choice of $W$ and given $f \in \Barron_C([-1,1]^{d-1})$,
  let $\Phi$ as in the assumption of the theorem.
  Note that $x \leq 2^x \leq \eps^{-x}$ for all $x \geq 0$, and hence
  \(
    W(\Phi)
    \leq C_3 \cdot W
    = C_3 \cdot \eps^{-1/\gamma}
    \leq \eps^{-(\gamma^{-1} + C_3)}
    \leq \eps^{-k}
    .
  \)
  Likewise, all weights of $\Phi$ are bounded in absolute value by
  \(
    C_1 \, W^N
    = C_1 \, \eps^{-\frac{N}{\gamma}}
    \leq \eps^{-(\frac{N}{\gamma} + C_1)}
    \leq \eps^{-k}
    .
  \)

  Overall, the ``quantization lemma'' \mbox{\cite[Lemma~VI.8]{ElbraechterDNNApproximationTheory}}
  shows that there exists an $(m,\eps)$-quantized network $\Psi$
  with $d$-dimensional input and $1$-dimensional output and such that
  \[
    {W(\Psi) \leq W(\Phi) \leq C_3 \cdot W = C_3 \cdot \eps^{-1/\gamma}}
    \qquad \text{and} \qquad
    \| R_\varrho \Phi - R_\varrho \Psi \|_{\sup}
    \leq \eps ,
  \]
  where the $\|\cdot\|_{\sup}$ norm is taken over $(-1,1)^d$.
  Hence, $\| H_f - R_\varrho \Psi \|_{L^1 ( (-1,1)^d )} \leq (2^d + C_2) \, \eps$.
  Since $\eps_W = W^{-\gamma} \to 0$ as $W \to \infty$
  with $W \in \mathcal{W}$, \Cref{thm:QuantizedLowerBound} shows that
  $\gamma = \frac{1}{1/\gamma} \leq \frac{1}{2} + \frac{1}{d-1}$, as claimed.
\end{proof}

%% file: 5-LearningBounds.tex
In this section, we provide error bounds for the performance of empirical risk minimization
for learning the indicator function of a set with boundary of Barron class.
We also briefly discuss the optimality of these results.
More precisely, we show that the best one can hope for is to (roughly) double the
``estimation-error rate'' that we obtain.
We conjecture that neither the lower bound nor the derived rate are optimal,
but we were unable to prove this.

In the following theorem, given a subset $\Omega \subset \R^d$, we use the notation
\[
  \chi_\Omega : \quad
  \R^d \to \{ \pm 1 \}, \quad
  x \mapsto \begin{cases}
              1  , & \text{if } x \in \Omega , \\
              -1 , & \text{otherwise}.
            \end{cases}
\]
Moreover, for $\CalA = (d, N_1, \dots, N_L) \in \N^{L+1}$,
we denote by $\NNWeights(\CalA)$ the set of neural networks $\Phi$
with input dimension $d$, $L$ layers, and $N_\ell$ neurons in the $\ell$th layer
for all $\ell \in \{1, \dots, L\}$. 
Finally, we define $\sign : \R \to \{ \pm 1 \}$ by $\sign(x) = 1$ for $x \geq 0$
while $\sign(x) = -1$ if $x < 0$.

\begin{theorem}\label{thm:LearningBound}
  Let $B,C \geq 1$, $M \in \N$, $d \in \N_{\geq 2}$, $\alpha \in (0,1]$, and $m \in \N$.
  Define
  \[
    N
    := \left\lceil
         \Big( (B C)^2 d \, m / \ln (B C M d \, m) \Big)^{1/(1+\alpha)}
       \right\rceil
    \in \N
  \]
  and $\CalA := \bigl(d,\,\, M (N + 2d + 2),\,\, M(4d + 2),\,\, M,\,\, 1\bigr)$.
  Let $\PP$ be a tube compatible probability measure on $\R^d$ with parameters $C,\alpha$,
  and let $\Omega \in \BarronBoundary_{B,M}(\R^d)$.
  Let $S_X = (X_1,\dots,X_m) \overset{\mathrm{iid}}{\sim} \PP$
  and define $Y_i := \chi_{\Omega} (X_i)$ for $i \in \FirstN{m}$.

  Then, given $\delta \in (0,1)$, with probability at least $1 - \delta$
  regarding the choice of $S_X$, any
  \begin{equation}
    \Phi^\ast
    \in \argmin_{\Phi \in \NNWeights(\CalA)}
          \sum_{i=1}^m
            \Indicator_{\sign(R_\varrho \Phi (X_i)) \neq Y_i}
    \label{eq:EmpiricalRiskMinimization}
  \end{equation}
  satisfies
  \begin{equation}
    \PP \Bigl(\sign\bigl(R_\varrho \Phi^\ast (X)\bigr) \neq \chi_\Omega (X)\Bigr)
    \leq C_0 \cdot
         \bigg(
           B C M \, d^{3/2} \cdot
           \Big(
             \frac{\ln (B C M d m)}{(B C)^2 d m}
           \Big)^{\gamma/2}
           + \Big( \frac{\ln(1/\delta)}{m} \Big)^{1/2}
         \bigg),
    \label{eq:GeneralizationBound}
  \end{equation}
  where $X \sim \PP$.
  Here, $C_0 \geq 1$ is an absolute constant and $\gamma = \frac{\alpha}{1 + \alpha}$.
\end{theorem}

\begin{rem*}
  1) The set
  \(
    \bigl\{
      \bigl( \sign(f(X_1)),\dots,\sign(f(X_m)) \bigr)
      \colon
      f : \R^d \to \R
    \bigr\}
    \subset \{ \pm 1 \}^m
  \)
  is finite, which implies that a minimizer as in \Cref{eq:EmpiricalRiskMinimization} always exists.

  \medskip{}

  2) In the common case where $\alpha = 1$
  (for instance, if $d \PP(x) = \Indicator_{[0,1]^{d}}(x) \, d x$),
  we have $\gamma = 1/2$, so that one gets
  \[
    \PP \Bigl(\sign\bigl(R_\varrho \Phi^\ast (X)\bigr) \neq \chi_\Omega (X)\Bigr)
    \lesssim \Bigl(\frac{\ln m}{m}\Bigr)^{1/4} + \Big( \frac{\ln(1/\delta)}{m} \Big)^{1/2} .
  \]
\end{rem*}

\begin{proof}
  All ``implied constants'' appearing in this proof are understood to be \emph{absolute} constants.

  Define $\Lambda := (B C)^2 d m / \ln (B C M d m)$, so that
  $N = \lceil \Lambda^{1/(1+\alpha)} \rceil$.
  If $\Lambda \leq 1$, then the right-hand side of \Cref{eq:GeneralizationBound} is
  at least $1$, so that the estimate is trivial.
  We can thus assume without loss of generality that $\Lambda > 1$, so that $N \geq 2$
  and $N \leq 1 + \Lambda^{1/(1+\alpha)} \leq 2 \, \Lambda^{1/(1+\alpha)}$.

  Let $\hypothesis := \{ \sign \circ R_\varrho \Phi \colon \Phi \in \NNWeights(\CalA) \}$.
  Note that since at most every neuron in layer $\ell$ can be connected
  to every neuron in layer $\ell+1$,
  the number $W(\CalA)$ of weights of a network with architecture $\CalA$ satisfies
  \(
    W(\CalA) \lesssim M^2 d^2 N .
  \)
  Therefore, \cite[Theorem~2.1]{BartlettAlmostLinearVC} shows that there
  are absolute constants $C_1,C_2 > 0$ such that
  \[
    \VC (\hypothesis)
    \leq C_1 \cdot M^2 d^2 N \cdot \ln(M^2 d^2 N)
    \leq C_2 \cdot M^2 d^2 N \cdot \ln(d M N) .
  \]

  Next, recall that $\Lambda \geq 1$ and hence
  $N \leq 2 \Lambda^{1/(1+\alpha)} \leq 2 \Lambda \lesssim (B C)^2 d m$.
  Therefore, $\ln(d M N) \lesssim 1 + \ln( (B C d)^2 M m ) \lesssim \ln (B C M d \, m)$,
  which easily implies that
  \begin{equation}
  \begin{split}
    \sqrt{\frac{\VC(\hypothesis)}{m}}
    & \lesssim m^{-1/2} M d \sqrt{N} \sqrt{\ln(d M N)} \\
    & \lesssim \big( \ln (B C d M m) \big)^{\frac{1}{2} (1 - \frac{1}{1+\alpha})}
               \cdot M
               \cdot d^{\frac{1}{2}(2 + \frac{1}{1+\alpha})}
               \cdot (B C)^{1/(1+\alpha)}
               \cdot m^{\frac{1}{2} (\frac{1}{1+\alpha} - 1)} \\
    & =        (B C)^{1 - \gamma}
               \cdot M
               \cdot d^{\frac{3}{2} - \frac{\gamma}{2}}
               \cdot \big( \ln (B C M d \, m) \big)^{\gamma / 2}
               \cdot m^{-\gamma/2}
      =:       (\ast) .
  \end{split}
  \label{eq:LearningBoundVCEstimate}
  \end{equation}
  To make use of this estimate, note that the
  \emph{Fundamental theorem of statistical learning theory}
  (see \cite[Theorem~6.8 and Definitions~4.1 and 4.3]{ShalevShwartzUnderstandingMachineLearning})
  shows for arbitrary $\eps, \delta \in (0,1)$ that if we set
  \[
    L_{\PP}(h) := \PP \bigl( h(X) \neq \chi_\Omega(X) \bigr)
    \quad \text{and} \quad
    L_{S}(h) := \frac{1}{m} \sum_{i=1}^m \Indicator_{h(X_i) \neq \chi_\Omega(X_i)},
  \]
  then, with probability at least $1 - \delta$ with respect to the choice
  of $S = (X_1,\dots,X_m) \overset{iid}{\sim} \PP$, we have
  \begin{equation}
    \forall \, h \in \hypothesis: \quad
      |L_{\PP}(h) - L_S(h)| \leq \eps
    ,
    \label{eq:GeneralizationBoundAbstract}
  \end{equation}
  provided that $m \geq C_3 \frac{\VC(\hypothesis) + \ln(1/\delta)}{\eps^2}$.
  Using the estimate $\sqrt{a+b} \leq \sqrt{a} + \sqrt{b}$ for $a,b \geq 0$,
  it is easy to see that the condition on $m$ is satisfied if
  \(
    \eps
    \geq \sqrt{C_3} \cdot
         \big(
           \sqrt{\VC(\hypothesis)/m}
           + \sqrt{\ln(1/\delta)/m}
         \big)
    .
  \)
  Finally, thanks to \Cref{eq:LearningBoundVCEstimate}, we see that
  there is an absolute constant $C_0 > 0$
  (which we can without loss of generality take to satisfy $C_0 \geq 24$)
  such that this condition holds as soon as
  \[
    \eps
    \geq \eps_0
    :=   \frac{C_0}{4} \cdot \Bigl[ (\ast) + \Big( \frac{\ln(1/\delta)}{m} \Big)^{1/2} \Bigr] .
  \]
  This is satisfied if we take $\eps$ as one fourth of the right-hand side
  of \Cref{eq:GeneralizationBound}; for this, note that in case of $\eps \geq 1$,
  Estimate \eqref{eq:GeneralizationBoundAbstract} is trivially satisfied.

  \medskip{}

  Now, choosing $\eps$ to be one fourth of the right-hand side of \Cref{eq:GeneralizationBound},
  we know that with probability at least $1 - \delta$ with respect to the choice of $S$,
  \Cref{eq:GeneralizationBoundAbstract} holds.
  Let us assume that $S = (X_1,\dots,X_m)$ is chosen such that this holds.
  Now, \Cref{thm:BarronBoundaryApproxGuarantee} shows that there is $\Phi_0 \in \NNWeights(\CalA)$
  such that
  \[
    \PP
    (
     \{
        x \in \R^d
        \colon
        \Indicator_\Omega (x) \neq R_\varrho \Phi_0 (x)
     \}
    )
    \leq 6 \, B C M \, d^{3/2} \, N^{-\alpha/2}
    \leq \frac{C_0}{4} B C M \, d^{3/2} \, \Lambda^{-\gamma/2}
    \leq \eps .
  \]
  It is not hard to see that there exists $\Phi_1 \in \NNWeights(\CalA)$ satisfying
  $R_\varrho \Phi_1 = -1 + 2 R_\varrho \Phi_0$ and that
  if $\Indicator_\Omega(x) = R_\varrho \Phi_0(x)$,
  then $h_1(x) = R_\varrho \Phi_1 (x) = \chi_\Omega(x)$ for
  $h_1 := \sign \circ (R_\varrho \Phi_1) \in \hypothesis$.
  Therefore,
  \(
    L_{\PP}(h_1)
    = \PP ( h_1(X) \neq \chi_\Omega(X))
    \leq \PP (\Indicator_\Omega (X) \neq R_\varrho \Phi_0 (X))
    \leq \eps .
  \)
  Overall, if $\Phi^\ast \in \NNWeights(\CalA)$ satisfies \Cref{eq:EmpiricalRiskMinimization},
  and if we set $h^\ast := \sign \circ R_\varrho \Phi^\ast$,
  then \Cref{eq:GeneralizationBoundAbstract} shows
  \[
    L_{\PP}(h^\ast)
    \leq L_S(h^\ast) + \eps
    \leq L_S(h_1) + \eps
    \leq L_{\PP}(h_1) + 2 \eps
    \leq 3 \eps
    \leq 4 \eps
    =    \mathrm{RHS} \! \text{ \eqref{eq:GeneralizationBound} }
    ,
  \]
  which proves \Cref{eq:GeneralizationBound}.
\end{proof}

\begin{remark}[Quantifying the non-optimality of the learning bound]\label{rem:LearningSharpness}
  By taking $\delta \sim m^{-\gamma/2}$, it is not hard to see
  that the bound in \Cref{thm:LearningBound} implies that the learning algorithm
  \[
    \bigl( (X_1,\chi_\Omega(X_1)), \dots, (X_m, \chi_\Omega(X_m)) \bigr)
    \mapsto A_S := \sign \circ R_\varrho \Phi_S^\ast
  \]
  with $\Phi_S^\ast$ a solution to \Cref{eq:EmpiricalRiskMinimization} satisfies
  \[
    \EE_S \bigl[ \| A_S - \chi_\Omega  \|_{L^1(\PP)} \bigr]
    \lesssim \big[ \ln(m) / m \big]^{\gamma / 2} ;
  \]
  here, we used that $|A_S - \chi_\Omega| \leq 2 \cdot \Indicator_{A_S \neq \chi_\Omega}$.
  For the uniform measure $d \PP = 2^{-d} \Indicator_{[-1,1]^{d}} \, d x$, we have
  $\gamma = 1/2$, and therefore
  \(
    \EE_S \bigl[ \| A_S - \chi_\Omega  \|_{L^1([-1,1]^d)} \bigr]
    \lesssim \big[ \ln(m) / m \big]^{1/4} .
  \)
  In the remainder of this remark, we sketch an argument showing that no learning algorithm
  $S \mapsto A_S$ can satisfy
  \begin{equation}
    \EE_S \bigl[ \| A_S - \chi_\Omega  \|_{L^1([-1,1]^d)} \bigr]
    \lesssim m^{-\theta}
    \quad \text{with} \quad
    \theta > \theta^\ast := \frac{1}{2} \frac{d + 2 + \Indicator_{2\Z + 1}(d)}{d-1}.
    \label{eq:LearningLowerBound}
  \end{equation}
  Note that $\theta^\ast \to \frac{1}{2}$ as $d \to \infty$, which still leaves a gap
  between this lower bound and the estimation-error rate $m^{-1/4}$ that we obtain.

  We expect the lower bound of \eqref{eq:LearningLowerBound} to be suboptimal.
  One reason why we assume so is that, for a general estimation problem,
  where the error of estimating a density from $m$ measurements is measured
  with respect to the Kullback-Leibler divergence,
  \cite[Theorem~1]{YangBarronMinimaxRates} yields a general lower bound
  in terms of the metric entropy of the class of densities.
  As we have seen in the proof of Theorem~\ref{thm:QuantizedLowerBound},
  the metric entropy of the set of horizon functions can be lower bounded
  by using fact that a ball in $B_s^{2,2}$ for $s > (d+1)/2$ embedds
  into the Fourier-analytic Barron space.
  By this observation it can be seen using \cite[Theorem~1]{YangBarronMinimaxRates}
  that a lower bound on the expected error of estimating $\chi_\Omega$
  from $m$ measurements as in Theorem~\ref{thm:LearningBound}
  measured with respect to the Kullback-Leibler divergence
  is given by $\CalO(m^{-(\frac{d+1}{4d} + \delta)})$ for any $\delta > 0$ .
  Note that this rate almost matches the upper bound given in Theorem~\ref{thm:LearningBound}
  for the $L^1$ estimation error.
  The argument in \cite{YangBarronMinimaxRates} yields bounds for $L^2$ distances
  under additional assumptions, see \cite[Theorems~4,5,6]{YangBarronMinimaxRates}.
  However, none of these assumptions are satisfied in our case.

  To prove \eqref{eq:LearningLowerBound}, assume by way of contradiction
  that some learning algorithm $S \mapsto A_S$ satisfies \Cref{eq:LearningLowerBound},
  uniformly for all $\Omega \in \BarronBoundary_{1,1}(\R^d)$.
  Let $Q := (-1,1)^{d-1}$ and ${s := 1 + \lfloor \frac{d+1}{2} \rfloor}$,
  as well as $\CalG := \{ f \in W^{s,2}(Q) \colon \| f \|_{W^{s,2}} \leq 1 \}$
  with the usual Sobolev space $W^{s,2}(Q)$.
  Since $s > \frac{d+1}{2}$, we see as in the proof of \Cref{prop:HorizonFunctionsCoveringNumbers}
  that there is $c > 0$ such that
  \[
    \forall \, f \in \CalG: \quad
      \| c \, f \|_{\sup} \leq 1
      \quad \text{and} \quad
      \Omega_f
      := \bigl\{ (x, t) \in [-1,1]^{d-1} \times [-1,1] \,\,\colon\,\, t \leq c \, f(x) \bigr\}
      \in \BarronBoundary_{1,1}(\R^d) .
  \]

  Let $W = (W_1,\dots,W_m) \overset{\mathrm{iid}}{\sim} U([-1,1]^d)$, and write $W_i = (X_i, X_i')$
  with $X_i \in [-1,1]^{d-1}$ and $X_i ' \in [-1, 1]$.
  Given $f \in \CalG$, let
  \[
    Y_i
    := - 1 + 2 \cdot \Indicator_{X_i ' \leq c f(X_i)}
    = \chi_{\Omega_f} (W_i)
    ,
  \]
  and set $S_f := \bigl( (W_1, Y_1), \dots, (W_m, Y_m) \bigr)$.
  By \Cref{eq:LearningLowerBound}, there is $C > 0$ independent of $m$ such that
  \[
    \EE_W \bigl[ \| A_{S_f} - \chi_{\Omega_f} \|_{L^1([-1,1]^d)}\bigr]
    \leq C \cdot m^{-\theta} .
  \]

  Note that $S_f$ is uniquely determined by fixing $W$ and $f$,
  and that $S_f$ does not depend fully on $f$, but only on $m$ point samples of $f$.
  Define
  \[
    B_W : \quad
    [-1,1]^{d-1} \to \R, \quad
    x \mapsto \frac{1}{c} \cdot \Big( -1 + \int_{-1}^1 \frac{1 + A_{S_f} (x,t)}{2} \, d t \Big) .
  \]
  Note that $B : (W,f) \mapsto B_W$ is a Monte-Carlo algorithm
  in the sense of \cite[Section~2]{HeinrichRandomApproximation},
  and for each (random) choice of $W$, $B$ computes its output based on $m$ point samples of $f$.
  To motivate the definition of $B_W$, note because of $\| c \, f \|_{\sup} \leq 1$ that
  \[
    \int_{-1}^1 \frac{1 + \chi_{\Omega_f} (x,t)}{2} \, d t
    = \int_{-1}^1
        \Indicator_{\Omega_f} (x,t)
      \, d t
    = \int_{-1}^1
        \Indicator_{t \leq c \, f(x)}
      \, d t
    = \int_{-1}^{c \, f(x)} \, d t
    = c \, f(x) + 1 ,
  \]
  and hence
  \(
    f(x) = \frac{1}{c} \Big( -1 +  \int_{-1}^1 \frac{1 + \chi_{\Omega_f}(x,t)}{2} \, d t \Big) .
  \)
  This implies
  \[
    \| B_W - f \|_{L^1([-1,1]^{d-1})}
    \leq \frac{1}{2 c}
         \int_{[-1,1]^{d-1}}
           \Big|
             \int_{-1}^1
               A_{S_f} (x,t) - \chi_{\Omega_f} (x,t)
             \, d t
           \Big|
         \, d x
    \leq \frac{1}{2 c} \| A_{S_f} - \chi_{\Omega_f} \|_{L^1} ,
  \]
  and hence
  \[
    \EE_W \| B_W - f \|_{L^1}
    \leq \frac{1}{2 c} \EE_W \| A_{S_f} - \chi_{\Omega_f} \|_{L^1}
    \leq \frac{C}{2 c} \cdot m^{-\theta} .
  \]
  Note that this holds for every $f \in \CalG$ and recall from above that $B : (W,f) \mapsto B_W$
  is a Monte-Carlo algorithm that depends on $f$ only through $m$ point samples.
  However, it is known from information-based complexity
  (see for instance \cite[Theorem~6.1]{HeinrichRandomApproximation})
  that such an error bound for a Monte-Carlo algorithm can only hold if
  \(
    \theta
    \leq \frac{s}{d-1}
    = \frac{1}{2} \frac{d + 2 + \Indicator_{2\Z + 1}(d)}{d-1}
    = \theta^\ast
    .
  \)
\end{remark}

%% file: 6-GeneralMeasures.tex
In this section, we show that for general probability measures,
one cannot derive any non-trivial minimax bound regarding the approximation
of sets with Barron class boundary using ReLU neural networks.

The following general result shows that for sets of infinite $VC$-dimension
and general probability measures, no non-trivial minimax approximation results
using neural networks can be derived.
To conveniently formulate the result, we use the notation
\[
  \NNWeights_{\! N,L}
  := \big\{
       \Phi
       \,\,\colon\,\,
       \Phi \text{ NN with input dimension } d,
       \text{ with } L(\Phi) \leq L
       \text{ and } N(\Phi) \leq N
     \big\} .
\]
Furthermore, we continue to write $\varrho$ for the ReLU.
The proof of the following lemma is based on (the proof of) the no-free-lunch theorem
as presented in \cite[Theorem~5.1]{ShalevShwartzUnderstandingMachineLearning}.

\begin{proposition}\label{prop:NoApproximationBoundsForInfiniteVC}
  Let $\Omega \subset \R^d$ be Borel measurable and let
  $\calF \!\subset\! \{ F : \Omega \!\to\! \{ 0,1 \} \colon F \text{ measurable} \}$
  such that $\VC(\calF) = \infty$.

  Then for arbitrary $N,L \in \N$ we have
  \[
    \sup_{\strut \mu \text{ Borel prob.~measure on } \Omega} \,\,
      \sup_{\strut F \in \calF} \,\,
        \inf_{\strut \Phi \in \NNWeights_{\! N,L}} \,\,
          \big\| F - \Indicator_{(0,\infty)} \circ R_\varrho \Phi \big\|_{L^1(\mu)}
    \geq \frac{1}{16} .
  \]
\end{proposition}

\begin{remark}\label{rem:OmittingHeavisideDoesNotHelp}
%
Even without composing the ReLU neural network $R_\varrho \Phi$ with $\Indicator_{(0,\infty)}$,
the above result implies that
\[
  \sup_{\strut \mu \text{ Borel prob.~measure on } \Omega} \,\,
    \sup_{\strut F \in \calF} \,\,
      \inf_{\strut \Phi \in \NNWeights_{\! N,L}} \,\,
        \big\| F - R_\varrho \Phi \big\|_{L^1(\mu)}
  \geq \frac{1}{32} .
\]

This follows by first noting that $\{ R_\varrho \Phi \colon \Phi \in \NNWeights_{\! N,L} \}$
is closed under addition of constant functions and secondly by noting that
\begin{equation}
  |y - \Indicator_{(0,\infty)} (z - \tfrac{1}{2})|
  \leq 2 \, |y - z|
  \qquad \forall \, y \in \{ 0,1 \} \text{ and } z \in \R .
  \label{eq:HeavisideCompositionEstimate}
\end{equation}
This estimate is trivial in case of $y = \Indicator_{(0,\infty)} (z - \frac{1}{2})$;
thus, let us assume that $y \neq \Indicator_{(0,\infty)}(z - \frac{1}{2})$.
Then there are two cases:
First, if $z \leq \frac{1}{2}$, then $\Indicator_{(0,\infty)} (z - \frac{1}{2}) = 0$ and $y = 1$,
which implies that
\(
  2 \, |y - z|
  \geq 2 (y - z)
  \geq 1
  = |y - \Indicator_{(0,\infty)}(z - \frac{1}{2})|
  .
\)
If otherwise $z > \frac{1}{2}$, then $\Indicator_{(0,\infty)}(z - \frac{1}{2}) = 1$ and $y = 0$,
so that $2 \, |y - z| = 2 |z| \geq 1 = |y - \Indicator_{(0,\infty)}(z - \frac{1}{2})|$.
This proves \eqref{eq:HeavisideCompositionEstimate}.
\end{remark}

\begin{proof}[Proof of \Cref{prop:NoApproximationBoundsForInfiniteVC}]
  Let $N,L \in \N$ be arbitrary.
  As shown for instance in \cite[Theorem~8.7]{AnthonyBartlettNeuralNetworkLearning},
  if we consider the function class
  \(
    \mathcal{N} := \{
                     \Indicator_{(0,\infty)} \circ R_\varrho \Phi
                     \colon
                     \Phi \in \NNWeights_{\! N,L}
                   \}
  \),
  then $\VC(\mathcal{N}) < \infty$.
  By the fundamental theorem of statistical learning theory
  (see for instance \cite[Theorem~6.7]{ShalevShwartzUnderstandingMachineLearning}),
  this means that $\mathcal{N}$ has the \emph{uniform convergence property},
  which implies (see \cite[Definition~4.3]{ShalevShwartzUnderstandingMachineLearning})
  that there is some $n \in \N$ such that for each measurable $F : \Omega \to \{ 0,1 \}$
  and each Borel probability measure $\mu$ on $\Omega$,
  if we choose $S_X = (X_1,\dots,X_n) \overset{\text{i.i.d.}}{\sim} \mu$,
  then with probability at least $1 - \frac{1}{10}$ with respect to the choice of $S_X$, we have
  \begin{equation}
    \sup_{\phi \in \mathcal{N}}
      \big| R_{\mu,F} (\phi) - R_{S_X,F}(\phi) \big|
    \leq \frac{1}{32} ,
    \label{eq:NoGeneralMeasureProofGeneralizationBound}
  \end{equation}
  where
  \[
    R_{\mu,F}(\phi) = \mu\bigl(\{ x \in \Omega \colon \phi(x) \neq F(x) \}\bigr)
    \qquad \text{and} \qquad
    R_{S_X,F} (\phi)
    = \frac{1}{n} \sum_{i=1}^n \Indicator_{\phi(X_i) \neq F(X_i)} .
  \]
  Note $|F - \phi| \in \{ 0,1 \}$, whence $R_{\mu,F} (\phi) = \| F - \, \phi \|_{L^1(\mu)}$
  and ${R_{S_X, F} (\phi) = \frac{1}{n} \! \sum_{i=1}^n \! |\phi(X_i) - F(X_i)|}$.

  Since $\VC(\calF) = \infty$, there is a set $\Omega_0 \subset \Omega$ of cardinality
  $|\Omega_0| = 2n$ such that $\Omega_0$ is shattered by $\calF$, meaning that if we set
  $\calG := \bigl\{ g : \Omega_0 \to \{ 0, 1\} \bigr\}$,
  then $\calG = \{ f|_{\Omega_0} \colon f \in \calF \}$.
  Let $\mu := U(\Omega_0)$ denote the uniform distribution on $\Omega_0$, meaning
  $\mu(\{ x \}) = 1 / |\Omega_0|$ for all $x \in \Omega_0$,
  and assume towards a contradiction that
  \begin{equation}
    \sup_{F \in \calF} \,
      \inf_{\phi \in \mathcal{N}} \,
        \| F - \phi \|_{L^1(\mu)}
    < \frac{1}{16} .
    \label{eq:InfiniteVCContradictionAssumption}
  \end{equation}
  Now, given any
  $S = \bigl( (X_i,Y_i) \bigr)_{i=1,\dots,n} \in \bigl(\Omega_0 \times \{ 0,1 \}\bigr)^n$,
  choose $\phi_S \in \mathcal{N}$ satisfying
  \begin{equation}
    \phi_S \in \argmin_{\phi \in \mathcal{N}}
                 \sum_{i=1}^n
                   |\phi(X_i) - Y_i| .
    \label{eq:NoGeneralMeasureProofERM}
  \end{equation}
  Such a function $\phi_S$ exists, since the expression $\sum_{i=1}^N |\phi(X_i) - Y_i|$
  only depends on $\phi|_{\Omega_0}$, while
  ${\{ \phi|_{\Omega_0} \colon \phi \in \mathcal{N} \} \subset \{0,1\}^{\Omega_0}}$
  is a finite set.
  Here, $\{ 0,1 \}^{\Omega_0} = \bigl\{ \psi : \Omega_0 \to \{ 0,1 \} \bigr\}$ is the set of all
  functions from $\Omega_0$ to $\{ 0,1 \}$.

  For $S_X = (X_1,\dots,X_n) \in \Omega_0^n$ and $g \in \calG$,
  let us define $S_X (g) := \bigl( (X_i, g(X_i))\bigr)_{i=1,\dots,n}$.
  Now, given an arbitrary $g \in \calG$, recall from above that $g = F|_{\Omega_0}$
  for some $F \in \calF$.
  Thanks to \eqref{eq:InfiniteVCContradictionAssumption},
  there is thus some $\phi^\ast \in \mathcal{N}$ satisfying
  $\| g - \phi^\ast \|_{L^1(\mu)} = \| F - \phi^\ast \|_{L^1(\mu)} < \frac{1}{16}$.
  Overall, we thus see that with probability at least $1 - \frac{1}{10}$
  with respect to the choice of $S_X = (X_1,\dots,X_n) \overset{\text{i.i.d.}}{\sim} \mu$, we have
  \begin{align*}
    \| g - \phi_{S_X (g)} \|_{L^1(\mu)}
    & = \| F - \phi_{S_X (F)} \|_{L^1(\mu)}
      = R_{\mu,F} \bigl( \phi_{S_X(F)} \bigr) \\
    & \overset{\eqref{eq:NoGeneralMeasureProofGeneralizationBound}}{\leq}
      \frac{1}{32} + R_{S_X, F} \bigl( \phi_{S_X(F)} \bigr) \\
    & \overset{\eqref{eq:NoGeneralMeasureProofERM}}{\leq}
      \frac{1}{32} + R_{S_X, F} \bigl( \phi^\ast \bigr) \\
    & \overset{\eqref{eq:NoGeneralMeasureProofGeneralizationBound}}{\leq}
      \frac{1}{16} + R_{\mu, F} \bigl( \phi^\ast \bigr)
      =    \frac{1}{16} + \| F - \phi^\ast \|_{L^1(\mu)}\\
    & \overset{\eqref{eq:InfiniteVCContradictionAssumption}}{<}    \frac{1}{8} .
  \end{align*}
  Since $|g - \phi_{S_X (g)}| \leq 1$, we thus see for every $g \in \calG$ that
  \[
    \EE_{S_X} \,\, \| g - \phi_{S_X(g)} \|
    \leq \frac{1}{10} + \frac{1}{8}
    < \frac{1}{4}
    \quad \text{and hence} \quad
    \EE_{S_X}
    \Big[
      \frac{1}{|\calG|}
      \sum_{g \in \calG}
        \| g - \phi_{S_X(g)} \|_{L^1(\mu)}
    \Big]
    < \frac{1}{4} .
  \]
  In the last part of the proof, we will show that this is impossible,
  by showing for \emph{every} ${S_X = (X_1,\dots,X_n) \in \Omega_0^n}$ that
  $\frac{1}{|\calG|} \sum_{g \in \calG} \| g - \phi_{S_X(g)} \|_{L^1(\mu)} \geq \frac{1}{4}$.

  Thus, let $S_X = (X_1,\dots,X_n) \in \Omega_0^n$ be fixed,
  and set $\Omega_1 := \{ X_1,\dots,X_n \}$,
  noting that $|\Omega_0 \setminus \Omega_1| \geq n$.
  Given $g \in \calG$ and $x \in \Omega_0$, define
  \[
    g^{(x)} : \quad
    \Omega_0 \to \{ 0,1 \}, \quad
    y \mapsto \begin{cases}
                g(y) ,    & \text{if } y \neq x , \\
                1 - g(x), & \text{otherwise} .
              \end{cases}
  \]
  It is easy to see that $\calG \to \calG, g \mapsto g^{(x)}$ is bijective,
  since $(g^{(x)})^{(x)} = g$.
  Furthermore, given any $x \in \Omega_0 \setminus \Omega_1$, note that $S_X (g) = S_X (g^{(x)})$,
  so that
  \[
    \bigl|g(x) - \phi_{S_X (g)} (x)\bigr|
    + \bigl|g^{(x)}(x) - \phi_{S_X (g^{(x)})} (x)\bigr|
    = \bigl|g(x) - \phi_{S_X (g)} (x)\bigr|
      + \bigl|g^{(x)}(x) - \phi_{S_X (g)} (x)\bigr|
    = 1 .
  \]
  Overall, we thus see
  \begin{align*}
    \frac{1}{|\calG|}
    \sum_{g \in \calG}
      \| g - \phi_{S_X (g)} \|_{L^1(\mu)}
    & \geq \frac{1}{2n}
           \frac{1}{|\calG|}
           \sum_{x \in \Omega_0 \setminus \Omega_1}
             \sum_{g \in \calG}
               |g(x) - \phi_{S_X(g)}(x)| \\
    & \geq \frac{1}{2n}
           \frac{1}{2|\calG|}
           \sum_{x \in \Omega_0 \setminus \Omega_1}
             \sum_{g \in \calG}
               \big[ |g(x) - \phi_{S_X(g)}(x)| + |g^{(x)}(x) - \phi_{S_X(g^{(x)})}(x)| \big] \\
    & =    \frac{|\Omega_0 \setminus \Omega_1|}{2n}
           \cdot \frac{|\calG|}{2 |\calG|}
      \geq \frac{1}{4} ,
  \end{align*}
  as claimed.
  This completes the proof.
\end{proof}

In \Cref{prop:NoApproximationBoundsForInfiniteVC}, the measure $\mu$ might depend on the choice
of $N,L \in \N$.
The next result shows that even if one restricts to a \emph{fixed} measure $\mu$
for all $N,L \in \N$, the approximation rate can get arbitrarily bad.

\begin{proposition}
  Let $\Omega \subset \R^d$ be Borel measurable and let
  $\calF \!\subset\! \{ F : \Omega \!\to\! \{ 0,1 \} \colon F \text{ measurable} \}$
  such that $\VC(\calF) = \infty$.

  Then for each null-sequence $(\eps_n)_{n \in \N}$ and arbitrary sequences
  $(N_n)_{n \in \N} \!\subset\! \N$ and ${(L_n)_{n \in \N} \!\subset\! \N}$,
  there is a Borel probability measure $\mu$ on $\Omega$
  and some $n_0 \in \N$ such that
  \[
    \sup_{F \in \calF} \,\,
      \inf_{\Phi \in \NNWeights_{\! N_n, L_n}}
        \| F - \Indicator_{(0,\infty)} \circ R_\varrho \Phi \|_{L^1(\mu)}
    \geq \eps_n
    \qquad \forall \, n \in \N_{\geq n_0} .
  \]
\end{proposition}

\begin{proof}
  Define $\tau_n := \sup_{k \geq n} \eps_k$, as well as $N_n ' := \max \{ N_1, \dots, N_n \}$
  and $L_n ' := \max \{ L_1, \dots, L_n \}$ for $n \in \N$.
  Note that $(\tau_n)_{n \in \N}$ is a non-increasing null-sequence;
  in particular, $\tau_n \geq 0$ for all $n \in \N$.
  Choose a strictly increasing sequence $(n_\ell)_{\ell \in \N} \subset \N$ satisfying
  $\tau_{n_\ell} \leq 2^{-5-\ell}$, so that
  \({
    \kappa
    := \sum_{\ell=1}^\infty
         \tau_{n_\ell}
    \leq 2^{-5} \sum_{\ell=1}^\infty
                  2^{-\ell}
    =    \frac{1}{32} .
  }\)
  Now, \Cref{prop:NoApproximationBoundsForInfiniteVC} yields for each $\ell \in \N$
  a Borel probability measure $\mu_\ell$ and some $F_\ell \in \calF$ satisfying
  \({
    \inf_{\Phi \in \NNWeights_{\! N_{n_\ell}', L_{n_\ell}'}}
      \| F_\ell - \Indicator_{(0,\infty)} \circ R_\varrho \Phi \|_{L^1(\mu_\ell)}
    \geq \frac{1}{32} .
  }\)
  Fix some $\omega_0 \in \Omega$ and define
  $\mu := 32 \sum_{\ell=1}^\infty \tau_{n_\ell} \, \mu_{\ell+1} + (1 - 32 \kappa) \delta_{\omega_0}$,
  so that $\mu$ is a Borel probability measure on $\Omega$.

  Now, given any $n \in \N_{\geq n_1}$, let $\ell \in \N$ with $n_\ell \leq n < n_{\ell+1}$,
  so that ${\tau_{n_\ell} = \sup_{k \geq n_\ell} \eps_k \geq \eps_n}$
  and $N_n \leq N_n ' \leq N_{n_{\ell+1}}'$ as well as $L_n \leq L_n ' \leq L_{n_{\ell+1}} '$.
  Therefore,
  \begin{align*}
    & \sup_{F \in \calF} \,\,
        \inf_{\Phi \in \NNWeights_{\! N_n, L_n}} \,\,
          \big\| F - \Indicator_{(0,\infty)} \circ R_\varrho \Phi \big\|_{L^1(\mu)} \\
    & \geq 32 \, \tau_{n_{\ell}} \cdot
           \inf_{\Phi \in \NNWeights_{\! N_{n_{\ell+1}}' , L_{n_{\ell+1}}' }}
             \| F_{\ell+1} - \Indicator_{(0,\infty)} \circ R_\varrho \Phi \|_{L^1(\mu_{\ell+1})}
      \geq \tau_{n_\ell}
      \geq \eps_n .
  \end{align*}
  Since $n \in \N_{\geq n_1}$ was arbitrary, we are done.
\end{proof}

Finally, we show that the class of Barron horizon functions (and thus also the class of sets
with boundary of Barron class) has infinite VC dimension, so that the previous results apply
in this setting.

\begin{lemma}\label{lem:BarronHorizonVCDimension}
  Let $d \geq 2$ and $Q = [-1,1]^d$, as well as $C > 0$ and $M \in \N$.
  Then
  \[
    \VC\bigl(\BarronBoundary_{C,M} (\R^d)\bigr)
    \geq \VC \bigl(\BarronHorizon_C(Q)\bigr)
    = \infty .
  \]
\end{lemma}

\begin{proof}
  Let $n \in \N$ be arbitrary.
  For each $k \in \FirstN{n}$, choose
  $\varphi_n^{(k)} \in C_c^\infty \bigl( (\frac{k-1}{n}, \frac{k}{n}) \times (-1,1)^{d-2}\bigr)$
  satisfying $\varphi_n^{(k)} \geq 0$
  and $\varphi_n^{(k)} (\frac{k-1}{n} + \frac{1}{2n}, 0, \dots, 0) = 1$.
  Define $X := [-1,1]^{d-1}$ and use \Cref{rem:AllBarronSetsAreSubsetsOfTheApproximationSpace}
  to select $C' > 0$ satisfying $\Barron_{C'}(X, 0) \subset \BarronApproximation_{C}(X)$.
  It is easy to see that there is some $\tau_n > 0$ satisfying
  $\tau_n \, \varphi_n^{(k)} \in \Barron_{C'/n}(X)$.
  Now, given $\theta = (\theta_1,\dots,\theta_n) \in \{ 0,1 \}^n$, define
  \[
    f_n^{(\theta)}
    := \tau_n \sum_{k=1}^n (2 \theta_k - 1) \varphi_n^{(k)}
    \in \Barron_{C'} (X,0)
    \subset \BarronApproximation_{C}(X) .
  \]
  This implies $H_n^{(\theta)} \in \BarronHorizon_C (Q)$, where
  \(
    H_n^{(\theta)} (x)
    := \Indicator_{f_n^{(\theta)}(x_1,\dots,x_{d-1}) \geq x_d} .
  \)
  Furthermore, in view of
  \({
    f_n^{(\theta)} (\frac{k-1}{n} + \frac{1}{2n}, 0, \dots, 0)
    = (2 \theta_k - 1) \tau_n ,
  }\)
  we see that
  \[
    H_n^{(\theta)} \bigl( \tfrac{k-1}{n} + \tfrac{1}{2n}, 0, \dots, 0 \bigr)
    = \Indicator_{(2 \theta_k - 1) \tau_n \geq 0}
    = \Indicator_{2 \theta_k \geq 1}
    = \theta_k .
  \]
  Therefore, $\BarronHorizon_C (Q)$ shatters the set
  \(
    \bigl\{
      \bigl( \frac{k-1}{n} + \frac{1}{2n}, 0 \dots, 0 \bigr)
      \colon
      k \in \FirstN{n}
    \bigr\}
    \subset Q ,
  \)
  which shows that $\VC \bigl(\BarronHorizon_C (Q)\bigr) \geq n$.
  Since this holds for every $n \in \N$, we are done.
\end{proof}

%% file: 7-BarronSpaces.tex

In the literature (see for instance \cite{ma2020towards,wojtowytsch2020representation,ma2018priori}),
there are at least three different function spaces that are referred to as \emph{Barron spaces}.
In the terminology that we used in the introduction, these are the
\emph{Fourier-analytic Barron space} and the \emph{infinite-width Barron spaces,}
either using the ReLU or the Heaviside activation function.
In the current literature, the relationship between these spaces has only been understood partially.
Therefore, we clarify this issue in this section.

To fix the terminology, let us write $\CalP_d$ for the set of all Borel probability measures
on $\R \times \R^d \times \R$.
Given a (measurable) function $\phi : \R \to \R$ and $\mu \in \CalP_d$, we write
\[
  \mu_\phi (x)
  := \int_{\R \times \R^d \times \R}
       a \cdot \phi\bigl(\langle w, x \rangle + c\bigr)
     \, d \mu(a,w,c)
  \quad \text{for } x \in \R^d ,
\]
whenever the integral exists.
Let us denote the Heaviside function by $H := \Indicator_{[0,\infty)}$
and the ReLU by $\varrho : \R \to \R, x \mapsto \max \{ 0, x \}$.
Then, given a set $\emptyset \neq U \subset \R^d$ and $s \geq 0$, we define
\begin{align*}
  \Barron_H (U)
  & := \Big\{
         f : U \to \R
         \,\, \colon \,\,
         \exists \, \mu \in \CalP_d : \,
           \| \mu \|_H < \infty
           \text{ and }
           \forall \, x \in U: f(x) = \mu_H (x)
       \Big\} , \\
  \Barron_\varrho (U)
  & := \Big\{
         f : U \to \R
         \,\, \colon \,\,
         \exists \, \mu \in \CalP_d : \,
           \| \mu \|_{\varrho} < \infty
           \text{ and }
           \forall \, x \in U : f(x) = \mu_\varrho (x)
       \Big\} , \\
  \Barron_{\Fourier,s} (U)
  & := \Big\{
         f : U \to \R
         \,\, \colon \,\,
         \exists \, F : \R^d \to \CC :
         \| F \|_{\Fourier,s} \!<\! \infty
         \text{ and }
         \forall \, x \in U \! :
           f(x) \!=\!\! \int_{\R^d} \!\! e^{i \langle x,\xi \rangle} F(\xi) \, d \xi
       \Big\} ,
\end{align*}
where
\[
  \| \mu \|_H
  := \int_{\R \times \R^d \times \R}
       |a|
     \, d \mu(a,w,c)
  \quad \text{and} \quad
  \| \mu \|_\varrho
  := \int_{\R \times \R^d \times \R}
       |a| \cdot (|w| + |c|)
     \, d \mu(a,w,c) ,
\]
while $\| F \|_{\Fourier,s} := \int_{\R^d} (1 + |\xi|)^s \, |F(\xi)| \, d \xi$.
Finally, the norms on these spaces are given by
\[
  \| f \|_{\Barron_H}
  := \inf
     \big\{
       \| \mu \|_{H}
       \quad\colon\quad
       \mu \in \CalP_d \text{ and } f = \mu_H |_U
     \big\}
\]
and similarly for $\| f \|_{\Barron_\varrho}$, while
\[
  \| f \|_{\Barron_{\Fourier,s}}
  := \inf
     \Big\{
       \| F \|_{\Fourier,s}
       \colon
       F : \R^d \to \CC \text{ measurable and }
       f(x) \!=\!\! \int_{\R^d} e^{i\langle x,\xi \rangle} F(\xi) \, d \xi \text{ for all } x \in U
     \Big\} .
\]

From the literature, the following properties of these spaces are known.

\begin{lemma}\label{lem:BarronSpacesElementary}
  Let $\emptyset \neq U \subset \R^d$ be bounded.
  Then the following hold:
  \begin{enumerate}[label=\arabic*)]
    \item $\Barron_{\varrho}(U) \hookrightarrow \Barron_H (U)$.
          If $U$ has nonempty interior, then the inclusion is strict.

    \item $\Barron_{\Fourier,1}(U) \hookrightarrow \Barron_H (U)$.

    \item $\Barron_{\Fourier,2}(U) \hookrightarrow \Barron_\varrho (U)$.
  \end{enumerate}
\end{lemma}

\begin{rem*}
  Regarding part \emph{1)}, an easy modification of the proof shows
  that it would in fact suffice for $U$ to satisfy
  $\{ t x + (1 - t) y \colon t \in [0,1] \} \subset U$ for certain $x \neq y$,
  even if $U$ is not open.
\end{rem*}

\begin{proof}
  \textbf{Ad 1):}
  Every function in $\Barron_\varrho (U)$ is Lipschitz continuous;
  see \mbox{\cite[Theorem~3.3]{wojtowytsch2020representation}}.
  On the other hand, choosing $\mu$ to be a Dirac measure, we see
  ${H_{w,c} = \bigl(x \mapsto H(\langle w,x \rangle + c)\bigr) \!\in \Barron_H(U)}$ for arbitrary
  $w \in \R^d$ and $c \in \R$.
  If $U$ has nonempty interior, one can choose $w,c$ in such a way that $H_{w,c}$ is discontinuous
  on $U$, and therefore cannot belong to $\Barron_{\varrho} (U)$.
  This shows that the inclusion has to be strict if $U$ has nonempty interior.

  The inclusion $\Barron_\varrho (U) \subset \Barron_H (U)$ is probably folklore;
  since we could not locate a reference, however, we provide the proof.
  Since $U$ is bounded, we have $U \subset \overline{B_R}(0)$ for a suitable $R > 0$.
  Set $C := 1 + R$ and note that
  \[
    \varrho(y) = \int_0^C H(y - t) \, d t
    \qquad \forall \, y \in \R \text{ with } |y| \leq R .
  \]
  Now, given $w \in \R^d$ and $c \in \R$, define $\theta_{w,c} := |w| + |c|$
  and note note $|\langle w, x \rangle + c| \leq C \cdot \theta_{w,c}$ for all $x \in U$.
  Recall that $\varrho(\gamma x) = \gamma \, \varrho(x)$ for $\gamma \geq 0$ and $x \in \R$.
  Therefore, given a measure $\mu \in \CalP_d$, and setting $\Omega := \R \times \R^d \times \R$,
  we see for all $x \in U$ that
  \begin{align*}
    \mu_\varrho (x)
      & = \int_{\Omega}
          a \cdot \varrho ( \langle w, x \rangle + c)
        \, d \mu(a,w,c)
      = \int_{\Omega}
          a \, \theta_{w,c}
          \cdot \varrho ( \langle \tfrac{w}{\theta_{w,c}}, x \rangle + \tfrac{c}{\theta_{w,c}})
        \, d \mu(a,w,c) \\
    & = \int_{\Omega}
          \int_0^C
            a \, \theta_{w,c}
            \cdot H ( \langle \tfrac{w}{\theta_{w,c}}, x \rangle + \tfrac{c}{\theta_{w,c}} - t)
          \, d t
        \, d \mu(a,w,c) \\
    & = \int_{\Omega}
          \alpha \cdot H(\langle \omega,x \rangle + s)
        \, d \nu(\alpha, \omega, s)
      = \nu_H (x) ,
  \end{align*}
  where $\nu := \Theta^{-1}(\mu \otimes \lambda)$ is the pushforward of the product measure
  $\mu \otimes \lambda$ (with $\lambda$ denoting the Lebesgue measure on $[0,C]$)
  under the map
  \[
    \Theta : \quad
    \Omega \times [0,C] \to \Omega, \quad
    \big( (a,w,c), t \big) \mapsto \big(
                                     a \cdot \theta_{w,c} , ~
                                     \tfrac{w}{\theta_{w,c}}, ~
                                     \tfrac{c}{\theta_{w,c}} - t
                                   \big) .
  \]
  Finally, note that
  \[
    \| \nu \|_{H}
    =\!\!
      \int_\Omega \!
        |\alpha|
      \, d \nu(\alpha,\omega,s)
    = \!\!
      \int_{\Omega}
        \int_0^C \!\!
          |a \cdot \theta_{w,c}|
        \, d t
      \, d \mu(a, w, c)
    \leq
         C \! \int_{\Omega} \! |a| \, (|w| + |c|) \, d \mu(a,w,c)
    =    C \, \| \mu \|_{\varrho} .
  \]
  This easily shows that
  $\| f \|_{\Barron_H (U)} \leq C \cdot \| f \|_{\Barron_\varrho (U)} < \infty$
  for all $f \in \Barron_\varrho (U)$.

  \medskip{}

  \noindent
  \textbf{Ad 2):}
  This follows from \cite[Theorem~2]{BarronNeuralNetApproximation}.

  \medskip{}

  \noindent
  \textbf{Ad 3):}
  This essentially follows from \cite[Theorem~9]{ma2020towards}, which is itself a consequence of
  (the proof of) \cite[Theorem~6]{KlusowskiBarronRiskBoundsRidgeFunctions}.

  More precisely, since $U \subset \R^d$ is bounded, we can choose $x_0 \in \R^d$ and
  $R \geq 1$ such that $U \subset x_0 + [0,R]^d$.
  Let $f \in \Barron_{\Fourier,2}(U)$ with $\| f \|_{\Barron_{\Fourier,2}} \leq 1$.
  This implies $f(x) = \int_{\R^d} e^{i \langle x,\xi \rangle} F(\xi) \, d \xi$ for $x \in U$,
  where $\| F \|_{\Fourier,2} \leq 2$.
  Define $G, H : \R^d \to \CC$ by $G(\xi) = \frac{1}{2} \bigl(F(\xi) + \overline{F (-\xi)}\bigr)$
  and $H(\xi) = R^{-d} \cdot e^{i \langle \frac{x_0}{R}, \xi \rangle} \cdot G(\xi / R)$.
  A direct calculation shows $\| G \|_{\Fourier,2} \leq 2$ and $\| H \|_{\Fourier,2} \leq 2 R^2$.
  Next, define $g, h : \R^d \to \R$ by
  $g(x) := \int_{\R^d} e^{i \langle x,\xi \rangle} G(\xi) \, d \xi$
  and $h(x) := \int_{\R^d} e^{i \langle x,\xi \rangle} H(\xi) \, d \xi$.
  It is straightforward to verify $h(\frac{x - x_0}{R}) = g(x) = f(x)$ for $x \in U$.

  By elementary properties of the Fourier transform, we see
  ${\int_{\R^d} |\xi|^2 \, |\widehat{h}(\xi)| \, d \xi \leq C}$ and
  ${h \in C^1}$ with $\| h \|_{\sup}, \| \nabla h \|_{\sup} \leq C$ where $C = C(d,R)$.
  Thanks to \cite[Theorem~9]{ma2020towards}, this implies
  ${\| h \|_{\Barron_\varrho ([0,1]^d)} \leq C' < \infty}$.
  Therefore, $h(y) = \int_{\Omega} a \, \varrho(c + \langle w, x \rangle) \, d \mu(a,w,c)$
  for all $y \in [0,1]^d$, where $\mu \in \CalP_d$ satisfies $\| \mu \|_{\varrho} \leq 2 C'$.
  Because of $y = \frac{x - x_0}{R} \in [0,1]^d$ for $x \in U$, this implies
  \({
    f(x)
    = h(\frac{x - x_0}{R})
    = \int_{\Omega}
        a \, \varrho
             \bigl(
               \langle \frac{w}{R}, x \rangle + c - \frac{\langle w, x_0 \rangle}{R}
             \bigr)
      \, d \mu(a,w,c)
    = \nu_\varrho (x) ,
  }\)
  where $\nu = \Psi^\ast \mu$ is the pushforward of $\mu$ under the map
  \(
    \Psi :
    \Omega \to \Omega ,
    (a,w,c) \mapsto \big(
                      a, \frac{w}{R}, c - \frac{\langle w,x_0 \rangle}{R}
                    \big) .
  \)
  A direct calculation shows $\| \nu \|_{\varrho} \leq (1 + |x_0|) \| \mu \|_{\varrho} \leq C''$
  for $C'' = C'' (d,R,x_0)$.
  Hence, $f \in \Barron_\varrho(U)$ with $\| f \|_{\Barron_\varrho(U)} \leq C''$.
\end{proof}

The previous lemma collected several relations between the different Barron-type spaces
from the literature.
The question of how the spaces $\Barron_\varrho$ and $\Barron_{\Fourier,1}$ are related,
however, has, to the best of our knowledge, not been answered until now.
While it is claimed in \cite[Theorem~3.1]{wojtowytsch2020representation} that
$\Barron_{\Fourier,1}$ embeds continuously into $\Barron_{\varrho}$,
citing \cite{BarronUniversalApproximation} as a reference,
we believe that this mischaracterizes the results of \cite{BarronUniversalApproximation}. 
In fact, in \cite{BarronUniversalApproximation} (or rather \cite{BarronNeuralNetApproximation}),
it is merely shown that $\Barron_{\Fourier,1}$ embeds into $\Barron_{H}$, not $\Barron_\varrho$.
As we will see in \Cref{prop:FourierBarronNotInReLUBarron} below,
we actually have $\Barron_{\Fourier,1} \nsubseteq \Barron_{\varrho}$.
The proof will be based on the following lemma, which shows that the partial derivatives
of functions in $\Barron_\varrho$ are ``uniformly of bounded variation along the coordinate axes''.
This lemma is similar in spirit to \cite[Example~4.1]{wojtowytsch2020representation},
which essentially corresponds to the one-dimensional case of the result given here.
In the following lemma, we use for a Lipschitz continuous function $g : \R^d \to \R$,
$i,j \in \FirstN{d}$, and $x \in \R^d$, the following functions
\begin{equation}
  g_{j,i,x} : \quad
  \R \to \R, \quad
  t \mapsto (\partial_j g) (x + t e_i) ,
  \label{eq:SpecialFunctionDefinition}
\end{equation}
where $(e_1,\dots,e_d)$ denotes the standard basis of $\R^d$.

\begin{lemma}\label{lem:ReLUBarronBoundedVariation}
  Let $\emptyset \neq U \subset \R^d$ be bounded.
  For every $f \in \Barron_{\varrho}(U)$, there exists a Lipschitz continuous function
  $g : \R^d \to \R$ satisfying $f = g|_U$ and
  \begin{equation}
    \sup_{i,j \in \FirstN{d}, x \in \R^d}
      \| g_{j,i,x} \|_{BV}
    \leq 4 \, \| f \|_{\Barron_\varrho} ,
    \label{eq:ReLUBarronBV}
  \end{equation}
  where we write $\| h \|_{BV} := \| h \|_{\sup} + \mathrm{TV}(h)$ for $h : \R \to \R$,
  with $\mathrm{TV}(h)$ denoting the total variation of $h$;
  see for instance \cite[Chapter~3.5]{FollandRA} for the definition.
\end{lemma}

\begin{rem*}
  The partial derivative $\partial_j g$ appearing in \Cref{eq:SpecialFunctionDefinition}
  above is the weak derivative of $g$,
  and thus a priori only uniquely defined up to changes on a null-set.
  What is meant is that there is a version of this derivative such that
  $g_{j,i,x}$ is of bounded variation for all $i,j \in \FirstN{d}$ and $x \in \R^d$,
  and such that \Cref{eq:ReLUBarronBV} holds.
\end{rem*}

\begin{proof}
  The claim is clear in case of $\| f \|_{\Barron_\varrho} = 0$;
  thus, let us assume that $\| f \|_{\Barron_\varrho \strut} > 0$.
  By definition of $\Barron_\varrho (U)$ there is a probability measure
  $\mu \in \CalP_d$ satisfying $\| \mu \|_\varrho \leq \frac{5}{4} \| f \|_{\Barron_\varrho}$
  and $f = \mu_\varrho |_U$.
  Define $\Omega := \R \times \R^d \times \R$ and $g := \mu_\varrho$,
  noting that $g : \R^d \to \R$ is well-defined, since
  \[
    \int_{\Omega}
      |a \, \varrho(\langle w,x \rangle + c)|
    \, d \mu(a,w,c)
    \leq (1 + |x|)
         \int_{\Omega}
           |a| \cdot (|w| + |c|)
         \, d \mu(a,w,c)
    < \infty
  \]
  for each $x \in \R^d$.
  Furthermore, since $\varrho$ is $1$-Lipschitz, we see for all $x,y \in \R^d$ that
  \[
    |g(x) - g(y)|
    \leq \int_{\Omega}
           |a| \cdot |\langle w, x-y \rangle|
         \, d \mu(a,w,c)
    \leq |x-y|
         \int_{\Omega}
           |a| \cdot (|w| + |c|)
         \, d \mu(a,w,c)
    \leq |x-y| \cdot \| \mu \|_{\varrho} ,
  \]
  meaning that $g$ is Lipschitz continuous.

  Now, note that $x \mapsto \varrho(\langle w,x \rangle + c)$ either vanishes identically
  (in case of $w = c = 0$) or otherwise is differentiable on
  ${\{ x \in \R^d \colon \langle w,x \rangle + c \neq 0 \}}$
  which is open and of full measure, with partial derivatives
  $\partial_j [\varrho (\langle w,x \rangle + c)] = w_j H(\langle w,x \rangle + c)$.
  Furthermore, ${x \mapsto \varrho(\langle w,x \rangle + c)}$ is Lipschitz continuous and hence
  weakly differentiable, and the weak derivative coincides almost everywhere with
  the classical derivative; see for instance \cite[Theorems~4 and 5 in Section~5.8]{EvansPDE}.
  Therefore, for any $\varphi \in C_c^\infty (U)$ and $j \in \FirstN{d}$,
  Fubini's theorem shows that
  \begin{align*}
    \int_U
      f(x) \, \partial_j \varphi (x)
    \, d x
    & = \int_{\Omega}
          a
          \int_U
            \varrho (\langle w, x \rangle + c)
            \partial_j \varphi(x)
          \, d x
        \, d \mu(a,w,c) \\
    & = - \int_{\Omega}
            \int_U
              a \, w_j \, H(\langle w,x \rangle + c)
              \varphi(x)
            \, d x
          \, d \mu(a,w,c)
      = - \int_U \varphi(x) g_j (x) \, d x ,
  \end{align*}
  meaning that $g_j = \partial_j g$ for
  \[
    g_j : \quad
    \R^d \to \R, \quad
    x \mapsto \int_U a \, w_j \, H(\langle w,x \rangle + c) \, d \mu(a,w,c) .
  \]

  Now, using the convention $\sign(x) = 1$ for $x \geq 0$ and $\sign (x) = - 1$ if $x < 0$,
  given $i,j \in \FirstN{d}$, define
  \[
    M_{\alpha,\beta}
    := \bigl\{
         (a,w,c) \in \Omega
         \,\,\colon\,
         \sign (a w_j) = \alpha \text{ and } \sign(w_i) = \beta
       \bigr\}
    \quad \text{for} \quad
    \alpha,\beta \in \{ \pm 1 \} .
  \]
  Since the Heaviside function $H$ is non-decreasing, it is then straightforward to see
  for each $x \in \R^d$ that each of the functions
  \(
    F_{\alpha,\beta,x} :
    \R \to \R,
    t \mapsto \int_{M_{\alpha,\beta}}
                a \, w_j \, H(w_i \, t + \langle w, x \rangle + c)
              \, d \mu(a,w,c)
  \)
  is monotonic and $g_{j,i,x} = \sum_{\alpha,\beta \in \{ \pm 1 \}} F_{\alpha,\beta,x}$.
  Furthermore, each of the $F_{\alpha,\beta,x}$ is bounded; precisely,
  \[
    |F_{\alpha,\beta,x} (t)|
    \leq \int_{M_{\alpha,\beta}}
           |a| \, |w_j|
         \, d \mu(a,w,c)
    \leq \int_{M_{\alpha,\beta}}
           |a| \cdot (|w| + |c|)
         \, d \mu(a,w,c) ,
  \]
  so that
  $\sum_{\alpha,\beta \in \{ \pm 1 \}} \| F_{\alpha,\beta,x} \|_{\sup} \leq \| \mu \|_\varrho$.
  It is easy to see (see \cite[Section~3.5]{FollandRA}) that every monotonic function
  $h : \R \to \R$ satisfies $\| h \|_{\BV} \leq 3 \| h \|_{\sup}$.
  Therefore, $g_{j,i,x}$ is of bounded variation with
  \[
    \| g_{j,i,x} \|_{\BV}
    \leq \sum_{\alpha,\beta \in \{ \pm 1 \}}
           \| F_{\alpha,\beta,x} \|_{\BV}
    \leq 3 \| \mu \|_\varrho
    \leq \frac{15}{4} \| f \|_{\Barron_\varrho} ,
  \]
  which easily implies the claim.
\end{proof}

We will also need the following technical lemma.
It is a well-known property of BV functions;
see for instance the proof of \cite[E6.10]{AltFA}.
For the sake of completeness and for readers unfamiliar with
functions of bounded variation, we provide a proof in Appendix~\ref{sec:TotalVariationTechnical}.

\begin{lemma}\label{lem:TechnicalBVBound}
  Let $g : \R \to \R$ be bounded and of bounded variation.
  Then, for arbitrary $\varphi \in C_c^\infty (\R)$, we have
  \(
    |\int_{\R} \varphi'(t) \, g(t) d t|
    \leq \| \varphi \|_{\sup} \, \mathrm{TV}(g) .
  \)
\end{lemma}

With these preparations, we can finally show that for most domains $U$,
we have that $\Barron_{\Fourier,1}(U)$ is not contained in $\Barron_{\varrho}(U)$.

\begin{proposition}\label{prop:FourierBarronNotInReLUBarron}
  Let $U \subset \R^d$ have nonempty interior and let $\alpha \geq 0$.
  If $\Barron_{\Fourier,\alpha}(U) \subseteq \Barron_\varrho(U)$, then $\alpha \geq 2$.
  In particular, $\Barron_{\Fourier,1}(U) \nsubseteq \Barron_\varrho (U)$.
\end{proposition}

\begin{proof}
  The proof is divided into three steps:

  \textbf{Step 1} \emph{(Setup of Banach spaces $X, Y'$):}
  We define $\beta := \max \{ 1, \alpha \}$ and
  \[
    X := \bigl\{
           \Fourier^{-1} f
           \,\, \colon \,\,
           f \in L_{(1+|\xi|)^\beta}^1 (\R^d; \CC)
           \text{ and } \Fourier^{-1} f \text { is real-valued}
         \bigr\}
  \]
  with norm
  \(
    \| \Fourier^{-1} f \|_X
    := \| f \|_{L_{(1+|\xi|)^{\beta}}^1}
    = \int_{\R^d} (1 + |\xi|)^{\beta} \, |f(\xi)| \, d \xi ,
  \)
  which is well-defined since the Fourier transform is injective on
  $L^1(\R^d) \supset L_{(1+|\xi|)^\beta}^1 (\R^d)$.
  It is straightforward to verify that $X$ is a Banach space (with $\R$ as the scalar field,
  since we require $\Fourier^{-1} f$ to be real-valued for $f \in X$),
  and by differentiation under the integral it is easy to see that $X \hookrightarrow C_b^1 (\R^d)$,
  where
  \(
    C_b^1 (\R^d)
    = \{
        g \in C^1(\R^d; \R)
        \colon
        \| g \|_{C_b}^1 < \infty
      \}
  \)
  and $\| g \|_{C_b^1} := \| g \|_{\sup} + \sum_{j=1}^d \| \partial_j g \|_{\sup}$.

  Since $U$ has nonempty interior, we have $U_0 := x_0 + (-3\eps, 3\eps)^d \subset U$
  for certain $x_0 \in \R^d$ and $\eps \in (0, 1)$.
  Let $Y := \bigl( C_c^\infty ( (-2\eps, 2\eps); \R), \| \cdot \|_{\sup} \bigr)$,
  and let $Y'$ denote the dual space of $Y$.
  Note that $Y'$ is a Banach space (see for instance \cite[Proposition~5.4]{FollandRA}),
  even though $Y$ is not.

  \medskip{}

  \noindent
  \textbf{Step 2} \emph{(Constructing a bounded operator $\Gamma : X \to Y'$):}
  Assume that $\Barron_{\Fourier,\alpha} \subset \Barron_\varrho (U)$.
  Then, since $|\xi|^\alpha \leq (1+|\xi|)^\beta$, we have for
  $f \in X$ that $f|_U \in \Barron_{\Fourier,\alpha}(U) \subset \Barron_\varrho (U)$,
  so that \Cref{lem:ReLUBarronBoundedVariation} shows that there is a Lipschitz continuous function
  $g : \R^d \to \R$ satisfying $g|_U = f|_U$ and such that for some choice of the weak derivative
  $\partial_1 g$ of $g$, if we set $e_1 = (1,0,\dots,0) \in \R^d$ and
  $g_x : \R \to \R, t \mapsto (\partial_1 g)(x + t \, e_1)$, then
  $\sup_{x \in \R^d} \| g_x \|_{\BV} \leq 4 \, \| f|_U \|_{\Barron_\varrho (U)} =: C_f$.

  Since $f|_{U_0} = g|_{U_0}$ and $f$ is continuously differentiable,
  we have $\partial_1 g = \partial_1 f$ almost everywhere on $U_0$.
  By Fubini's theorem, this implies for almost every $z \in (-\eps,\eps)^{d-1}$ that
  \(
    g_{x_0 + (0,z)}(t)
    = (\partial_1 g)(x_0 + (0,z) + t \, e_1)
    = (\partial_1 f)(x_0 + (0,z) + t \, e_1)
  \)
  for almost all $t \in (-2\eps, 2\eps)$.

  For arbitrary $\varphi \in Y$ and $z$ as above, we thus see by \Cref{lem:TechnicalBVBound} that
  \[
    \Big| \!
      \int_{\R} \!
        \varphi'(t) \,
        (\partial_1 f) \! \bigl( x_0 + (0,z) + t \, e_1 \bigr)
      d t
    \Big|
    \!=\! \Big|\! \int_{\R} \! \varphi'(t) \, g_{x_0 + (0,z)} (t) d t \Big|
    \!\leq\! \big\| g_{x_0 + (0,z)} \big\|_{\BV} \| \varphi \|_{\sup}
    \!\leq C_f \, \| \varphi \|_{\sup} .
  \]
  Recall that this holds for almost all $z \in (-\eps,\eps)^{d-1}$, and thus in particular
  for a dense subset of $(-\eps,\eps)^{d-1}$.
  By continuity of $\partial_1 f$, we can thus take the limit $z \to 0$ to see that
  \(
    |\int_{\R} \varphi'(t) \, (\partial_1 f) (x_0 + t \, e_1) \, d t|
    \leq C_f \cdot \| \varphi \|_{\sup}
  \)
  for all $\varphi \in Y$.
  We have thus shown that the linear map
  \[
    \Gamma : \quad
    X \to Y', \quad
    f \mapsto \Big(
                \varphi \mapsto \int_{\R} \varphi'(t) \, (\partial_1 f) (x_0 + t \, e_1) \, d t
              \Big)
  \]
  is well-defined.
  Note that if $f_n \xrightarrow[n\to\infty]{X} f$, then $\partial_1 f_n \to \partial_1 f$
  with uniform convergence.
  Using this observation, it is straightforward to verify that $\Gamma$ has closed graph,
  and is thus a bounded linear map, thanks to the closed graph theorem.

  Finally, note that if $f \in X \cap C^2(\R^d)$, then we see by partial integration that
  \[
    \Big|\! \int_{\R} \varphi(t) \, (\partial_1^2 f) (x_0 + t \, e_1) \, d t \Big|
    = \Big|\! \int_{\R} \varphi'(t) \, (\partial_1 f) (x_0 + t \, e_1) \, d t \Big|
    \leq \| \Gamma f \|_{Y'} \| \varphi \|_{\sup}
    \leq \| \Gamma \| \, \| f \|_X \, \| \varphi \|_{\sup}
  \]
  for all $\varphi \in Y = C_c^\infty ( (-2\eps, 2\eps); \R)$.
  By the dual characterization of the $L^1$-norm (see for instance \cite[Corollary~6.13]{AltFA}),
  this implies
  \begin{equation}
    \int_{-2\eps}^{2\eps}
      \big| (\partial_1^2 f)(x_0 + t \, e_1) \big|
    \, d t
    \leq \| \Gamma \| \cdot \| f \|_X
    \qquad \forall \, f \in X \cap C^2(\R^d) .
    \label{eq:FourierBarronReLUBarronAlmostContradiction}
  \end{equation}

  \medskip{}

  \noindent
  \textbf{Step 3} \emph{(Completing the proof):}
  Pick $\gamma \in C_c^\infty (\R^d)$ with $0 \leq \gamma \leq 1$ and such that $\gamma \equiv 1$
  on ${U_0 = x_0 + (-3\eps, 3\eps)^d}$.
  For $n \in \N$, define
  \(
    f_n :
    \R^d \to \R,
    x \mapsto \gamma(x)
              \cdot \cos
                    \bigl(
                      \frac{n \pi}{\eps} (x - x_0)\cdot e_1
                    \bigr)
    .
  \)
  Writing $T_y f (x) = f(x-y)$ and $M_\xi f (x) = e^{i \langle x,\xi \rangle} \, f(x)$
  for translation and modulation,
  and using the identity $\cos(x) = \frac{1}{2} (e^{i x} + e^{- i x})$, it is easy to see
  \(
    f_n
    = \frac{1}{2}
      T_{x_0}
      \big[
        M_{n \pi e_1 / \eps} T_{-x_0} \gamma
        + M_{-n \pi e_1 / \eps} T_{-x_0} \gamma
      \big] ,
  \)
  where $e_1 = (1,0,\dots,0)$.
  Consequently, elementary properties of the Fourier transform show that
  \(
    \widehat{f_n}
    = \frac{1}{2}
      M_{- x_0}
      \big[
        T_{n \pi e_1 / \eps} M_{x_0} \widehat{\gamma}
        + T_{-n \pi e_1 / \eps} M_{x_0} \widehat{\gamma}
      \big] ,
  \)
  and hence
  \[
    |\widehat{f_n}(\xi)|
    \leq \frac{1}{2}
         \big(
           |\widehat{\gamma}(\xi - \tfrac{n \pi}{\eps} e_1)|
           + |\widehat{\gamma}(\xi + \tfrac{n \pi}{\eps} e_1)|
         \big)
    \qquad \text{for all } \xi \in \R^d .
  \]
  Since
  \(
    1 + |\xi|
    \leq 1 + |\xi \pm \frac{n \pi}{\eps} e_1| + \frac{n \pi}{\eps}
    \leq (1 + |\xi \pm \frac{n \pi}{\eps} e_1|) (1 + \frac{n \pi}{\eps})
    \leq \frac{2 n \pi}{\eps} (1 + |\xi \pm \frac{n \pi}{\eps} e_1|) ,
  \)
  this shows that
  \begin{equation}
  \begin{split}
    \| f_n \|_X
    & = \int_{\R^d}
          (1 + |\xi|)^{\beta} \cdot |\widehat{f_n}(\xi)|
        \, d \xi \\
    & \leq \frac{1}{2}
           \Big(
             \frac{2 \pi n}{\eps}
           \Big)^{\beta}
           \int_{\R^d}
             \sum_{\theta \in \{\pm 1\}}
               \bigl(1 + |\xi + \theta \tfrac{n \pi}{\eps} e_1|\bigr)^{\beta}
               \bigl|\widehat{\gamma} (\xi + \theta \tfrac{n \pi}{\eps} e_1)\bigr|
           \, d \xi \\
    & \leq \Big(
             \frac{2 \pi n}{\eps}
           \Big)^{\beta}
           \int_{\R^d}
             (1 + |\eta|)^{\beta} \, |\widehat{\gamma}(\eta)|
           \, d \eta
      \lesssim n^{\beta} .
  \end{split}
  \label{eq:FourierBarroNonInclusionNormEstimate}
  \end{equation}

  On the other hand, for $t \in (-2\eps, 2\eps)$ we see because of $\gamma \equiv 1$
  on $U_0$ that
  \[
    (\partial_1^2 f_n)(x_0 + t \, e_1)
    = \frac{d^2}{d t^2}
        f_n (x_0 + t \, e_1)
    = \frac{d^2}{d t^2}
        \cos \Big( \frac{n \pi}{\eps} t \Big)
    = - \Big( \frac{n \pi}{\eps} \Big)^2
        \cos \Big( \frac{n \pi}{\eps} t \Big)
  \]
  and hence
  \begin{align*}
    \int_{-2\eps}^{2 \eps}
      |(\partial_1^2 f_n)(x_0 + t \, e_1)|
    \, d t
    & = (n \pi / \eps)^2
        \int_{-2\eps}^{2 \eps}
          \bigl|\cos\bigl(n \pi t / \eps\bigr)\bigr|
        \, d t
      = \frac{n \pi}{\eps}
        \int_{2 \pi n}^{2 \pi n}
          |\cos(s)|
        \, d s \\
    & \overset{(\ast)}{=}
      \frac{4 \pi \, n^2}{\eps}
      \int_0^\pi
        |\cos(s)|
      \, d s
      = \frac{8 \pi n^2}{\eps} .
  \end{align*}
  Here, we used at $(\ast)$ that $s \mapsto |\cos(s)|$ is $\pi$-periodic and even.
  Combining the last calculation with
  \Cref{eq:FourierBarronReLUBarronAlmostContradiction,eq:FourierBarroNonInclusionNormEstimate},
  we arrive at
  \(
    n^2 \lesssim
    \int_{-2\eps}^{2 \eps}
      |(\partial_1^2 f_n)(x_0 + t \, e_1)|
    \, d t
    \lesssim \| f_n \|_X
    \lesssim n^{\beta} ,
  \)
  for all $n \in \N$.
  This is only possible if $\beta \geq 2$,
  and since $\beta = \max\{\alpha, 1\}$ this requires $\alpha \geq 2$.
\end{proof}

%% file: A-EmpiricalProcessBound.tex

In this section, we prove a ``uniform law of large numbers,'' similar to the pseudo-dimension based
generalization bound in \cite[Theorem~11.8]{MohriFoundationsOfMachineLearning},
which is used in the third part of the proof of Proposition~\ref{prop:BarronFunctionApproximation}.
The result given here is probably well-known; but since we could not locate a reference,
we provide the proof.
The main difference to the bound in \cite{MohriFoundationsOfMachineLearning} is that we
estimate the expected sampling error, instead of giving a high probability bound;
this allows us to omit a log factor.
Furthermore, we use a complexity measure of the hypothesis class that
differs slightly from the usual pseudo-dimension.

\begin{proposition}\label{prop:PseudoDimensionGeneralizationBound}
  There is a universal constant $\kappa > 0$ with the following property:
  If $(\Omega,\calF,\mu)$ is a probability space, if $a,b \in \R$ with $a < b$, and if
  $\emptyset \neq \calG \subset \{ g : \Omega \to [a,b] \colon g \text{ measurable} \}$
  satisfies
  \[
    d := \sup_{\lambda \in \R}
           \VC (\{ I_{g,\lambda} \colon g \in \calG \})
    < \infty,
    \qquad \text{where} \qquad
    I_{g,\lambda} : \quad
    \Omega \to \{ 0,1 \}, \quad
    \omega \mapsto \Indicator_{g(\omega) > \lambda} ,
  \]
  then for any $n \in \N$ and $S = (X_1,\dots,X_n) \overset{\text{i.i.d.}}{\sim} \mu$,
  we have
  \[
    \EE_{S}
    \bigg[
      \sup_{g \in \calG}
        \Big|
          \EE_{X \sim \mu} [ \, g(X) \,]
          - \frac{1}{n} \sum_{i=1}^n g(X_i)
        \Big|
    \bigg]
    \leq \kappa \cdot (b-a) \cdot \sqrt{\frac{d}{n}} .
  \]
\end{proposition}

\begin{rem*}
  Here, as in most sources studying empirical processes
  (see e.g.~\mbox{\cite[Section~7.2]{VershyninHighDimensionalProbability}}),
  we interpret $\EE [\sup_{i \in I} X_i ]$ as
  $\sup_{I_0 \subset I \text{ finite}} \EE [\sup_{i \in I_0} X_i]$,
  in order to avoid measurability issues.
\end{rem*}

\begin{proof}
  Given a sample $S = (X_1,\dots,X_n) \in \Omega^n$,
  we write $\mu_S := \frac{1}{n} \sum_{i=1}^n \delta_{X_i}$ for the associated empirical measure.
  We want to bound
  \[
    \EE
    \Bigl[\,
      \sup_{g \in \calG}
        \big|
          \EE_{X \sim \mu} [g(X)] - \EE_{X \sim \mu_S} [g(X)]
        \big|
    \,\Bigr]
    ,
  \]
  where the outer expectation is with respect to
  $S = (X_1,\dots,X_n) \overset{\text{i.i.d.}}{\sim} \mu$.
  First, by replacing $\calG$ with $\calG^\ast := \{ g - a \colon g \in \calG \}$,
  it is easy to see that we can assume $a = 0$ without loss of generality.
  Define $M := b = b - a$.
  Then, for any $g \in \calG$ and any probability measure $\nu$ on $\Omega$,
  the \emph{layer cake formula} (see e.g.~\cite[Proposition~6.24]{FollandRA}) shows
  \[
    \EE_{X \sim \nu} [g(X)]
    = \int_0^M
        \nu ( \{ \omega \in \Omega : g(\omega) > \lambda \})
      \, d \lambda
    = \int_0^M
        \EE_{X \sim \nu} [I_{g,\lambda} (X)]
      \, d \lambda .
  \]
  Therefore,
  \begin{align*}
    \Big| \EE_{X \sim \mu} [g(X)] - \EE_{X \sim \mu_S} [g(X)] \Big|
    & = \bigg|
          \int_0^M
            \EE_{X \sim \mu} \bigl[I_{g,\lambda} (X)\bigr]
            - \EE_{X \sim \mu_S} \bigl[I_{g,\lambda} (X)\bigr]
          \, d \lambda
        \bigg| \\
    & \leq \int_0^M
               \Big|
                 \EE_{X \sim \mu} \bigl[I_{g,\lambda} (X)\bigr]
                 - \EE_{X \sim \mu_S} \bigl[I_{g,\lambda} (X)\bigr]
               \Big|
             \, d \lambda .
  \end{align*}
  In combination with the elementary estimate
  \(
    \sup_{g \in \calG}
      \int_0^M
        \Gamma_g(\lambda)
      \, d \lambda
    \leq \int_0^M
           \sup_{g \in \calG}
             \Gamma_g(\lambda)
         \, d \lambda
  \)
  and Tonelli's theorem, this implies
  \begin{align*}
    \EE
    \Big[
      \sup_{g \in \calG}
        \big|
          \EE_{X \sim \mu} [g (X)] - \EE_{X \sim \mu_S} [g(X)]
        \big|
    \Big]
    & \leq \EE
           \Big[
             \sup_{g \in \calG}
               \int_0^M
                 \Big|
                   \EE_{X \sim \mu} \bigl[I_{g,\lambda} (X)\bigr]
                   - \EE_{X \sim \mu_S} \bigl[I_{g,\lambda} (X)\bigr]
                 \Big|
               \, d \lambda
           \Big] \\
    & \leq \int_0^M
           \EE
           \Big[
             \sup_{g \in \calG}
               \Big|
                 \EE_{X \sim \mu} \bigl[I_{g,\lambda} (X)\bigr]
                 - \EE_{X \sim \mu_S} \bigl[I_{g,\lambda} (X)\bigr]
               \Big|
           \Big]
           \, d \lambda \\
    & \overset{(\ast)}{\leq}
           \int_0^M
             \kappa \cdot \sqrt{\frac{\VC(\{ I_{g,\lambda} \colon g \in \calG \})}{n}}
           \, d \lambda
      \leq \kappa \cdot M \cdot \sqrt{\frac{d}{n}} .
  \end{align*}
  Here, the step marked with $(\ast)$ is an immediate consequence of the bound for the suprema
  of empirical processes based on the VC dimension given in
  \cite[Theorem~8.3.23]{VershyninHighDimensionalProbability}.
\end{proof}

%% file: B-TotalVariation.tex

\begin{proof}[Proof of Lemma~\ref{lem:TechnicalBVBound}]
  \textbf{Step 1:} We first show that if $h : \R \to \R$ is non-decreasing and bounded, then
  \(
    |\int_{\R} \varphi'(t) \, h(t) \, d t|
    \leq \| \varphi \|_{\sup} \cdot \lim_{x \to \infty} [h(x) - h(-x)]
  \)
  for every $\varphi \in C_c^\infty(\R)$.
  To see this, define $c := \lim_{x \to -\infty} h(x)$ and
  $\widetilde{h} : \R \to \R, x \mapsto \lim_{y \downarrow x} h(y) - c$.
  It is straightforward to see that $\widetilde{h}$ is non-decreasing, bounded, and right-continuous
  with $\lim_{x \to -\infty} \widetilde{h}(x) = 0$, so that $\widetilde{h} \in \mathrm{NBV}$
  in the notation of \cite[Section~3.5]{FollandRA}.
  Furthermore, since a monotonic function can have at most countably many
  discontinuities (see \cite[Theorem~3.23]{FollandRA}), we have $\widetilde{h} = h - c$
  on the complement of a countable set, and hence almost everywhere.
  Since we also have $\int_{\R} \varphi'(t) \, d t = 0$ thanks to the compact support of $\varphi$,
  if we denote by $\mu$ the unique Borel measure on $\R$ satisfying
  $\widetilde{h}(x) = \mu \big( (-\infty, x] \big)$ for all $x \in \R$,
  then the partial integration formula in \cite[Theorem~3.36]{FollandRA} shows
  as claimed that
  \begin{align*}
    \Big| \int_{\R} \varphi'(t) \, h(t) \, d t \Big|
    \!=\! \Big| \int_{\R} \varphi'(t) \, [h(t) - c] \, d t \Big|
    & = \Big| \int_{\R} \varphi'(t) \, \widetilde{h}(t) \, d t \Big|
      \!=\! \Big| \int_{\R} \widetilde{h}(t) \, d \varphi(t) \Big|
      \!=\! \Big| \int_{\R} \varphi(t) \, d \mu(t) \Big| \\
    & \leq \| \varphi \|_{\sup} \cdot \mu(\R)
      = \| \varphi \|_{\sup}
        \cdot \lim_{x \to \infty}
              \big[\, \widetilde{h} (x) - \widetilde{h} (-x) \,\big] \\
    & = \| \varphi \|_{\sup} \cdot \lim_{x \to \infty} \bigl[ h (x) - h (-x) \bigr] .
  \end{align*}

  \noindent
  \textbf{Step 2:} Define
  \[
    T_g : \,\,
    \R \to \R, \,\,
    x \mapsto \sup
              \Big\{
                \sum_{j=1}^n |g(x_j) - g(x_{j-1})|
                \quad \colon \quad
                n \in \N \text{ and } -\infty < x_0 < \dots < x_n = x
              \Big\} .
  \]
  Then $T_g$ is non-decreasing and satisfies $\lim_{x\to-\infty} T_g (x) = 0$
  and $\lim_{x\to\infty} T_g (x) = \mathrm{TV}(g)$; furthermore, $g_1 := \frac{1}{2} (T_g + g)$
  and $g_2 := \frac{1}{2} (T_g - g)$ are both non-decreasing and bounded with $g = g_1 - g_2$;
  all of these properties can be found in \cite[Section~3.5]{FollandRA}.
  Note that
  \({
    \lim_{x \to \infty} [\, g_i (x) - g_i(-x) \,]
    = \frac{1}{2} \mathrm{TV}(g)
      + \frac{(-1)^{i-1}}{2}
        \lim_{x \to \infty}
          [g(x) - g(-x)] .
  }\)
  In combination with the estimate from Step~1 (applied to $h = g_i$),
  this implies as claimed that
  \[
    \Big| \int_{\R} \! \varphi'(t) \, g(t) \, d t \Big|
    \!\leq\! \sum_{i=1}^2
               \Big| \int_{\R} \! \varphi'(t) \, g_i(t) \, d t \Big|
    \!\leq\! \| \varphi \|_{\sup}
             \sum_{i=1}^2
               \lim_{x \to \infty}
                 \bigl[\, g_i (x) - g_i(-x) \,\bigr]
    \!=\!    \| \varphi \|_{\sup} \, \mathrm{TV} (g) .
    \qedhere
  \]
  %
  %
\end{proof}